\documentclass[english]{article}
\usepackage[T1]{fontenc}
\usepackage{geometry}
\usepackage{color}
\usepackage{babel}
\usepackage{array}
\usepackage{verbatim}
\usepackage{float}
\usepackage{mathtools}
\usepackage{bm}
\usepackage{multirow}
\usepackage{amsmath}
\usepackage{amsthm}
\usepackage{amssymb}
\usepackage{graphicx}
\usepackage{booktabs}
\usepackage{todonotes}
\usepackage{algorithm}
\usepackage{algorithmic}
\usepackage[hyphens]{url}
\usepackage{authblk}
\usepackage{enumitem}
\usepackage{subcaption}
\usepackage[round]{natbib}
\usepackage{hyperref}
\usepackage{cleveref}
\renewcommand{\cite}[1]{\citep{#1}}

\theoremstyle{plain}
\newtheorem{theorem}{Theorem}

\newtheorem{lemma}{Lemma}
\newtheorem{corollary}{Corollary}
\newtheorem{definition}{Definition}
\newtheorem{assumption}{Assumption}

\newtheorem{appxlem}{Lemma}[section]
\newtheorem{appxdef}{Definition}[section]

\global\long\def\loll{$(L_0, L_1)$}

\begin{document}
\title{\bf{\LARGE Trust Region Methods For Nonconvex Stochastic Optimization Beyond Lipschitz Smoothness}\footnote{Chenghan Xie and Chenxi Li contribute equally}}

\author{Chenghan Xie\thanks{Fudan University. Email: 20307130043@fudan.edu.cn} \hspace{1em} Chenxi Li\thanks{Shanghai University of Finance and Economics. Email: chenxili@stu.sufe.edu.cn} \hspace{1em} Chuwen Zhang \thanks{Shanghai University of Finance and Economics. Email: chuwzhang@gmail.com} \hspace{1em} Qi Deng \thanks{Shanghai University of Finance and Economics. Email: qideng@sufe.edu.cn} \hspace{1em} Dongdong Ge \thanks{Shanghai University of Finance and Economics. Email: ge.dongdong@mail.shufe.edu.cn
} \hspace{1em} Yinyu Ye
 \thanks{Stanford University. Email: yyye@stanford.edu} }

\global\long\def\inprod#1#2{\left\langle #1,#2\right\rangle }%
\global\long\def\inner#1#2{\langle#1,#2\rangle}%
\global\long\def\binner#1#2{\big\langle#1,#2\big\rangle}%
\global\long\def\norm#1{\Vert#1\Vert}%
\global\long\def\bnorm#1{\big\Vert#1\big\Vert}%
\global\long\def\Bnorm#1{\Big\Vert#1\Big\Vert}%
\global\long\def\red#1{\textcolor{red}{#1}}%
\global\long\def\blue#1{\textcolor{blue}{#1}}%
\global\long\def\pink#1{\textcolor{pink}{#1}}%
\global\long\def\yellow#1{\textcolor{yellow}{#1}}%
\global\long\def\green#1{\textcolor{green}{#1}}%
\global\long\def\orange#1{\textcolor{orange}{#1}}%
\global\long\def\purple#1{\textcolor{purple}{#1}}%
\global\long\def\brown#1{\textcolor{brown}{#1}}%
\global\long\def\indigo#1{\textcolor{indigo}{#1}}%
\global\long\def\teal#1{\textcolor{teal}{#1}}%

\global\long\def\brbra#1{\big(#1\big)}%
\global\long\def\Brbra#1{\Big(#1\Big)}%
\global\long\def\rbra#1{(#1)}%

\global\long\def\sbra#1{[#1]}%
\global\long\def\bsbra#1{\big[#1\big]}%
\global\long\def\Bsbra#1{\Big[#1\Big]}%
\global\long\def\abs#1{\vert#1\vert}%
\global\long\def\babs#1{\big\vert#1\big\vert}%

\global\long\def\cbra#1{\{#1\}}%
\global\long\def\bcbra#1{\big\{#1\big\}}%
\global\long\def\Bcbra#1{\Big\{#1\Big\}}%
\global\long\def\vertiii#1{\left\vert \kern-0.25ex  \left\vert \kern-0.25ex  \left\vert #1\right\vert \kern-0.25ex  \right\vert \kern-0.25ex  \right\vert }%
\global\long\def\matr#1{\bm{#1}}%
\global\long\def\til#1{\tilde{#1}}%
\global\long\def\wtil#1{\widetilde{#1}}%
\global\long\def\wh#1{\widehat{#1}}%
\global\long\def\mcal#1{\mathcal{#1}}%
\global\long\def\mbb#1{\mathbb{#1}}%
\global\long\def\mtt#1{\mathtt{#1}}%
\global\long\def\ttt#1{\texttt{#1}}%
\global\long\def\dtxt{\textrm{d}}%
\global\long\def\bignorm#1{\bigl\Vert#1\bigr\Vert}%
\global\long\def\Bignorm#1{\Bigl\Vert#1\Bigr\Vert}%
\global\long\def\rmn#1#2{\mathbb{R}^{#1\times#2}}%
\global\long\def\deri#1#2{\frac{d#1}{d#2}}%
\global\long\def\pderi#1#2{\frac{\partial#1}{\partial#2}}%
\global\long\def\limk{\lim_{k\rightarrow\infty}}%
\global\long\def\trans{\textrm{T}}%
\global\long\def\onebf{\mathbf{1}}%
\global\long\def\Bbb{\mathbb{B}}%
\global\long\def\hbf{\mathbf{h}}%
\global\long\def\zerobf{\mathbf{0}}%
\global\long\def\zero{\bm{0}}%

\global\long\def\Euc{\mathrm{E}}%
\global\long\def\Expect{\mathbb{E}}%
\global\long\def\rank{\mathrm{rank}}%
\global\long\def\range{\mathrm{range}}%
\global\long\def\diam{\mathrm{diam}}%
\global\long\def\epi{\mathrm{epi} }%
\global\long\def\inte{\operatornamewithlimits{int}}%
\global\long\def\dist{\operatornamewithlimits{dist}}%
\global\long\def\proj{\operatorname{Proj}}%
\global\long\def\cov{\mathrm{Cov}}%
\global\long\def\argmin{\operatornamewithlimits{argmin}}%
\global\long\def\argmax{\operatornamewithlimits{argmax}}%
\global\long\def\tr{\operatornamewithlimits{tr}}%
\global\long\def\dis{\operatornamewithlimits{dist}}%
\global\long\def\sign{\operatornamewithlimits{sign}}%
\global\long\def\prob{\mathrm{Prob}}%
\global\long\def\st{\operatornamewithlimits{s.t.}}%
\global\long\def\dom{\mathrm{dom}}%
\global\long\def\prox{\mathrm{prox}}%
\global\long\def\diag{\mathrm{diag}}%
\global\long\def\and{\mathrm{and}}%
\global\long\def\aleq{\overset{(a)}{\leq}}%
\global\long\def\aeq{\overset{(a)}{=}}%
\global\long\def\ageq{\overset{(a)}{\geq}}%
\global\long\def\bleq{\overset{(b)}{\leq}}%
\global\long\def\beq{\overset{(b)}{=}}%
\global\long\def\bgeq{\overset{(b)}{\geq}}%
\global\long\def\cleq{\overset{(c)}{\leq}}%
\global\long\def\ceq{\overset{(c)}{=}}%
\global\long\def\cgeq{\overset{(c)}{\geq}}%
\global\long\def\dleq{\overset{(d)}{\leq}}%
\global\long\def\deq{\overset{(d)}{=}}%
\global\long\def\dgeq{\overset{(d)}{\geq}}%
\global\long\def\eleq{\overset{(e)}{\leq}}%
\global\long\def\eeq{\overset{(e)}{=}}%
\global\long\def\egeq{\overset{(e)}{\geq}}%
\global\long\def\fleq{\overset{(f)}{\leq}}%
\global\long\def\feq{\overset{(f)}{=}}%
\global\long\def\fgeq{\overset{(f)}{\geq}}%
\global\long\def\gleq{\overset{(g)}{\leq}}%
\global\long\def\as{\textup{a.s.}}%
\global\long\def\ae{\textup{a.e.}}%
\global\long\def\Var{\operatornamewithlimits{Var}}%
\global\long\def\clip{\operatorname{clip}}%
\global\long\def\conv{\operatorname{conv}}%
\global\long\def\Cov{\operatornamewithlimits{Cov}}%
\global\long\def\raw{\rightarrow}%
\global\long\def\law{\leftarrow}%
\global\long\def\Raw{\Rightarrow}%
\global\long\def\Law{\Leftarrow}%
\global\long\def\vep{\varepsilon}%
\global\long\def\dom{\operatornamewithlimits{dom}}%
\global\long\def\tsum{{\textstyle {\sum}}}%
\global\long\def\Cbb{\mathbb{C}}%
\global\long\def\Ebb{\mathbb{E}}%
\global\long\def\Fbb{\mathbb{F}}%
\global\long\def\Nbb{\mathbb{N}}%
\global\long\def\Rbb{\mathbb{R}}%
\global\long\def\extR{\widebar{\mathbb{R}}}%
\global\long\def\Pbb{\mathbb{P}}%
\global\long\def\Mrm{\mathrm{M}}%
\global\long\def\Acal{\mathcal{A}}%
\global\long\def\Bcal{\mathcal{B}}%
\global\long\def\Ccal{\mathcal{C}}%
\global\long\def\Dcal{\mathcal{D}}%
\global\long\def\Ecal{\mathcal{E}}%
\global\long\def\Fcal{\mathcal{F}}%
\global\long\def\Gcal{\mathcal{G}}%
\global\long\def\Hcal{\mathcal{H}}%
\global\long\def\Ical{\mathcal{I}}%
\global\long\def\Kcal{\mathcal{K}}%
\global\long\def\Lcal{\mathcal{L}}%
\global\long\def\Mcal{\mathcal{M}}%
\global\long\def\Ncal{\mathcal{N}}%
\global\long\def\Ocal{\mathcal{O}}%
\global\long\def\Pcal{\mathcal{P}}%
\global\long\def\Scal{\mathcal{S}}%
\global\long\def\Tcal{\mathcal{T}}%
\global\long\def\Xcal{\mathcal{X}}%
\global\long\def\Ycal{\mathcal{Y}}%
\global\long\def\Zcal{\mathcal{Z}}%
\global\long\def\i{i}%

\global\long\def\abf{\mathbf{a}}%
\global\long\def\bbf{\mathbf{b}}%
\global\long\def\cbf{\mathbf{c}}%
\global\long\def\fbf{\mathbf{f}}%
\global\long\def\qbf{\mathbf{q}}%
\global\long\def\gbf{\mathbf{g}}%
\global\long\def\ebf{\mathbf{e}}%
\global\long\def\lambf{\bm{\lambda}}%
\global\long\def\alphabf{\bm{\alpha}}%
\global\long\def\sigmabf{\bm{\sigma}}%
\global\long\def\thetabf{\bm{\theta}}%
\global\long\def\deltabf{\bm{\delta}}%
\global\long\def\lbf{\mathbf{l}}%
\global\long\def\ubf{\mathbf{u}}%
\global\long\def\pbf{\mathbf{\mathbf{p}}}%
\global\long\def\vbf{\mathbf{v}}%
\global\long\def\wbf{\mathbf{w}}%
\global\long\def\xbf{\mathbf{x}}%
\global\long\def\ybf{\mathbf{y}}%
\global\long\def\zbf{\mathbf{z}}%
\global\long\def\dbf{\mathbf{d}}%
\global\long\def\Wbf{\mathbf{W}}%
\global\long\def\Abf{\mathbf{A}}%
\global\long\def\Ubf{\mathbf{U}}%
\global\long\def\Pbf{\mathbf{P}}%
\global\long\def\Ibf{\mathbf{I}}%
\global\long\def\Ebf{\mathbf{E}}%
\global\long\def\sbf{\mathbf{s}}%
\global\long\def\Mbf{\mathbf{M}}%
\global\long\def\Dbf{\mathbf{D}}%
\global\long\def\Qbf{\mathbf{Q}}%
\global\long\def\Lbf{\mathbf{L}}%
\global\long\def\Pbf{\mathbf{P}}%
\global\long\def\Xbf{\mathbf{X}}%

\global\long\def\abm{\bm{a}}%
\global\long\def\bbm{\bm{b}}%
\global\long\def\cbm{\bm{c}}%
\global\long\def\dbm{\bm{d}}%
\global\long\def\ebm{\bm{e}}%
\global\long\def\fbm{\bm{f}}%
\global\long\def\gbm{\bm{g}}%
\global\long\def\hbm{\bm{h}}%
\global\long\def\pbm{\bm{p}}%
\global\long\def\qbm{\bm{q}}%
\global\long\def\rbm{\bm{r}}%
\global\long\def\sbm{\bm{s}}%
\global\long\def\tbm{\bm{t}}%
\global\long\def\ubm{\bm{u}}%
\global\long\def\vbm{\bm{v}}%
\global\long\def\wbm{\bm{w}}%
\global\long\def\xbm{\bm{x}}%
\global\long\def\ybm{\bm{y}}%
\global\long\def\zbm{\bm{z}}%
\global\long\def\Abm{\bm{A}}%
\global\long\def\Bbm{\bm{B}}%
\global\long\def\Cbm{\bm{C}}%
\global\long\def\Dbm{\bm{D}}%
\global\long\def\Ebm{\bm{E}}%
\global\long\def\Fbm{\bm{F}}%
\global\long\def\Gbm{\bm{G}}%
\global\long\def\Hbm{\bm{H}}%
\global\long\def\Ibm{\bm{I}}%
\global\long\def\Jbm{\bm{J}}%
\global\long\def\Lbm{\bm{L}}%
\global\long\def\Obm{\bm{O}}%
\global\long\def\Pbm{\bm{P}}%
\global\long\def\Qbm{\bm{Q}}%
\global\long\def\Rbm{\bm{R}}%
\global\long\def\Ubm{\bm{U}}%
\global\long\def\Vbm{\bm{V}}%
\global\long\def\Wbm{\bm{W}}%
\global\long\def\Xbm{\bm{X}}%
\global\long\def\Ybm{\bm{Y}}%
\global\long\def\Zbm{\bm{Z}}%
\global\long\def\lambm{\bm{\lambda}}%
\global\long\def\alphabm{\bm{\alpha}}%
\global\long\def\albm{\bm{\alpha}}%
\global\long\def\taubm{\bm{\tau}}%
\global\long\def\mubm{\bm{\mu}}%
\global\long\def\yrm{\mathrm{y}}%
\global\long\def\iid{i.i.d}%
\global\long\def\vec{\operatorname{vec}}%

\global\long\def\rone{\textrm{I}}%
\global\long\def\rtwo{\textrm{II}}%
\global\long\def\rthree{\textrm{III}}%
\global\long\def\rfour{\textrm{IV}}%
\global\long\def\rfive{\textrm{V}}%
\global\long\def\rsix{\textrm{VI}}%
\global\long\def\rseven{\textrm{VII}}%
\global\long\def\reight{\textrm{VIII}}%
\global\long\def\rnine{\textrm{IX}}%
\global\long\def\rten{\textrm{X}}%

\maketitle

\begin{abstract}
    In many important machine learning applications, the standard assumption of having a globally Lipschitz continuous gradient may fail to hold. This paper delves into a more general
    $(L_0, L_1)$-smoothness setting, which gains particular significance within the realms of deep neural networks and distributionally robust optimization (DRO).  We demonstrate the significant advantage of trust region methods for stochastic nonconvex optimization under such generalized smoothness assumption.
    We show that first-order trust region methods can recover the normalized and clipped stochastic gradient as special cases and then provide a unified analysis to show their convergence to first-order stationary conditions.
    Motivated by the important application of DRO, we propose a generalized high-order smoothness condition, under which second-order trust region methods can achieve a complexity of $\mathcal{O}(\epsilon^{-3.5})$ for convergence to  second-order stationary points. By incorporating variance reduction,  the second-order trust region method obtains an even better complexity of $\mathcal{O}(\epsilon^{-3})$, matching the optimal bound for standard smooth optimization. To our best knowledge, this is the first work to show convergence beyond the first-order stationary condition for generalized smooth optimization.
    Preliminary experiments show that our proposed algorithms perform favorably compared with existing methods.
\end{abstract}

\section{Introduction}
We study the problem of minimizing a nonconvex function $F: \mathbb{R}^n \rightarrow \mathbb{R}$ which is expressed as the expectation of a stochastic function, i.e.,
\begin{equation}\label{eq:stochastic objective}
    \underset{x \in \mathbb{R}^n}{\min} \quad F(x)=\mathbb{E}_{\xi}[f(x ; \xi)],
\end{equation}
where the random variable $\xi$ is realized according to a distribution $\mathcal{P}$. Over the years, substantial progress has been made in studying functions that possess Lipschitzian gradients, commonly referred to as $L$-smoothness functions. Notable contributions in this area can be found in \cite{ghadimi2013stochastic,johnson2013accelerating,fang2018spider,carmonLowerBoundsFinding2019}, among others.

However, the assumption of Lipschitz smoothness may not hold in many important applications. For instance, in language models such as LSTM \cite{zhang2019gradient} and transformers \cite{crawshaw2022robustness}, the function smoothness parameter can exhibit a strong correlation with the gradient norm along the training trajectory.
Beyond these challenges in the standard Empirical Risk Minimization (ERM) framework, the  $L$-smoothness condition could also easily fail in distributionally robust optimization (DRO) \cite{delageDistributionallyRobustOptimization2010,duchi_learning_2021,NEURIPS2020_64986d86}.
DRO is particularly significant as it serves as a foundational element for ethical algorithms~\cite{kearnsEthicalAlgorithmScience2020} arising from accountability and fairness issues in machine learning \cite{fusterPredictablyUnequalEffects2022,tangWhatIsHowToFairness2023,berkFairnessCriminalJustice2021}.

Motivated by this challenge, \citet{zhang2019gradient} introduced first-order generalized smoothness, also known as \loll-smoothness,  where  the Hessian norm is unbounded but allowed to grow linearly with the gradient norm. This condition can be further relaxed without the need of twice differentiability. Specifically, the  \loll-smoothness condition \cite{zhang2020improved, reisizadeh2023variance} is defined as
\begin{equation}\label{eq.intro.l0l1}
    \|\nabla F(x)-\nabla F(x')\| \leq\left(L_0+L_1\|\nabla F(x)\|\right)\|x-x'\|
\end{equation}
holds for any $x,x'\in\Rbb^n$ such that $\|x-x'\|\le 1/L_1$, for  constants $L_0>0,L_1\ge 0$. 
\citet{jin2021non} showed that  \loll-smoothness~\eqref{eq.intro.l0l1} holds for a broad class of DRO objectives when expressed in the dual form.
Due to the difficulty in handling unbounded Lipschitz parameters, 
significant effort has been devoted to developing efficient algorithms under \loll-smoothness \cite{crawshaw2022robustness,reisizadeh2023variance,wang2022provable}. Typically, these works focus on developing more stable stepsizes for stochastic gradient descent through techniques like gradient clipping and step size normalization.

Despite these recent progresses, existing research remains limited to identifying approximate first-order stationary points (FOSP), which may be suboptimal in nonconvex settings. 
This drawback prompts the central question addressed in this paper:
\textit{Is it possible to develop an effective method capable of achieving approximate second-order stationary points under the conditions of generalized smoothness?}

In this paper, we firmly answer this question by proposing an algorithmic framework based on classical trust region methods \cite{sorensen1982newton,conn2000trust}.
The crux of our method is to impose a trust region radius, which also coincides with the mutual concept of the aforementioned gradient-based methods.
On one hand, this positions our method as a unifying analysis for gradient clipping and normalized gradient \cite{zhang2019gradient,jin2021non} in which combinations of them can be derived. On the other hand,
this framework naturally extends to finding second-order solutions if granted second-order derivatives. To our special interest,
a second-order theory of generalized smoothness is proposed for DRO, which further empowers the complexity analysis of our framework.
Our developments consist of four major steps:
\begin{itemize}[leftmargin=*]
    \item Firstly, we propose a unified trust region framework, under which the first-order variant, FOTRGS, unifies NSGD and clipped gradient methods with a weaker requirement of variance condition. 
    \item Secondly, we propose a second-order theory of generalized smoothness and variance condition. We show that many divergence-based DRO problems with $\psi$-divergence satisfy our proposed assumptions.
    \item Thirdly, under the unified framework, we propose SOTRGS, namely, second-order trust region methods for generalized smoothness, and prove that it can achieve a second-order stationary point with $\mathcal{O}(\epsilon^{-3.5})$ sample complexity, which is better than first-order methods without variance-reduction techniques.
    \item Finally, we employ variance reduction techniques and propose SOTRGS-VR, demonstrating that identifying a second-order stationary point can be achieved in an optimal complexity of $\mathcal{O}(\epsilon^{-3})$.
    
\end{itemize}
A brief comparison of our methods and existing proposals are presented in Table \ref{tab.overview}. To our best knowledge, both the second-order generalized smoothness and convergence to SOSP are novel.
In addition to the theoretical contribution, we conduct extensive experiments on DRO problems with imbalanced datasets, which justify the empirical advantage of our proposed methods.

\section{Related Works}

\paragraph{$\bm{(L_0, L_1)}$-smoothness}

The concept of $(L_0, L_1)$-smoothness was first introduced by \citet{zhang2019gradient} to understand the superior performance of clipped algorithms over traditional non-adaptive gradient methods in natural language processing. Under the $(L_0, L_1)$-smoothness setting, \citet{zhang2019gradient} shows that normalized and clipped gradient methods converge to an $\epsilon$-stationary point of the nonconvex objective function with at most \( \mathcal{O}(\epsilon^{-4}) \) gradient samples. This initiative sparked a series of follow-up studies, including \citet{zhang2020improved,qian2021understanding,zhao2021convergence}. \citet{zhang2020improved} proposes a general framework which combines momentum acceleration with the clipped method. More recently, \citet{reisizadeh2023variance}applies the variance reduced techniques to the clipped gradient method and improves the gradient complexity to \( \mathcal{O}(\epsilon^{-3}) \).

A parallel line of research has focused on analyzing algorithms that go beyond the normalized and clipping gradient methods in the $(L_0, L_1)$-smoothness setting. These include studies by \citet{wang2022provable, li2023convergence} on Adam, \citet{crawshaw2022robustness} on unclipped gradient methods, and more recently \citet{sun2023convergence} for $(L_0, L_1)$-smoothness in the variational inference problems. Another vein of research has sought to relaxed a heavy reliance on bounded variance assumptions; see \citet{faw2023beyond, wang2023convergence} and the references therein.

We are also aware of the works on even more general smoothness conditions based on \loll-smoothness. \citet{chen2023generalized} proposes a new notion of $\alpha$-symmetric generalized smoothness, which is roughly as general as $\left(L_0, L_1\right)$-smoothness. \citet{crawshaw2022robustness} and \citet{pan2023toward} provide a coordinate-wise type of \loll-smoothness. \citet{li2023convex} showed that classic first-order methods such as stochastic gradient and accelerated methods still have convergence guarantee under a mild $\ell$-smoothness condition, which allows the Hessian norm to be bounded by a more general non-decreasing function $\ell(\|\nabla F(x)\|)$. Despite these advances, no previous work has contributed to the second-order generalization of $(L_0, L_1)$-smoothness for second-order stationary points.

\begin{table*}[htbp]
    \centering
    \small
    \begin{tabular}{cccc}
        \hline
        \textbf{Algorithm}                               & \textbf{Smoothness}   & \textbf{Complexity}            & \textbf{Property} \\ \hline
        SGD \cite{ghadimi2013stochastic}                 & Lipschitz             & $\mathcal{O}(\epsilon^{-4})$\  & FOSP              \\
        SPIDER \cite{fang2018spider}                     & Lipschitz             & $\mathcal{O}(\epsilon^{-3})$   & FOSP              \\
        STR \cite{shen2019stochastic}                                        & Lipschitz             & $\mathcal{O}(\epsilon^{-3.5})$ & SOSP              \\ 
        SCR \cite{tripuraneni2018stochastic} & Lipschitz             & $\mathcal{O}(\epsilon^{-3.5})$ & SOSP \\  \hline
        ClippedSGD \cite{zhang2019gradient}              & FO-Generalized Smooth & $\mathcal{O}(\epsilon^{-4})$   & FOSP              \\
        Clipped+ \cite{zhang2020improved}                & FO-Generalized Smooth & $\mathcal{O}(\epsilon^{-4})$   & FOSP              \\
        NSGD \cite{jin2021non}                           & FO-Generalized Smooth & $\mathcal{O}(\epsilon^{-4})$   & FOSP              \\
        $(L_0,L_1)$-SPIDER \cite{reisizadeh2023variance} & FO-Generalized Smooth & $\mathcal{O}(\epsilon^{-3})$   & FOSP              \\\hline
         \textbf{FOTRGS}                                           & FO-Generalized Smooth & $\mathcal{O}(\epsilon^{-4})$ & FOSP              \\
         \textbf{FOTRGS-VR}                                         & FO-Generalized Smooth & $\mathcal{O}(\epsilon^{-3})$ & FOSP              \\
        \textbf{SOTRGS}                                           & \textbf{SO-Generalized Smooth} & $\mathcal{O}(\epsilon^{-3.5})$ & SOSP              \\
        \textbf{SOTRGS-VR}                                        & \textbf{SO-Generalized Smooth} & $\mathcal{O}(\epsilon^{-3})$   & SOSP              \\ \hline
        Lower bound \cite{arjevani2020second}            & Lipschitz             & $\Omega(\epsilon^{-3}) $       & SOSP              \\ \hline
    \end{tabular}%
    \caption{Comparison of related algorithms. FOSP: First-order stationary point; SOSP: Second-order stationary point}\label{tab.overview}
\end{table*}

\paragraph{Distributionally robust optimization} Distributionally robust optimization (DRO) \cite{delageDistributionallyRobustOptimization2010}, originally designed for a middle ground between stochastic programming \cite{shapiroLecturesStochasticProgramming2014} and robust optimization \cite{ben-talRobustOptimization2009}, has attracted great interest in machine learning research communities in recent years for the purposes of distribution shifts and algorithmic fairness \cite{levy2020large,duchi_learning_2021}. For $\phi$-divergence penalized DRO, \citet{levy2020large} prove that it can be transformed into a stochastic optimization problem after duality arguments. \citet{jin2021non} later proves that it fits the settings \loll-smoothness that opens the possibility of a better understanding of first-order methods.

\paragraph{Trust region methods}
Trust region methods are renowned for their ability to reliably find second-order stationary points \cite{conn2000trust}. For stochastic optimization, \citet{shen2019stochastic} proposed a sample-efficient stochastic trust region (STR) algorithm for finite-sum minimization problems and achieved \( \mathcal{O}(\sqrt{n}/\epsilon^{1.5}) \) complexity to find \( (\epsilon,\sqrt{\epsilon}) \)-SOSP. Other works \cite{curtis2019stochastic, curtis2020fully} tackled the fully stochastic setting and proved they
could achieve \( \mathcal{O}(\epsilon^{-3.5}) \) complexity to find \( (\epsilon,\sqrt{\epsilon}) \)-SOSP.
Trust region methods are also widely used in the real of policy optimization \cite{schulman2015trust,liu2023stochastic}.
However, despite these advances, none of the previous studies have explored the properties of trust region methods under the generalized smoothness setting.

\paragraph{Variance reduction techniques}
Variance reduction techniques are first applied to accelerate the convergence speed of SGD for convex finite-sum optimization problems \cite{johnson2013accelerating,zhang2013linear,wang2013variance}.
As for the non-convex setting, Stochastic variance-reduced gradient (SVRG) and Stochastically Controlled Stochastic Gradient (SCSG) improves the convergence rate to a first-order stationary point from $\mathcal{O}(\epsilon^{-4})$ to $\mathcal{O}(\epsilon^{-10/3})$ \cite{allen2016variance, reddi2016stochastic,lei2017non}. Recently, several new variance reduction techniques are able to achieve the optimal complexity rate of $\mathcal{O}(\epsilon^{-3})$ \cite{fang2018spider,cutkosky2019momentum,tran2019hybrid,liu2020optimal,li2021page}. In this paper, we use the techniques in \citet{fang2018spider}  to construct the variance-reduced trust region methods.

\section{Preliminaries}
\paragraph{Notations} For a square matrix $A\in\mathbb{R}^{n\times n}$, we define norm for matrix as $\|A\|=\sqrt{\sigma_M}$, where $\sigma_M$ is the  eigenvalue of $A^T A$ with largest absolute value. For a vector $v\in\mathbb{R}^n$, we use $\|v\|$ to express the standard Euclidean norm. $\|v\|_A:=\sqrt{v^T A v}$ where $A$ is a positive-definite matrix. We assert that objective function $F$ is bounded below throughout the paper and define $F^*:=\inf_{x} F(x)>-\infty $, $\Delta_F := F(x_0) - F^*$.

We review preliminary characteristics of $(L_0, L_1)$-smooth functions introduced in prior
works. In the pioneer work~\cite{zhang2019gradient}, a function $F$ is said to be $\left(L_0, L_1\right)$ smooth if there exist constants $L_0>0$ and $L_1 \geq 0$ such that for all ${x} \in \mathbb{R}^n$,
\begin{equation}\label{eq:old l0l1}
    \left\|\nabla^2 F({x})\right\| \leq L_0+L_1\|\nabla F({x})\| .
\end{equation}
Note that the twice-differentiability assumption in this definition could be relaxed. Specifically, we adopt the ($\left. L_0, L_1\right)$-smoothness assumption as follows:

\begin{assumption}\label{assm L0L1}
    (($\left. L_0, L_1\right)$-smoothness). A differentiable function $F$ is said to be $\left(L_0, L_1\right)$-smooth if there exist constants $L_0>0$, $L_1 \geq 0$ such that if $\|x-x'\| \leq 1 / L_1$, then
    \begin{equation*}
         \|\nabla F(x)-\nabla F(x')\| \leq\left(L_0+L_1\|\nabla F(x)\|\right)\|x-x'\|.
    \end{equation*}
\end{assumption}

If $F$ is twice differentiable, Assumption \ref{assm L0L1} implies condition \eqref{eq:old l0l1}. Moreover, condition \eqref{eq:old l0l1} implies Assumption \ref{assm L0L1} with constants $(2L_0,2L_1)$ (see \cite{reisizadeh2023variance}). We then state the required condition on the noise of the stochastic gradient.
\begin{assumption}\label{assm gradient variance}
    ($(G_0,G_1)$-bounded gradient variance)
    The stochastic gradient $\nabla f(\cdot ; \xi)$ is unbiased and $(G_0,G_1)$-variance-bounded, that is,
    $$
        \begin{aligned}
            \mathbb{E}_{\xi}[\nabla f(x ; \xi)]                 & =\nabla F(x),                        \\
            \mathbb{E}_{\xi}\|\nabla f(x ; \xi)-\nabla F(x)\|^2 & \leq G_0^2+G_1^2\|\nabla F(x)\|^2 .
        \end{aligned}
    $$
\end{assumption}
Note that $(G_0,G_1)$-bounded variance is more general than the standard bounded variance assumption $\mathbb{E}_{\xi}\|\nabla f(x ; \xi)-\nabla F(x)\|^2 \leq \sigma^2$. We extend standard assumptions to Assumption \ref{assm gradient variance} following \citet{faw2023beyond}. In addition, one can verify that DRO satisfies Assumption \ref{assm L0L1} and \ref{assm gradient variance}; for details, see Section \ref{sec:DRO}.

Let $\mcal S$ be the batch of samples. We define the
batch stochastic component function by
\begin{align*}
    f(x;\mcal S) :=\frac{1}{|\mcal S|} \sum_{\xi\in \mcal S} f(x;\xi).
\end{align*}

Our goal is to find first-order and second-order stationary points defined as follows.
\begin{definition}
    We say that $x$ is a first-order approximate stationary point ~($\epsilon$-FOSP) of $F(\cdot)$ if 
    \[
        \|\nabla F(x)\| \leq c_1 \cdot \epsilon.
    \]
    We say that $x$ is a second-order approximate stationary point ~(($\epsilon,\sqrt{\epsilon}$)-SOSP) of $F(\cdot)$ if 
    \[
        \|\nabla F(x)\| \leq c_1 \cdot \epsilon, \; \lambda_{\min}(\nabla^2 F(x))\geq -c_2 \cdot \sqrt{\epsilon}
    \]
    for some positive constants $c_1, c_2 > 0$.
\end{definition}

\paragraph{DRO} Instead of assuming a known underlying probability distribution, DRO minimizes the worst-case loss over a set of distributions $Q$ around the original distribution $P$. 
This can be formally stated as the following  problem \cite{delageDistributionallyRobustOptimization2010,rahimian2019distributionally,shapiro2017distributionally}:
$$
    \min_{x \in \mathbb{R}^n} \quad \Psi(x):=\sup _{Q \in \mathcal{U}(P)} \mathbb{E}_{\xi \sim Q}[\ell(x ; \xi)],
$$
Here, $\xi$ is some random sample and $\ell(x,\xi)$ stands for the stochastic loss function.
The uncertainty set $\mathcal{U}(P)$ with respect to certain distance measure $d$ is defined as 
$
    \mathcal{U}(P) := \{Q: d(Q,P)\leq r\}.
$

Another popular and equivalent formulation of DRO is to add a regularization term rather than imposing the uncertainty set constraints,  which leads to the penalized DRO form:
\begin{equation}\label{eq:primal problem}
    \min_{x \in \mathbb{R}^n} \quad \Psi(x):=\sup _Q\left\{\mathbb{E}_{\xi \sim Q}[\ell(x ; \xi)]-\lambda d(Q, P)\right\},
\end{equation}
where $\lambda>0$ is the prespecified regularization weight. 
In this paper, we adopt the widely used $\psi$-divergence~\cite{shapiro2017distributionally}. 
The $\psi$-divergence between $Q$ and $P$ is defined as $d_\psi(Q, P):=\int \psi\left(\frac{\mathrm{d} Q}{\mathrm{~d} P}\right) \mathrm{d} P$, where $\psi$ is a valid divergence function, namely,  $\psi$ is non-negative, and it satisfies $\psi(1)=0$ and $\psi(t)=+\infty$ for all $t<0$. The conjugate function $\psi^*$ is defined as $\psi^*(t):=\sup _{s \in \mathbb{R}}(s t-\psi(s))$.

\section{Methodology}\label{sec:method}
In this section, we first propose a unified trust region framework for generalized smoothness. 
Then, by specifying a scaling matrix, we give a general first-order trust region algorithm that covers normalized gradient and clipped gradient methods. Moreover, we devote ourselves to a second-order theory of smoothness based on which a second-order trust region method is introduced. 
We also extend our framework to include variance-reduced versions for both first-order and second-order trust region methods. Lastly, we discuss inexact second-order variants to facilitate scalable implementations. For the sake of brevity, we have relegated all proofs of the theoretical results to the appendix.

\subsection{A Unified Trust Region Framework for Generalized Smoothness}
We now introduce our unified trust region framework for generalized smoothness, as described in Algorithm~\ref{alg:unified}. In each iteration, the framework involves solving  the following constrained quadratic subproblem
\begin{equation}\label{eq:unified}
    \begin{aligned}
         & \min_{d \in \mathbb{R}^n} &  & m_t(d):=F(x_t)+g_t^T d+\frac{1}{2} d^T B_t d \\
         & \   \textup{ s.t.}        &  & \|d\| \leq \Delta_t,
    \end{aligned}
\end{equation}
It is important to note that the square matrix $B_t$ is not predetermined in this abstract framework.  By making different choices for $B_t$, we can develop more specific first- and second-order methods under this unified framework. For example, both normalized gradient and clipped gradient can be viewed as a special case with a certain choice of $B_t$. 
As our analysis will demonstrate, when only first-order information is available, the trust-region algorithm guarantees convergence as long as $B_t$ has a bounded norm. Furthermore, leveraging second-order information can enhance our convergence towards high-order optimality conditions.
\begin{algorithm}[ht]
    \caption{The trust region framework}\label{alg:unified}
    \footnotesize
    \begin{algorithmic}[1]
        \STATE Given $T$, error $\epsilon$
        \FOR{$t=0,1, \dots ,T-1$}
        \STATE  Draw samples $\mcal S_1$ and compute $g_t=\nabla f(x_t;\mcal S_1)$
        \STATE (if needed) Draw samples $\mcal S_2$ and compute $H_t=\nabla^2 f(x_t;\mcal S_2)$
        \STATE Compute  step $d_{t+1}$ by solving the subproblem \eqref{eq:unified}
        \STATE Update: $x_{t+1} \gets x_{t}+d_{t+1}$
        \ENDFOR
    \end{algorithmic}
\end{algorithm}

For generality, we first provide some important properties about the solution of subproblem~\eqref{eq:unified}.
By the optimality condition of subproblem(\citet{conn2000trust}, the vector $d_{t+1}$ is the global solution to problem \ref{eq:unified} if and only if there exists a Lagrange multiplier $\lambda_t$ such that $(d_{t+1},\lambda_t)$ is the solution to the following equations:
\begin{equation}\label{eq:optimal condition}
    (B_t+\lambda I)d + g_t = 0, \lambda(\Delta_t-\norm{d})=0, (B_t + \lambda I)\succeq 0
\end{equation}

\begin{lemma}[Model reduction] \label{lem:model reduction unified} For any matrix variable $B_t$, at the $t$-th iteration, let $d_{t+1}$ and $\lambda_t$ be the optimal primal and dual solution of \eqref{eq:optimal condition}. We have the following amount of decrease on $m_t$
    $$m_t(d_{t+1}) - m_t(0) \leq -\frac{1}{2}\lambda_t \norm{d_{t+1}}^2.$$
\end{lemma}

\subsection{First-Order Trust Region Methods}
We first consider the first-order trust region method for generalized smoothness, FOTRGS, where only gradient information is used in Algorithm \ref{alg:unified}. We show that as long as $\norm{B_t}$ is uniformly bounded by a constant, by setting proper parameters, Algorithm \ref{alg:unified} is able to return an $\epsilon$-FOSP. 

\begin{theorem}[Sample complexity of FOTRGS] \label{thm:fo_psd}
    Suppose Assumption \ref{assm L0L1} - \ref{assm gradient variance} hold. 
    Let $B_t$ be a matrix with bounded norm i.e. there exists a constant $\beta$ such that $\norm{B_t}\leq \beta$. 
    By setting $\epsilon\leq \min\left\{\frac{4L_0G_0+16\beta G_0}{L_1G_0+2L_0G_1+8\beta G_1},\frac{4L_0+16\beta}{L_1}\right\}$, $\Delta_t=\Delta=(4L_0+16\beta)^{-1}\epsilon$, $|\mcal S_1| = 64G_0^2\epsilon^{-2}$, $T = 32\Delta_F (L_0+4\beta)\epsilon^{-2}$ in Algorithm \ref{alg:unified}, we have
    $ \Expect\norm{\nabla F(x_{\bar{t}})} \leq \epsilon, $
    where $\bar{t}$ is sampled from $\{0,1,\ldots,T-1\}$ uniformly at random. 
    Moreover, the sample complexity of finding an $\epsilon$-FOSP is bounded by $$\mcal O\left(\frac{\Delta_F (L_0+\beta) G_0^2}{\epsilon^{4}}\right).$$
\end{theorem}

When fixing $B_t$ as specific constants, we are able to represent the normalized and clipped gradient method in this framework.  To be specific, if we set $B_t=0$, then we are able to cover the normalized gradient descent method in trust region framework. 
\begin{corollary}[Equivalence to the normalized method] \label{coro:nsgd}
    Suppose Assumption~\ref{assm L0L1} - \ref{assm gradient variance} hold. Let $B_t=0$ in Algorithm \ref{alg:unified}, then the solution of the subproblem \eqref{eq:unified} is
    $$d_{t+1}=\frac{\Delta_t }{\norm{g_{t}}}\cdot (-g_t).$$ 
    By setting $\epsilon\leq \min\left\{\frac{4L_0G_0}{L_1G_0+2L_0G_1},\frac{4L_0}{L_1}\right\}$, $\Delta_t=\Delta=(4L_0)^{-1}\epsilon$, $|\mcal S_1|= 64G_0^2\epsilon^{-2}$, $T = 32\Delta_F L_0\epsilon^{-2}$, we have
    $ \Expect\norm{\nabla F(x_{\bar{t}})} \leq \epsilon, $
    where $\bar{t}$ is sampled from $\{0,1,...,T-1\}$ uniformly at random. Moreover, the sample complexity of finding an $\epsilon$-FOSP is bounded by
    $\mcal O\left(\frac{\Delta_F L_0 G_0^2}{\epsilon^{4}}\right).$
\end{corollary}

By setting $B_t=\rho I$, we are also able to represent the clipped method in this unified framework.

\begin{corollary}[Equivalence to the clipped method] \label{coro:clip}
    Suppose Assumption \ref{assm L0L1} - \ref{assm gradient variance} hold. Let $B_t=\rho I$ in Algorithm \ref{alg:unified}, then the solution of the subproblem \eqref{eq:unified} is
    $$d_{t+1}=\min\left\{\frac{\Delta_t}{\norm{g_{t}}},\frac{1}{\rho}\right\}\cdot(-g_{t}).$$
    By setting $\epsilon\leq \min\ \left\{\frac{4L_0G_0+16\rho G_0}{L_1G_0+2L_0G_1+8\rho G_1}, 4L_0+16\rho L_1^{-1}\right\}$, $\Delta_t=\Delta=(4L_0+16\rho)^{-1}\epsilon$, $|\mcal S_1|=64G_0^2\epsilon^{-2}$, $T = 32\Delta_F (L_0+4\rho)\epsilon^{-2}$ in Algorithm \ref{alg:unified}, we have
    $\Expect\norm{\nabla F(x_{\bar{t}})} \leq \epsilon, $
    where $\bar{t}$ is sampled from $\{0,1,\ldots,T-1\}$ uniformly at random. Moreover, the sample complexity of finding an $\epsilon$-FOSP is bounded by $\mcal O\left(\frac{\Delta_F (L_0+\rho) G_0^2}{\epsilon^{4}}\right).$
\end{corollary}

A few remarks are in order. First, it's worth noting that our proposed first-order trust-region method offers greater flexibility in step size compared to normalized and clipped gradient methods, as we can choose different $B_t$ values in each iteration. Exploring more choices for $B_t$ remains an interesting direction for future research. Second, our complexity results closely align with some recent work.  For instance, the prior work \cite{reisizadeh2023variance} has analyzed the convergence rate of the clipped method. Under similar assumptions, an $\epsilon$-FOSP can be found by the clipped method with $\mathcal{O}(\epsilon^{-4})$ gradient samples. A key distinction between our analysis and prior work lies in the variance bound requirements on stochastic gradients. Specifically, while \citet{reisizadeh2023variance} requires a uniform variance bound
$ \mathbb{E}_{\xi}\|\nabla f(x ; \xi)-\nabla F(x)\|^2 \leq \sigma^2$,
we allow for a variance bound related to the gradient norm of the current point, as stated in Assumption~\ref{assm gradient variance}. 
This makes our analysis more general and extends its applicability to the DRO setting.

\subsection{A Second-Order Theory of Generalized Smoothness}\label{sec:DRO}
This subsection introduces a generalized second-order smoothness condition, drawing inspiration from the $(L_0, L_1)$-smoothness concept.  Subsequently, we demonstrate that DRO is a significant application that aligns with this newly proposed second-order condition.
\begin{assumption}[Second-order generalized smoothness and variance condition]\label{assm: hessian}
$F$ is twice-differentiable and satisfies that there exist constants $\delta>0$, $M_0>0$ and $M_1\geq 0$ such that if $\|x-x'\| \leq \delta$, then
    \begin{equation*}
        \|\nabla^2 F(x)-\nabla^2 F(x')\|\leq (M_0+M_1\|\nabla F(x)\|)\|x-x{'}\|.
    \end{equation*}
    Moreover, the stochastic Hessian is unbiased and $(K_0,K_1)$ variance-bounded, that is,
    $$
        \begin{aligned}
            \mathbb{E}_{\xi}[\nabla^2 f(x ; \xi)]                   & =\nabla^2 F(x),                      \\
            \mathbb{E}_{\xi}\|\nabla^2 f(x ; \xi)-\nabla^2 F(x)\|^2 & \leq K_0^2+K_1^2\|\nabla F(x)\|^2 .
        \end{aligned}
    $$
\end{assumption}
Similar to the $(L_0,L_1)$-smoothness, we can interpret the proposed second-order generalized smoothness from the perspective of the boundness of higher-order derivatives. Further discussion of this condition can be found in the appendix. We claim that Penalized DRO \eqref{eq:primal problem} satisfies this assumption. The original formulation  involves a max operation over distributions, which makes optimization challenging. By duality arguments (see details in \citet[Section A.1.2]{levy2020large}), we can write \eqref{eq:primal problem}  equivalently as 
\begin{equation}\label{dual problem}
    \Psi(x)=\min _{\eta \in \mathbb{R}} \mathcal{L}(x, \eta) := \lambda \mathbb{E}_{\xi \sim P} \psi^*\left(\frac{\ell(x ; \xi)-\eta}{\lambda}\right)+\eta .
\end{equation}
This suggests that to minimize the DRO objective, one can perform a joint minimization of $\mathcal{L}(x, \eta)$ over $(x, \eta) \in \mathbb{R}^{n+1}$. Crucially, it is sufficient to find an $(\epsilon,\sqrt{\epsilon})$-SOSP of $\Psi(x)$ by optimizing $\mathcal{L}(x, \eta)$ instead. To establish this relationship more formally, we build the connection between the gradient and Hessian of $\Psi(x)$ and those of $\mathcal{L}(x, \eta)$ as follows.
\begin{theorem}\label{lm:connection2}
    Under mild assumptons for $\ell \text{ and } \psi^*$, if some $(x,\eta)$ is a $(\epsilon,\sqrt{\epsilon}) $-SOSP for $\mathcal{L}(x,\eta)$,  then $x$ is also a $(\epsilon,\sqrt{\epsilon}) $-SOSP for $\Psi(x)$.
\end{theorem}
The following theorem analyzes the smoothness and variance properties of $\mathcal{L}(x,\eta)$, which motivates us to propose our second-order generalized smoothness and variance conditions.
\begin{theorem}\label{thm:dro}
    Under mild assumptons for $\ell \text{ and } \psi^*$, the objective $\mathcal{L}(x,\eta)$, serving as $F$, satisfies Assumption \ref{assm L0L1}, \ref{assm gradient variance} and \ref{assm: hessian}.
\end{theorem}

\subsection{Second-Order Trust Region Methods}
We propose SOTRGS by setting $B_t=H_t$ in Algorithm~\ref{alg:unified}. We first present a result on bounding the variance of Hessian.
\begin{lemma}[Variance bounds on Hessian estimators]\label{lm var for hessian}
    Suppose that Assumption~\ref{assm: hessian} holds in Algorithm~\ref{alg:unified}, if we set $|\mcal S_2 |=22\log(n)\epsilon^{-1}$, then
    $$
        \Expect_t\bsbra{\|H_t-\nabla^2 F(x_t)\|^2} \leq
        (K_0^2+K_1^2\|\nabla F(x_t)\|^2)\epsilon,
    $$
    where $\Expect_t$ denotes the  expectation conditioned on all the randomness before the $t$-th iteration.
\end{lemma}

Next, we provide the convergence result of the second-order trust region method in the generalized smoothness setting.

\begin{theorem}[Sample complexity of SOTRGS]\label{thm:full sgd}
    Suppose Assumptions~\ref{assm L0L1}, \ref{assm gradient variance} and \ref{assm: hessian} hold.
    Let $\Delta_t=\Delta=\sqrt{\epsilon}$,
    by setting $B_t=H_t$,  $\epsilon < \min\left\{\frac{3}{5M_1+18G_1+12K_1},\frac{1}{L_1^2}\right\} $, $|\mcal S_1|=\epsilon^{-2}$,  $\left|\mathcal S_2\right|=22\log(n)\epsilon^{-1}$, $T=\mathcal{O}(\epsilon^{-{3}/{2}})$ in Algorithm~\ref{alg:unified}, we have
    $$
        \Expect[\left\|\nabla F\left(x_{\bar{t}+1}\right)\right\|] \leq \mathcal{O}(\epsilon) ,\\
        \Expect[\lambda_{\min}(\nabla^2 F(x_{\bar{t}+1}))]\geq-\mathcal{O}(\sqrt{\epsilon}),
    $$
    where $\bar{t}$ is sampled from $\{0,1, \ldots, T-1\}$ uniformly  at random. Moreover, the sample complexity of finding an $(\epsilon,\sqrt{\epsilon})$-SOSP is bounded by $$\mcal O\left(\frac{\Delta_F}{\epsilon^{7/2}}+\frac{\Delta_F}{\epsilon^{5/2}}\right).$$
\end{theorem}

To our best knowledge, this is the first work to show convergence achieving the second-order stationary points for generalized smooth optimization, and its sample complexity is better than first-order methods without variance reduction techniques.

\subsection{Variance Reduction}
We now turn our attention to the variance-reduced variants of the trust-region method. 
\citet{arjevani2020second} shows that for any $L, \sigma>0$, there exists a function $F$ of the form \eqref{eq:stochastic objective} satisfying $\sigma^2$ bounded variance and expected Lipschitz smoothness
with stochastic gradients $\nabla f(\cdot ; \xi)$ such that
$$
    \mathbb{E}_{\xi}[\nabla f({x} ; \xi)]=\nabla F({x}), \quad \mathbb{E}_{\xi}\|\nabla f({x} ; \xi)-\nabla F({x})\|^2 \leq \sigma^2,
$$
and
$$\mathbb{E}_{\xi}\left[\|\nabla f({x} ; \xi)-\nabla f({x'} ; \xi)\|^2\right]^{1 / 2} \leq L\|{x}-{x'}\|,$$
for which finding an $\epsilon$-stationary solution requires  $\Omega\left(\sigma \epsilon^{-3}+\sigma^2 \epsilon^{-2}\right)$ stochastic gradient queries. Since our generalized smoothness is more general than its requirements, the lower bound can be directly applied to our settings. To close the optimality gap, we employ a variance reduction technique~\cite{fang2018spider} to construct an improved gradient estimator $g_t$. Specifically, if $\bmod (t,q)=0$, then we take
$$ g_t = \nabla f(x_t; \mcal S_1);$$
otherwise, we compute $g_t$ based on the value of $g_{t-1}$
$$ g_t = \nabla f(x_t; \mcal S_3)
    - \nabla f(x_{t-1}; \mcal S_3)  + g_{t-1},$$
where $\mcal S_1,\mcal S_3$ and $q$ are parameters to be determined. The abstract variance-reduced trust region framework is presented in Algorithm~\ref{alg:vr_tr}. Next, we develop the sample complexity of both first-order and second-order variance reduced methods. 
\paragraph{First-order methods}
We apply the above gradient estimator and propose a variance-reduced first-order trust region method, FOTRGS-VR. We give the upper bound of sample complexity for finding an $\epsilon$-FOSP in the following theorem.

\begin{theorem} \label{thm:fo_vr}
    Suppose Assumption \ref{assm L0L1}, \ref{assm gradient variance} and \ref{assm:spider} hold. Let $B_t$ be a positive semi-definite matrix with bounded norm i.e. there exists a constant $\beta$ such that $\norm{B_t}\leq \beta$. By setting $\epsilon\leq\min\left\{\frac{G_1^2}{2L_1^2},\frac{1}{L_1}\right\}$, $\Delta_t=\Delta=\epsilon$, $|\mcal S_1|=\epsilon^{-2}$, $|\mcal S_{3}|=\epsilon^{-1}$,
    $q=(8G_{1}\epsilon)^{-1}$, $T=\mcal O(\epsilon^{-2})$, then we have
    $ \Expect\norm{\nabla F(x_{\bar{t}})} \leq \mcal O(\epsilon)$, 
    where $\bar{t}$ is sampled from $\{0,1,\ldots,T-1\}$ uniformly at random. Moreover, the total complexity of finding an $\epsilon$-FOSP is bounded by $$\mcal O\left(\frac{\Delta_F}{\epsilon^{3}}\right).$$
\end{theorem}

\paragraph{Second-order methods}
To reduce the second-order oracle complexity, we apply the same idea to both the gradient and Hessian estimator in the second-order trust region method. Similar to the analysis of the first-order variance-reduced trust region method, the following theorem gives the upper bound of sample complexity for finding an $(\epsilon,\sqrt{\epsilon})$-SOSP.
\begin{theorem}[Sample complexity of SOTRGS-VR] \label{thm:so_vr} Suppose Assumption~\ref{assm L0L1}, \ref{assm gradient variance}, \ref{assm: hessian} and \ref{assm:spider} hold. Let $B_t$ be the Hessian estimator as shown in Algorithm \ref{alg:vr_tr}. By setting $\epsilon\leq\min\ \left\{\frac{G_1^4}{4L_1^4}, \frac{1}{36G_1^2},\frac{1}{L_1^2}\right\}$, $\Delta_t=\Delta=\sqrt{\epsilon}$,
    $|\mcal S_{1}|=\epsilon^{-2}$, $|\mcal S_{2}|=22\log(n)\epsilon^{-1}$, $|\mcal S_{3}|=\epsilon^{-3/2}$, $T=\mcal O(\epsilon^{-3/2})$, then we have
    $ \Expect\norm{\nabla F(x_{\bar{t}+1})} \leq \epsilon,\ \Expect[\lambda_{\min}(\nabla^2 F(x_{\bar{t}+1}))]\geq-\mathcal{O}(\sqrt{\epsilon}),$ 
    where $\bar{t}$ is sampled from $\{0,1,\ldots,T-1\}$ uniformly at random. Moreover, the total complexity of finding an $(\epsilon,\sqrt{\epsilon})$-SOSP is bounded by $$\mcal O\left(\frac{\Delta_F}{\epsilon^{3}}+\frac{\Delta_F}{\epsilon^{5/2}}\right).$$
\end{theorem}
\begin{algorithm}[h]
    \caption{Variance-reduced trust region method}\label{alg:vr_tr}
    \small
    \begin{algorithmic}[1]
        \STATE Given $T$, error $\epsilon$
        \FOR{$t=0,1, \dots ,T-1$}
        \IF {$\bmod (t,q)=0$}
        \STATE  Draw samples $\mcal S_1$ and compute $g_t=\nabla f(x_t;\mcal S_1)$
        \ELSE
        \STATE  Draw samples $\mcal S_3$ and compute $g_t = g_{t-1}+ \nabla f(x_t; \mcal S_3) - \nabla f(x_{t-1}; \mcal S_3) $
        \ENDIF
        \STATE (if needed) Draw samples $\mcal S_2$ and compute $H_t=\nabla^2 f(x_t;\mcal S_2)$ 
        \STATE Compute  step $d_{t+1}$ by solving the subproblem \eqref{eq:unified}
        \STATE Update: $x_{t+1} \gets x_{t}+d_{t+1}$
        \ENDFOR
    \end{algorithmic}
\end{algorithm}
\subsection{Inexactness and Scalability}
For large-scale machine learning problems, exactly solving the second-order trust region subproblem~\eqref{eq:unified}  can be computationally prohibitive.  To mitigate this, we can relax the need for exact Hessian calculations and subproblem solutions by allowing for inexact approximations.
In the sequel, we assume $\Tilde{\nabla}^2 F(x)$ is the approximation of $\nabla^2 F(x)$. At each $x_t$ we adopt a low-dimensional subspace with orthonormal basis $V_t\in \Rbb^{n\times k}$ for $k\ll n$,  and compute second-order derivatives in the subspace.
Inspired by the work of~\citet{cartis_adaptive_2011,zhang2022drsom}, we propose the following regularity assumption on inexactness in Hessian approximation.
\begin{assumption}\label{assm:drsom} For certain constants $C_0,C_1>0$, there exists a $V_t$ whose columns form an orthonormal basis such that
    {\small$$\|(\nabla^2 F(x_t)-\tilde{\nabla}^2 F(x_t))d_{t+1}\|\leq (C_0+C_1\|\nabla F(x_t)\|)\|d_{t+1}\|^2,$$}
    where $\tilde \nabla^2 F(x) := V_tV_t^T\nabla^2 F(x)V_tV_t^T$ is the projected Hessian in the column space of $V_t$.
\end{assumption}
Setting $B_t=\tilde{H}_t:=V_tV_t^\trans H_tV_tV_t^\trans$ and then using the auxiliary variable $y=V_t^Tx^t$, Algorithm~\ref{alg:unified} only needs to  solve an approximate trust-region subproblem with a much lower dimension.
Theoretically, our new assumption can be satisfied in various ways \cite{xuNewtontypeMethodsNonconvex2020,cartisEvaluationComplexityAlgorithms2022}. We leave the details in Appendix \ref{sec.lowrsogs}.

The following theorem provides the performance bound of the inexact version of the second-order trust region method.
\begin{theorem}\label{thm drtr}
    Suppose Assumptions~\ref{assm L0L1}, \ref{assm gradient variance}, \ref{assm: hessian} and \ref{assm:drsom} hold. Let $B_t=V_tV_t^\trans H_tV_tV_t^\trans$.
    In Algorithm~\ref{alg:unified}, let $\epsilon < \min\big\{\frac{3}{5M_1+18G_1+24K_1+6C_1},\frac{1}{L_1^2}\big\}$, $\Delta_t=\Delta=\sqrt{\epsilon}$, $|\mcal S_1|=\epsilon^{-2}$,  $|\mcal S_2|=22\log(n)\epsilon^{-1}$, $T=\mathcal{O}(\epsilon^{-3/2})$, then we have
    $$
        \Expect[\left\|\nabla F(x_{\bar{t}+1})\right\|] \leq \mathcal{O}(\epsilon),~
        \Expect[\lambda_{\min}(\nabla^2 \Tilde{F}(x_{\bar{t}+1}))]\geq-\mathcal{O}(\sqrt{\epsilon}),
    $$
    where $\bar{t}$ is sampled from $\{0,1, \ldots, T-1\}$ uniformly  at random.
\end{theorem}

\section{Experiments}

We perform three sets of experiments in machine learning with a focus on DRO to justify our analysis. Due to space limitation, we only present a brief description and the tuned methods with the best performance for SGD, FOTRGS, and SOTRGS; complete details are left in Appendix \ref{sec.lowrsogs}.

\subsection{Basic Settings}

We focus on classification tasks with \textbf{imbalanced} distributions arising from applications with heterogeneous (but often latent) subpopulations. Since in standard datasets like MNIST, Fashion MNIST and CIFAR-10, the population ratios (number of images per class) are the same, we create a perturbed dataset that inherits a disparity \cite{hashimotoFairnessDemographicsRepeated2018} by choosing only a subset of training samples for each one of the categories. Since all these datasets consist of 10 categories, we fix them at a uniform set of levels without loss of generality. In all the tests, the worst class only takes a proportion of $0.254$ from the samples; after preprocessing, we only use $33,260$ out of the original $50,000$ training samples.

We adopt penalized DRO  for classification tasks with two specific divergence functions satisfying Assumption \ref{assm: hessian}: the smoothed $\chi^2$ and smoothed CVaR. To fairly compare the algorithms, we perform a grid search over the parameters. The complete description is left in Appendix \ref{sec:expappendix}.

\subsection{Experiment Results}
The results show the trust region methods are efficient in DRO with second-order generalized smoothness in training efficiency and test accuracy, especially for minority classes. In all our experiments, we do not differentiate between smoothed and the original divergence functions and may use them interchangeably. 
\begin{figure}[thb!]
    \quad\begin{subfigure}{.99\columnwidth}
    \captionsetup{font={small}}
        \centering
        \includegraphics[width=\linewidth]{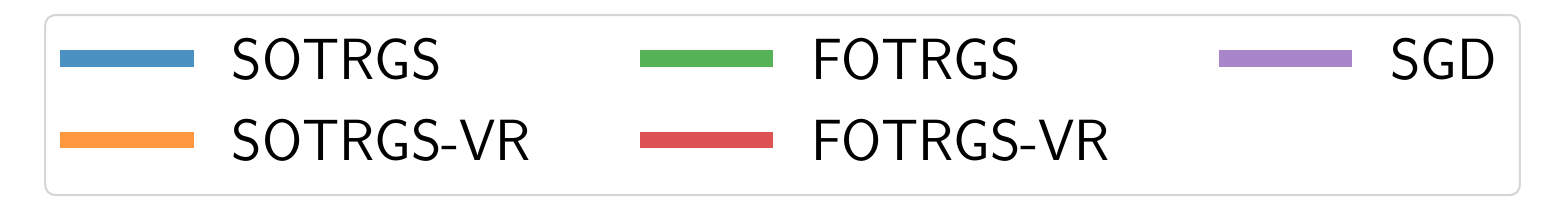}
    \end{subfigure}
    \begin{subfigure}{.24\columnwidth}
        \includegraphics[width=1.02\linewidth]{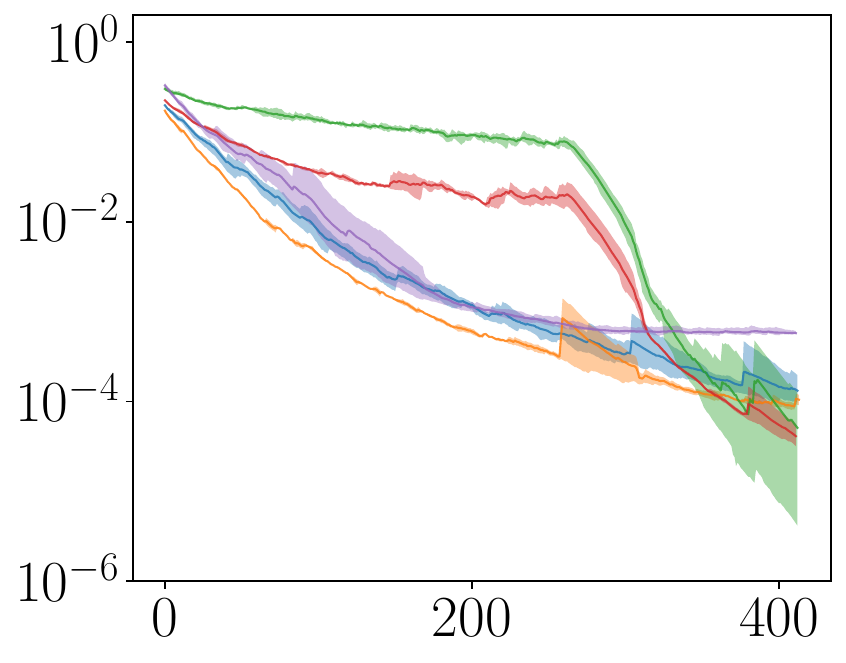}
        \subcaption{MNIST with $\chi^2$}\label{fig.mnist.x2}
    \end{subfigure}
    \begin{subfigure}{.24\columnwidth}
        \includegraphics[width=1.02\linewidth]{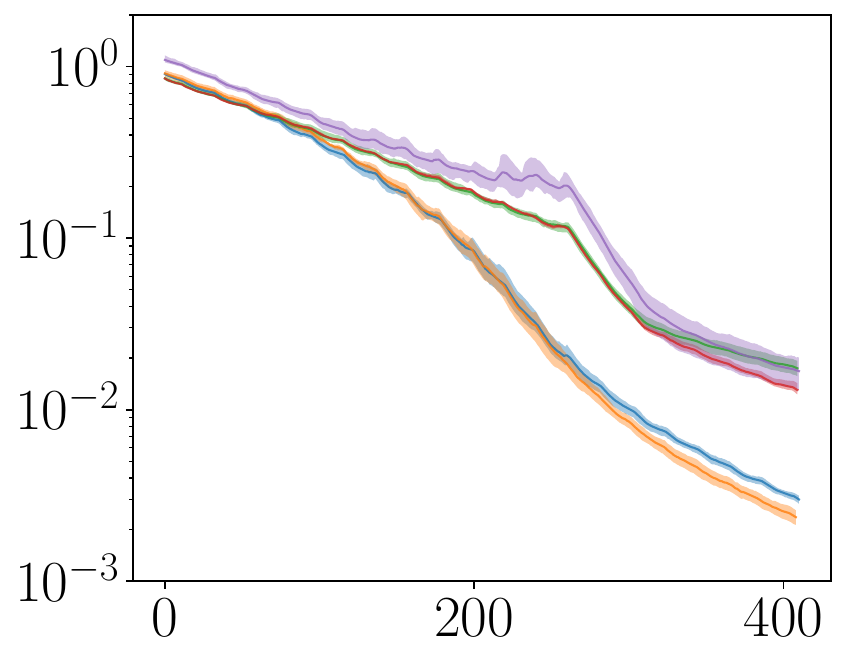}
        \subcaption{F-M with $\chi^2$}\label{fig.fmnist.x2}
    \end{subfigure}
     \begin{subfigure}{.24\columnwidth}
    \captionsetup{font={small}}
        \centering
        \includegraphics[width=1.02\linewidth]{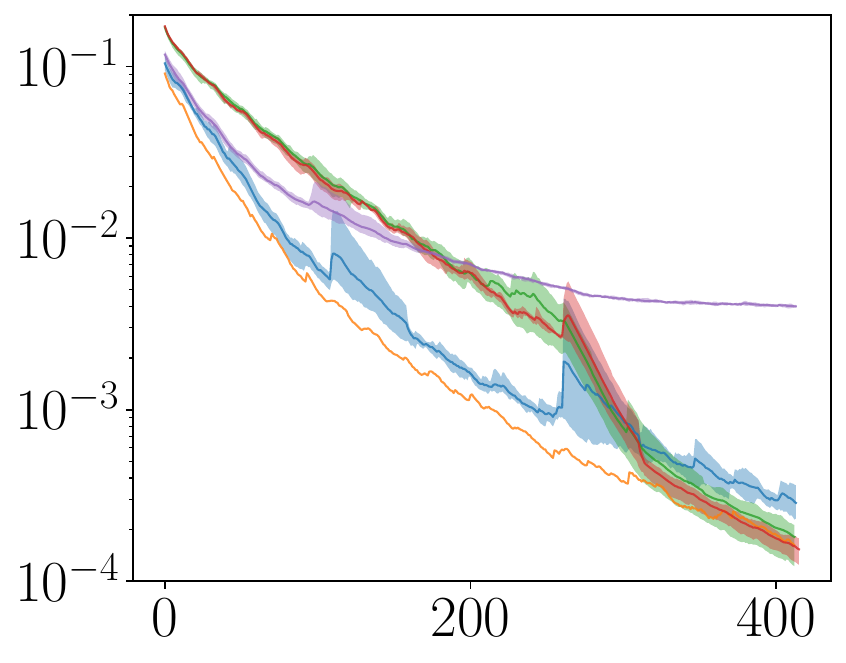}
        \subcaption{MNIST with CVaR }\label{fig.mnist.cvar}
   \end{subfigure}
    \begin{subfigure}{.24\columnwidth}
    \captionsetup{font={small}}
        \centering
        \includegraphics[width=1.02\linewidth]{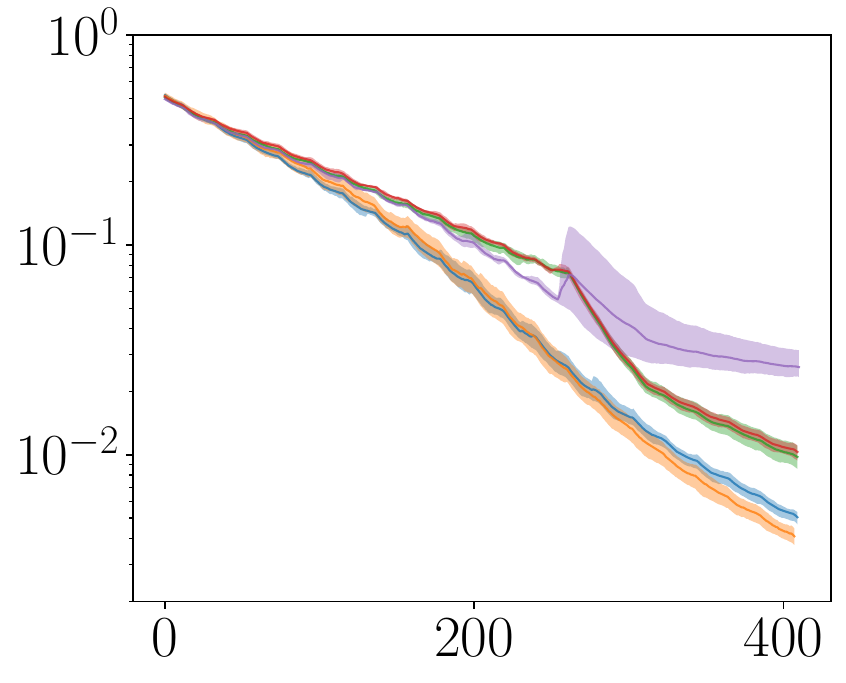}
        \subcaption{F-M with CVaR}\label{fig.fmnist.cvar}
    \end{subfigure}
    
    \caption{Training curves with different smoothed DRO loss on imbalanced MNIST and Fashion-MNIST datasets. We report the per-step losses by aggregating every $20$ iteration. Shaded areas indicate the range of variability across 5 repetitions.}
\normalsize
\end{figure}
Figure (\ref{fig.mnist.x2}) and (\ref{fig.fmnist.x2}) present the training curves of SGD, first-order (FOTRGS) and second-order (SOTRGS) trust region methods on MNIST and Fashion MNIST datasets, respectively.  The normalized SGD outperforms standard SGD as a representative of the FOTRGS family. Furthermore, it is clear that SOTRGS accelerates the rate of convergence in all our tests. We leave test results in Appendix \ref{sec:expappendix}. 

Table (\ref{tab.compact.x2}) and (\ref{tab.compact.cvar}) presents the test accuracy of different methods. It is clear that the trust region methods have an advantage in preserving fairness for minority classes while also achieving the best overall average performance.

\begin{table}[h!]
\small
\begin{subtable}[h]{0.98\columnwidth}
\centering
    \begin{tabular}{lrrrrrrrrrr}
    \toprule
    {} &      Worst Category &      Overall Accuracy &  \\
    \midrule
    SOTRGS  &  0.681 &         0.889 &  \\
    FOTRGS  &  0.705 &         0.898 &  \\
    \midrule
    SGD     &  0.629 &         0.894 &  \\
    \bottomrule
    \end{tabular}
     \caption{Imbalanced CIFAR10 with $\chi^2$ loss.}\label{tab.compact.x2}
\end{subtable}
\begin{subtable}[h]{0.98\columnwidth}
\centering
\begin{tabular}{lrrrrrrrrrr}
    \toprule
    {} &      Worst Category &      Overall Accuracy &  \\
    \midrule
    SOTRGS  &  0.616 &         0.896 &  \\
    FOTRGS  &  0.615 &         0.899 &  \\
    \midrule
    SGD     &  0.607 &         0.888 &  \\
    \bottomrule
    \end{tabular}
    \caption{Imbalanced CIFAR10 with CVaR loss}\label{tab.compact.cvar}
\end{subtable}
    \caption{Test accuracy on imbalanced CIFAR10. Besides overall test accuracy, we also present the worst-performing class indicated as the ``worst category''.}\label{tab.compact.all}
    \normalsize
\end{table}

\section{Discussion}
This work opens up several intriguing avenues for future exploration. One primary question that emerges is whether machine learning problems, beyond DRO, exhibit properties of second-order generalized smoothness.  Additionally, it would be interesting to see whether our framework can be extended to more complex scenarios, which involve either a constrained domain or  an additional non-smooth proximal term in the objective.

\bibliography{ref}
\newpage{}

\appendix
\begin{center}
    \huge Appendix
\end{center}
 
\paragraph{Appendix Structure} The appendix sections will proceed as follows. Section \ref{Sec: unified frame} provides proof of the properties of the unified trust region framework. Section \ref{Sec: appx-FO} provides proof of first-order algorithms, and Section \ref{Sec: Appx_SO} provides proof of second-order algorithms. Section \ref{sec:SOGS} provides further discussion of second-order generalized smoothness. Section \ref{sec:moredro} provides applications on divergence-based penalized DRO and proves that it satisfies our second-order generalized smoothness. Section \ref{Sec：appx-VR} provides proof for variance-reduced variants. Section \ref{sec.lowrsogs} and Section \ref{sec:expappendix} provides the guidelines about how to implement second-order methods efficiently and the details of the experiments respectively.
\section{Proof of the Unified Framework}\label{Sec: unified frame}
When $L_0=0$, Assumption 1 simplifies to the $L_1$ smoothness employed in standard analysis. Similar to the analysis of $L$-smooth functions, we give the following descent inequality for $\left(L_{0},L_{1}\right)$-smooth functions.

\begin{appxlem}\label{lem:fo_decrease} Under Assumption \ref{assm L0L1}, when $\|x-x'\|\leq \frac{1}{L_1}$, we have
    \begin{align*}
        F(x') \leq F(x) + \nabla F(x)^T(x'-x)+ \tfrac{1}{2}\norm{x'-x}^2\brbra{L_0+L_1\norm{\nabla F(x)}}.
    \end{align*}
\end{appxlem}
We omit the proof since it is well-known in the literature for $\left(L_{0},L_{1}\right)$ smoothness.

\subsection{Proof of Lemma \ref{lem:model reduction unified} (Model reduction)}
\begin{proof}
    According to the optimality condition \eqref{eq:optimal condition}, we have
    \begin{align*}
        \begin{split}
            m_t(d_{t+1}) - m_t(0) &= g_t^T d_{t+1} + \frac{1}{2}d_{t+1}^T B_t d_{t+1} \\
            &= -\lambda_t d_{t+1}^T d_{t+1} - \frac{1}{2}d_{t+1}^T B_t d_{t+1} \\
            &= -\frac{1}{2}\lambda_t\norm{d_{t+1}}^2 - \frac{1}{2}d_{t+1}^T \brbra{B_t+\lambda_t I} d_{t+1} \\
            &\leq -\frac{1}{2}\lambda_t\norm{d_{t+1}}^2.
        \end{split}
    \end{align*}
    The last inequality is because $B_t + \lambda_t I \succeq 0$.
\end{proof}
\section{Proof of First-Order Methods}\label{Sec: appx-FO}
We first present a result on bounding the gradient variance.
\begin{appxlem}[Variance bounds on gradient estimators]\label{lm sgd_var-bound for g-estimator} Suppose that Assumption~\ref{assm gradient variance} holds in Algorithm~\ref{alg:unified}, then we have
    $$
        \Expect_t\left\|g_t-\nabla F(x_t)\right\|^2 \leq \frac{G_0^2+G_1^2\|\nabla F(x_t)\|^2}{|\mcal S_1|},
    $$
    where $\Expect_t$ denotes the  expectation conditioned on all the randomness before the $t$-th iteration.
\end{appxlem}
\begin{proof}
    \begin{equation*}
        \begin{aligned}	\Expect_t\bsbra{\left\|{g}_t-\nabla F(x_t)\right\|^2}
             & =\Expect_t\Big[\Big\|\frac{1}{|\mathcal{S}_1|}\sum_{\xi \in \mathcal{S}_1}\nabla f(x_t;\xi)-\nabla F(x_t)\Big\|^2\Big] \\
             & =\frac{\sum_{\xi \in \mathcal{S}_1}\Expect_t\bsbra{\norm{\nabla f(x_t;\xi) - \nabla F(x_t)}^2}}{|\mathcal{S}_1|^2}     \\
             & \leq \frac{G_0^2+G_1^2 \|\nabla F(x_t)\|^2}{|\mcal S_1|}.
        \end{aligned}
    \end{equation*}
    The last inequality is because Assumption \ref{assm gradient variance}.
\end{proof}

\subsection{Proof of Theorem \ref{thm:fo_psd}}
Firstly, we give the following lemma.

\begin{appxlem} \label{lem:psd} For $\lambda_t$ satisfing the optimality condition \eqref{eq:optimal condition} of the subproblem \eqref{eq:unified}, we have
    \begin{align} \label{eq:bound_for_labd}
        \lambda_t \geq \frac{\norm{g_t}}{\Delta} -\norm{B_t}.
    \end{align}
\end{appxlem}

\begin{proof} According to the optimality condition \eqref{eq:optimal condition}, we have $-g_t = \brbra{B_t+\lambda_t I}d_{t+1}$ and
    \begin{align*}
        \begin{split}
            \|g_t\|&= \left\|\brbra{B_t+\lambda_t I}d_{t+1}\right\| \leq \left\|B_t+\lambda_t I\right\|\cdot\|d_{t+1}\| \leq \brbra{\|B_t\|+\lambda_t} \Delta.
        \end{split}
    \end{align*}
    Therefore, we have
    \begin{align*}
        \begin{split}
            \lambda_t \geq \frac{\norm{g_t}}{\Delta} -\norm{B_t}.
        \end{split}
    \end{align*}
\end{proof}

Then we prove Theorem \ref{thm:fo_psd} based on Lemma \ref{lem:psd}.

\begin{proof}According to Lemma \ref{lem:fo_decrease}, we have
    \begin{equation} \label{eq:fo_descent}
        \begin{split}F(x_{t+1})\leq & F(x_{t})+\inprod{\nabla F(x_{t})}{x_{t+1}-x_{t}}+\frac{1}{2}\brbra{L_{0}+L_{1}\norm{\nabla F(x_{t})}}\norm{x_{t+1}-x_{t}}^{2}\\
            \overset{(\Diamond)}{=} & F(x_{t}) + \inprod{\nabla F(x_{t})-g_t}{d_{t+1}} -\lambda_t\norm{d_{t+1}}^2 - d_{t+1}^TB_td_{t+1} + \frac{1}{2}\brbra{L_{0}+L_{1}\norm{\nabla F(x_{t})}}\norm{d_{t+1}}^2 \\
            \overset{(\triangle)}{\leq} & F(x_{t}) + \norm{\nabla F(x_{t})-g_{t}}\Delta - \lambda_t\Delta^2 + \norm{B_t}\Delta^2 + \frac{1}{2}\brbra{L_{0}+L_{1}\norm{\nabla F(x_{t})}}\Delta^2 \\
            \overset{(\vee)}{\leq} & F(x_{t}) + \norm{\nabla F(x_{t})-g_{t}}\Delta - \norm{g_t}\Delta + 2\norm{B_t}\Delta^2 + \frac{1}{2}\brbra{L_{0}+L_{1}\norm{\nabla F(x_{t})}}\Delta^2 \\
            \leq & F(x_{t}) + 2\norm{\nabla F(x_{t})-g_{t}}\Delta - \norm{\nabla F(x_{t})}\Delta + 2\norm{B_t}\Delta^2 + \frac{1}{2}\brbra{L_{0}+L_{1}\norm{\nabla F(x_{t})}}\Delta^2.
        \end{split}
    \end{equation}
    where in $(\Diamond)$ we use the optimality condition \eqref{eq:optimal condition}, in $(\triangle)$ we follow the fact that 
$\lambda_t\norm{d_{t+1}}^2=\lambda_t\Delta^2$ according to the observation that either $\lambda_t=0$ or $\norm{d_{t+1}}=\Delta$ because of the optimality condition, in $(\vee)$ we use Lemma \ref{lem:psd}. 
    Therefore, we are able to bound $\Expect\norm{\nabla F(x_{t})}$ based on \eqref{eq:fo_descent} and Lemma \ref{lm sgd_var-bound for g-estimator}
    \begin{equation} \label{eq:gradient_bound_thm1}
        \begin{split}\Expect\norm{\nabla F(x_{t})}\leq & \frac{1}{\Delta}\Expect\bsbra{F(x_{t})-F(x_{t+1})}+2\Expect\norm{\nabla F(x_{t})-g_{t}}+2\norm{B_t}\Delta+\frac{1}{2}\brbra{L_{0}+L_{1}\Expect\norm{\nabla F(x_{t})}}\Delta\\
            \leq & \frac{1}{\Delta}\Expect\bsbra{F(x_{t})-F(x_{t+1})}+2\frac{1}{\sqrt{|\mcal S_{1}|}}\brbra{G_{0}+G_{1}\Expect\norm{\nabla F(x_{t})}}+2\beta\Delta+\frac{1}{2}\brbra{L_{0}+L_{1}\Expect\norm{\nabla F(x_{t})}}\Delta.
        \end{split}
    \end{equation}
    Because $\Delta=\frac{1}{4L_0+16\beta}\epsilon$, $|\mcal S_1|= \frac{64G_0^2}{\epsilon^{2}}$ and $\epsilon\leq \frac{4L_0G_0+16\beta G_0}{L_1G_0+2L_0G_1+8\beta G_1}$, we have $ 1-\frac{L_1}{8L_0+32\beta}\epsilon-\frac{G_1}{4G_0}\epsilon \geq \frac{1}{2} $ and
    \begin{equation}
        \begin{split}\left({1-\frac{L_1}{8L_0+32\beta}\epsilon-\frac{G_1}{4G_0}\epsilon}\right)\Expect\norm{\nabla F(x_{t})}\leq & \frac{4L_0+16\beta}{\epsilon}\Expect\bsbra{F(x_{t})-F(x_{t+1})}+\brbra{\frac{1}{4}\epsilon+\frac{1}{8}\epsilon}\\
            \Expect\norm{\nabla F(x_{t})}\leq & \frac{8L_0+32\beta}{\epsilon}\Expect\bsbra{F(x_{t})-F(x_{t+1})}+\frac{3}{4}\epsilon.
        \end{split}
        \label{eq:gradient_bound}
    \end{equation}
    By adding inequality \eqref{eq:gradient_bound} from 0 to $T-1$, we have
    \begin{equation}
        \begin{split}\frac{1}{T}\sum_{t=0}^{T-1}\Expect\norm{\nabla F(x_{t})} & \leq \frac{8L_0+32\beta}{T\epsilon}\Expect\bsbra{F(x_{0})-F(x_{T})}+\frac{3}{4}\epsilon\\
            & \leq \frac{8L_0+32\beta}{T\epsilon}\Delta_F+\frac{3}{4}\epsilon.
        \end{split}
    \end{equation}
    When taking $T = \frac{32\Delta_F (L_0+4\beta)}{\epsilon^{2}}$, we have $\frac{1}{T}\sum_{t=0}^{T-1}\Expect\norm{\nabla F(x_{t})} \leq \epsilon$. Then we finish the proof.
\end{proof}

\subsection{Proof of Corollary \ref{coro:nsgd}}
According to the optimality condition \eqref{eq:optimal condition}, we have
\begin{align*}
    \begin{split}
        \lambda_t d_{t+1}=-g_t, \lambda_t \norm{d_{t+1}} = \lambda_t\cdot\Delta = \norm{g_t}.
    \end{split}
\end{align*}
Therefore, we have
\begin{align*}
    \begin{split}
        \lambda_t = \frac{\norm{g_t}}{\Delta},
        d_{t+1} = \frac{\Delta}{\norm{g_t}}\cdot (-g_t).
    \end{split}
\end{align*}
Then we finish the proof for the equivalence between the normalized gradient descent and the first-order trust region method with $B_t=0$.

\subsection{Proof of Corollary \ref{coro:clip}}
Similar to the proof of Corollary \ref{coro:nsgd}, according to the optimality condition \eqref{eq:optimal condition}, we have
\begin{align*}
    \begin{split}
        d_{t+1} = \frac{1}{\rho + \lambda_t} \cdot (-g_t), \ \norm{d_{t+1}} = \frac{1}{\rho + \lambda_t}\norm{g_t}.
    \end{split}
\end{align*}
If $\rho < \frac{\norm{g_t}}{\Delta}$, then we have $\lambda_t = \frac{\norm{g_t}}{\norm{d_{t+1}}}-\rho\geq \frac{\norm{g_t}}{\Delta}-\rho> 0$, then we have
\begin{align*}
    \begin{split}
        \norm{d_{t+1}} &= \Delta = \frac{1}{\rho + \lambda_t}\norm{g_t}, \  \rho + \lambda_t = \frac{\norm{g_t}}{\Delta}, \\
        d_{t+1} &= \frac{\Delta}{\norm{g_t}} \cdot (-g_t).
    \end{split}
\end{align*}
If $\rho \geq \frac{\norm{g_t}}{\Delta}$, then we have $\norm{d_{t+1}} = \frac{1}{\rho + \lambda_t}\norm{g_t} \leq \frac{1}{\rho}\norm{g_t}\leq \Delta$, then we have $\lambda_t = 0$ and
\begin{align*}
    d_{t+1} = \frac{1}{\rho}\cdot(-g_t).
\end{align*}
Putting the two cases together gives
\begin{align*}
    d_{t+1} = \min\left\{\frac{\norm{g_t}}{\Delta}, \frac{1}{\rho}\right\} \cdot(-g_t).
\end{align*}
Thus, we complete the proof.

\section{Further Discussion of Second-Order Generalized Smoothness}\label{sec:SOGS}
As mentioned before, we can interpret the second-order generalized smoothness from the perspective of the boundness of higher-order derivatives. According to \cite{nesterov2018lectures}, we first give the following definition.
\begin{appxdef}
    For $F\in C^3(\mbb R^n)$, define 
    \begin{equation*}
        F'''(x)[u]=\lim_{\alpha\rightarrow0}\frac{1}{\alpha}\left[\nabla^2 F(x+\alpha u)-\nabla^2F(x)\right].
    \end{equation*}
\end{appxdef}
For any $F\in C^3(\mbb R^n)$, the Lipschitz continuity of hessian is equivalent to the following condition 
\begin{align*}
    \norm{F'''(x)[u]} \leq M\norm{u}, \ \forall u\in\mbb R^n.
\end{align*}
For second-order generalized smoothness, we can give a similar equivalent condition
\begin{align}\label{third order boundedness}
    \begin{split}
        \norm{F'''(x)[u]}\leq (M_0+M_1\norm{\nabla F(x)})\norm{u}.
    \end{split}
\end{align}
Next, we will show that condition \eqref{third order boundedness} and Assumption \ref{assm: hessian} is equivalent in a sense. Let Assumption \ref{assm: hessian} hold, for any $u\in\mbb R^n$, by the definition of $F'''(x)[u]$, we have
\begin{align*}
    \begin{split}
        \norm{F'''(x)[u]}&=\norm{\lim_{\alpha\rightarrow0}\frac{1}{\alpha}\left[\nabla^2 F(x+\alpha u)-\nabla^2F(x)\right]} \\
        &=\lim_{\alpha\rightarrow0}\frac{1}{\alpha}\norm{\nabla^2 F(x+\alpha u)-\nabla^2F(x)}\\
        &\leq (M_0+M_1\norm{\nabla F(x)})\norm{u}.
    \end{split}
\end{align*}
For the other direction, if $\norm{F'''(x)[u]}\leq (M_0+M_1\norm{\nabla F(x)})\norm{u}$ holds for any $u\in\mbb R^n$ and $F$ is $(L_0,L_1)$-smooth, then Assumption \ref{assm: hessian} holds locally.
\begin{appxlem}
    Assume that $F$ satisfies condition \eqref{third order boundedness} and $F$ is $(L_0,L_1)$-smooth. For any $c > 0$, if $\norm{x-y}\leq \tfrac{2c}{\max\{L_0,L_1\}}$, then
        $$\norm{\nabla^2 F(x)-\nabla^2 F(y)}\leq (A+B\norm{\nabla F(x)})\norm{x-y},$$
    where $A=M_0+cM_1$, $B=cM_1$.
\end{appxlem}
\begin{proof}
    Let $g(t)$ be defined as $g(t)=\nabla^2 F(x+t(y-x)), t\in[0,1]$, $g'(t)=F'''(x+t(y-x))[y-x]$. Then we have
    \begin{align*}
        \norm{\nabla^2 F(y) - \nabla^2 F(x)} &= \norm{g(1)-g(0)} \\
        &= \left\Vert\int_0^1 F'''(\gamma(x+t(y-x)))[y-x]dt\right\Vert \\
        &\leq \int_0^1 \norm{F'''(\gamma(x+t(y-x)))[y-x]}dt \\
        &\overset{(\triangle)}{\leq} \int_0^1 (M_0+M_1\norm{\nabla F(x+t(y-x))})\norm{y-x}dt \\
        &\overset{(\Diamond)}{\leq} M_0\norm{y-x}+M_1\norm{y-x}\int_0^1 (L_0+L_1\norm{\nabla F(x)})t\norm{y-x}dt \\
        &=M_0\norm{y-x}+\frac{1}{2}M_1\norm{y-x}(L_0+L_1\norm{\nabla F(x)}) \\
        &\leq M_0\norm{y-x}+\tfrac{c}{\max\{L_0,L_1\}}M_1\norm{y-x}(L_0+L_1\norm{\nabla F(x)}) \\
        &\leq (M_0+cM_1)\norm{y-x}+cM_1\norm{\nabla F(x)},
    \end{align*}
    where $(\triangle)$ uses the condition \eqref{third order boundedness}, $(\Diamond)$ uses the $(L_0,L_1)$ smoothness assumption. 
\end{proof}
By taking $c=\min\{2,\tfrac{M_0}{M_1}\}$, we have $A\leq 2M_0, B\leq 2M_1$. Therefore, Assumption \ref{assm: hessian} holds with $2M_0$ and $2M_1$ locally. 

\section{Further Discussion of DRO}\label{sec:moredro}
Firstly, we make some standard assumptions for divergence-based DRO. 
\begin{assumption}\label{assm:dro}
    We give the following assumption for DRO settings:
    \begin{itemize}
        \item Given any $\xi$ in its support set, the loss function $\ell(x, \xi)$ is twice differentiable, and it satisfies
              \begin{equation*}
                  \begin{aligned}
                      \|\ell(x,\xi)-\ell(x',\xi)\|                   & \leq G\|x-x'\|,  \\
                      \|\nabla \ell(x,\xi)-\nabla\ell(x',\xi)\|      & \leq L \|x-x'\|, \\
                      \|\nabla^2 \ell(x,\xi)-\nabla^2 \ell(x',\xi)\| & \leq M \|x-x'\|. \\
                  \end{aligned}
              \end{equation*}
        \item 
        The conjugate $\psi^*$ is strictly convex, twice continuously differentiable and it satisfies
              \begin{equation*}
                  \begin{aligned}
                      |(\psi^{*})'(x)-(\psi^{*})'(x')|         & \leq N_1 |x-x'|, \\
                      |(\psi^{*})''(x)-(\psi^{*})''(x')| & \leq N_2 |x-x'|. \\
                  \end{aligned}
              \end{equation*}
        \item  For all $x \in \mathbb{R}^n$, the stochastic loss has a bounded variance, namely,  $\mathbb{E}_{\xi \sim P}\bsbra{\brbra{\ell(x, \xi)-\ell(x)}^2} \leq \sigma^2$ where $\ell(x)=\mathbb{E}_{\xi \sim P}[\ell(x, \xi)]$. Moreover, there exist constants $m, M$ such that $m\leq \ell (x,\xi)\leq M$ for any $x$ and $\xi\sim P$.
        \item There exists $\delta>0$ such that $\inf\{t|(\psi^*)'(t)\geq 1-\delta\} > -\infty$, $\sup\{t|(\psi^*)'(t)\leq 1+\delta\}<+\infty$
    \end{itemize}
\end{assumption}
In many applications, the DRO problem %
may not exhibit Lipschitz smoothness due to its composition with a nonsmooth divergence function. A representative example illustrating this approach is provided in \citet[Example 3.1]{jin2021non}. To address this, the following subsection shows how we employ smooth approximation techniques to satisfy Assumption~\ref{assm:dro}.

\subsection{Smoothing approximation of valid divergence functions}\label{sec:divergence function}
Here we give some commonly used divergence functions:
\begin{table}[h]
    \centering
    \caption{Some commonly used divergences and the corresponding conjugates.}
    \begin{tabular}{ccc}
        \hline
        Divergence &
        $\psi(t)$  & $\psi^*(t)$                                                                                   \\
        \hline
        \(\chi^2\) & \((t-1)^2\)                                                   & \(-1+\frac{1}{4}(t+2)_{+}^2\) \\

        K-L        & \(t \log t-t+1\)                                              & \(e^t-1\)                     \\

        CVaR       & \(\mathbb{I}_{\left[0, \alpha^{-1}\right)}(t), \alpha \in(0,1)\) & \(\alpha^{-1}(t)_{+}\)        \\
        \hline
    \end{tabular}
\end{table}

In the case of $\mathrm{CVaR}, \psi^*$ is not differentiable as shown in Table 2, which is undesirable from an optimization viewpoint. Following \citet{jin2021non}, we introduce a smoothed version of CVaR. The conjugate function of the smoothed CVaR is also smooth so that our results can be directly applied in smoothed CVAR.

For standard CVaR at level $\alpha, \psi_\alpha(t)$ takes zero when $t \in[0,1 / \alpha)$ and takes infinity otherwise. Instead, we consider the following smoothed version of CVaR:
$$
    \psi_\alpha^{\mathrm{smo}}(t)= \begin{cases}t \log t+\frac{1-\alpha t}{\alpha} \log \frac{1-\alpha t}{1-\alpha} & t \in[0,1 / \alpha) \\ +\infty & \text { otherwise }\end{cases}
$$
It is easy to see that $\psi_\alpha^{\text {smo }}$ is a valid divergence. The corresponding conjugate function is
$$
    \psi_\alpha^{\mathrm{smo}, *}(t)=\frac{1}{\alpha} \log (1-\alpha+\alpha \exp (t))
$$
In the case of  $\chi^2$-divergence, $\psi^{*}$ is not second-order continuous, which also violates our Assumption \ref{assm:dro}. Therefore, we introduce a smoothed version of $\chi^2$-divergence:
$$
    \psi_{\chi^2}^{\mathrm{smo}}(t)= \begin{cases}(t-1)^2 & t\geq1\\2(t\log t-t+1)&t \in[0,1) \\ +\infty & \text { otherwise }\end{cases}
$$
It is easy to see that $\psi_\alpha^{\text {smo }}$ is a valid divergence. The corresponding conjugate function is
$$
    \psi_{\xi^2}^{\mathrm{smo}, *}(t)=\begin{cases}-1+\frac{1}{4}(t+2)^2 & t\geq0\\2(\exp(\frac{t}{2})-1)&t <0  \end{cases}
$$
For a broader context on smooth approximation in nonconvex optimization, we refer the reader to \cite{nazarethHomotopyTechniquesLinear1986,chenSuperlinearConvergenceSmoothing1997,chenSmoothingMethodsNonsmooth2012}
\subsection{Proof of Theorem \ref{lm:connection2}}
Note that the objective is $ \Psi(x)=\min _{\eta \in \mathbb{R}} \mathcal{L}(x, \eta) := \lambda \mathbb{E}_{\xi \sim P} \psi^*\left(\frac{\ell(x ; \xi)-\eta}{\lambda}\right)+\eta$. Now we want to build the connections between the properties of $\Psi(x)$ and $\mathcal{L}(x,\eta)$, and then prove Theorem \ref{thm:dro}.
\begin{appxlem}\label{lm:connection1}
    Under the Assumption \ref{assm:dro}, $\arg \min _{\eta^{\prime}} \mathcal{L}\left(x, \eta^{\prime}\right)$ is a singleton set. $\Psi(x)$ is twice-differentiable, $\nabla \Psi(x)=\nabla_x \mathcal{L}(x, \eta_x^{*})$,  and $\nabla^2 \Psi(x)= \nabla_x^2 \mathcal{L}(x,\eta_x^*)+\frac{\lambda}{\mathbb{E}_\xi\left[(\psi^*)''(\frac{\ell(x,\xi)-\eta_x^*}{\lambda})\right]}\nabla_{x\eta}\mathcal{L}(x,\eta_x^*)\nabla_{x\eta}\mathcal{L}(x,\eta_x^*)^{\top}$ for $\eta_x^{*} = \arg \min _{\eta^{\prime}} \mathcal{L}\left(x, \eta^{\prime}\right)$.
\end{appxlem}
\begin{proof}
    Firstly, since $\psi^*$ is strictly convex, we have $\nabla_\eta^2 \mathcal{L}(x,\eta)$ >0, which means the $\arg \min _{\eta^{\prime}} \mathcal{L}\left(x, \eta^{\prime}\right)$ is a singleton set. Then from Lemma 2.6 in \citet{jin2021non}, $\nabla \Psi(x)=\nabla_x \mathcal{L}(x, \eta_x^{*})$. 
    We first use the implicit function theorem to prove that $\eta_x^*$ is actually a continuously differentiable function of $x$. Define $f(a,b):=\Expect_\xi\left[(\psi^*)'(\frac{l(a,\xi)-b}{\lambda})\right]-1$, $a\in\mathbb{R}^n, b\in \mathbb{R}$. By the optimality condition, we have   
        $\nabla_\eta \mathcal{L}(x,\eta_x^*)=0$.
   In other words,  we have 
    \begin{equation*}
        f(x,\eta_x^*)=0.
    \end{equation*}
    Note that by the strict convexity of $\psi^*$, we have $\nabla_b f(x,\eta_x^*)=-\frac{1}{\lambda}\mathbb{E}_\xi\left[(\psi^*)''(\frac{\ell(x,\xi)-\eta_x^*}{\lambda})\right]\neq 0$. Hence, by the implicit function theorem, for any $ x$, there exists an open set $U_x$ containing $x$ such that there exists a unique continuously differentiable function $g: U\rightarrow \mathbb{R}$ such that $g(x)=\eta_x^*$. Moreover, we have
    \begin{equation*}
    \begin{aligned}
          \nabla g(x)&=\left(\nabla_b f(x,\eta_x^*)\right)^{-1}\cdot \nabla_a f(x,\eta_x^*)\\
          &= -\frac{\Expect_\xi\left[\nabla \ell(x,\xi)(\psi^*)''(\frac{\ell(x,\xi)-\eta_x^*}{\lambda})\right]}{\mathbb{E}_\xi\left[(\psi^*)''(\frac{\ell(x,\xi)-\eta_x^*}{\lambda})\right]} \\
          &=\frac{\lambda\nabla_{x\eta}\mathcal{L}(x,\eta_x^*)}{\mathbb{E}_\xi\left[(\psi^*)''(\frac{\ell(x,\xi)-\eta_x^*}{\lambda})\right]},
    \end{aligned}
    \end{equation*}
    according to the chain rule, we have 
    \begin{align*}
        \nabla^2 \Psi(x) &=\nabla_x^2 \mathcal{L}(x,\eta_x^*)+\nabla_{x\eta}\mathcal{L}(x,\eta_x^*)\nabla g(x)^{\top} \\
        &= \nabla_x^2 \mathcal{L}(x,\eta_x^*)+\frac{\lambda}{\mathbb{E}_\xi\left[(\psi^*)''(\frac{\ell(x,\xi)-\eta_x^*}{\lambda})\right]}\nabla_{x\eta}\mathcal{L}(x,\eta_x^*)\nabla_{x\eta}\mathcal{L}(x,\eta_x^*)^{\top} \\
        &\succeq \nabla_x^2 \mathcal{L}(x,\eta_x^*).
    \end{align*}
\end{proof}

We first give some additional assumptions on the loss function $\psi^*$.
    \begin{assumption}\label{assm:bounded loss}
    There exist constants $m, M$ such that for any $\epsilon$-FOSP $(x,\eta)$, $m\leq \ell (x,\xi)\leq M$ a.e..
    \end{assumption}
    
    \begin{assumption}\label{assm:domain boundedness}
        There exists $\delta>0$ such that $\inf\{t|(\psi^*)'(t)\geq 1-\delta\} > -\infty$, $\sup\{t|(\psi^*)'(t)\leq 1+\delta\}<+\infty$
    \end{assumption}

    \begin{appxlem}\label{lem:eta boundedness}
    Given $\epsilon$ small enough, consider any $(x,\eta)$ such that $\|\nabla \mcal L(x,\eta)\|\leq \epsilon$. Let $\underline{t} = \inf\{t|(\psi^*)'(t)\geq 1-\epsilon\}$, $\overline{t} = \sup\{t|(\psi^*)'(t)\leq 1+\epsilon\}$. Under Assumptions~\ref{assm:dro}, \ref{assm:bounded loss} and~\ref{assm:domain boundedness}, we have $\eta\in [m - \lambda \overline{t}, M - \lambda \underline{t}]$.
\end{appxlem}
\begin{proof}
    Because $\|\nabla \mcal L(x,\eta)\|\leq \epsilon$, we have 
    \begin{align*}
        1-\epsilon \leq \mbb E_{\xi} \left[ (\psi^*)'\left(\frac{\ell(x;\xi)-\eta}{\lambda}\right) \right] \leq 1+\epsilon.
    \end{align*}
    Let event $A=\{\xi|\ell (x,\xi)< m\}\cup \{\xi|\ell (x,\xi)> M\}$. According to Assumption~\ref{assm:bounded loss}, we have $P(A)=0$. Then we have
    \begin{align*}
        1-\epsilon \leq \mbb E_{\xi} \left[ (\psi^*)'\left(\frac{\ell(x;\xi)-\eta}{\lambda}\right)\big|A\right]=E_{\xi} \left[ (\psi^*)'\left(\frac{\ell(x;\xi)-\eta}{\lambda}\right)\right] \leq 1+\epsilon.
    \end{align*}
    There exists $\xi_1,\xi_2\in \overline{A}$ such that 
    \begin{align*}
        (\psi^*)'\left(\frac{\ell(x;\xi_1)-\eta}{\lambda}\right) &\geq 1-\epsilon, \\
        (\psi^*)'\left(\frac{\ell(x;\xi_2)-\eta}{\lambda}\right) &\leq 1+\epsilon.
    \end{align*}
    Note that $(\psi^*)'$ is monotonically increasing because of the convexity of $\psi^*$, we have
    \begin{align*}
        \frac{\ell(x;\xi_1)-\eta}{\lambda} \geq \underline{t}, \quad
        \frac{\ell(x;\xi_2)-\eta}{\lambda} \leq \overline{t}.
    \end{align*}
    Therefore, we can bound $\eta$ by the following inequality
    \begin{align*}
        \begin{split}
            \eta &\leq \ell(x;\xi_1) - \lambda \underline{t} \leq M - \lambda \underline{t}, \\
            \eta &\geq \ell(x;\xi_2) - \lambda \overline{t} \geq m - \lambda \overline{t}.
        \end{split}
    \end{align*}    
\end{proof}

Now we are ready to prove of Theorem \ref{lm:connection2}.
\begin{proof}[Proof of Theorem \ref{lm:connection2}]
    Suppose that we have obtained an $(\epsilon,\sqrt{\epsilon}) $-SOSP pair $(x, \eta)$ such that
    \begin{equation*}
        \|\nabla \mathcal{L}(x,\eta)\|\leq c_1 \epsilon, \; \lambda_{\min}(\nabla^2 \mathcal{L}(x,\eta))\geq-c_2\sqrt{\epsilon}.
    \end{equation*}
    Let $x$ be fixed and $\eta^* \in \arg \min _\eta \mathcal{L}(x, \eta)$. Then we have
    \begin{equation}\label{eq: FO_dro_tui}
        \begin{aligned}
                 & \left\|\nabla_x \mathcal{L}(x, \eta)-\nabla_x \mathcal{L}\left(x, \eta^*\right)\right\|                            \\
              & =  \left\| \mathbb{E}_{\xi}\left[\left((\psi^{*})'\left(\frac{\ell(x ; \xi)-\eta}{\lambda}\right)-(\psi^{*})'\left(\frac{\ell(x ; \xi)-\eta^*}{\lambda}\right)\right) \nabla \ell(x ; \xi)\right]\right\| \\
           &  \leq G \cdot \mathbb{E}_{\xi}\left|(\psi^{*})'\left(\frac{\ell(x ; \xi)-\eta}{\lambda}\right)-(\psi^{*})'\left(\frac{\ell(x ; \xi)-\eta^*}{\lambda}\right)\right|                                           \\
             &  =   G \cdot\left|\mathbb{E}_{\xi}\left[(\psi^{*})'\left(\frac{\ell(x ; \xi)-\eta}{\lambda}\right)-(\psi^{*})'\left(\frac{\ell(x ; \xi)-\eta^*}{\lambda}\right)\right]\right|                               \\
           &  =    G\left|\nabla_\eta \mathcal{L}(x, \eta)-\nabla_\eta \mathcal{L}\left(x, \eta^*\right)\right|=G\left|\nabla_\eta \mathcal{L}(x, \eta)\right|,
        \end{aligned}
    \end{equation}
    where we use the fact that $(\psi^{*})'$ is monotonely increasing (due to the convexity of $\psi^*$). Hence, using Lemma~\ref{lm:connection1} we obtain
    $$
        \|\nabla \Psi(x)\|=\left\|\nabla_x \mathcal{L}\left(x, \eta^*\right)\right\| \leq\left\|\nabla_x \mathcal{L}(x, \eta)\right\|+G\left|\nabla_\eta \mathcal{L}(x, \eta)\right| \leq \sqrt{G^2+1}c_1\epsilon\leq c_1'\epsilon.
    $$
    Because of the boundedness of $\eta$ given in lemma \ref{lem:eta boundedness}, $(\psi^*)''$ is lower bounded by some positive constant and therefore $\psi^*$ is strongly convex on this internal. So there exists $\mu >0$ such that 
    \begin{align*}
        \psi^*(y)-\psi^*(x)\geq(\psi^{*})'(x)\cdot(y-x)+\frac{\mu}{2}|y-x|^2.
    \end{align*}
    Therefore, we have
    \begin{equation*}
        \begin{aligned}
             & \|\nabla^2_x \mathcal{L}\left(x, \eta_x^{*}\right)-\nabla^2_x \mathcal{L}\left(x,
            \eta\right)\|   \\
            &=\Big\|\mathbb{E}_{\xi}\left[\left((\psi^*)''\left(\frac{\ell(x ; \xi)-\eta_x^{*}}{\lambda}\right)-(\psi^*)''\left(\frac{\ell(x ; \xi)-\eta}{\lambda}\right)\right) \frac{\nabla \ell(x ; \xi) \nabla \ell(x ; \xi)^T}{\lambda}\right] \\
             &\quad + \mathbb{E}_{\xi}\left[\left((\psi^*)'\left(\frac{\ell(x ; \xi)-\eta_x^{*}}{\lambda}\right)-(\psi^*)'\left(\frac{\ell(x ; \xi)-\eta}{\lambda}\right)\right) \nabla^2 \ell(x ; \xi)\right]\Big\| \\
            & \leq \frac{G^2}{\lambda}\mathbb{E}_{\xi}\left|\left((\psi^*)''\left(\frac{\ell(x ; \xi)-\eta}{\lambda}\right)-(\psi^*)''\left(\frac{\ell(x ; \xi)-\eta^*}{\lambda}\right)\right)\right|+L|\nabla_\eta\mathcal{L}(x,\eta)|   \\
             & \stackrel{(\Diamond)}{\leq}  \frac{G^2}{\lambda^2}N_2\mathbb{E}_{\xi}[|\eta-\eta_x^{*}|] +Lc_1\epsilon\\
              &\stackrel{(\triangle)}{\leq}  \frac{G^2}{\lambda^2}N_2\mathbb{E}_{\xi}\left[\frac{\lambda}{\mu}\left|(\psi^*)'\left(\frac{\ell(x ; \xi)-\eta}{\lambda}\right)-(\psi^*)'\left(\frac{\ell(x ; \xi)-\eta^*}{\lambda}\right)\right|\right]+Lc_1\epsilon \\
             & \stackrel{(\vee)}{\leq}  \frac{G^2}{\lambda\mu}N_2|\nabla_{\eta}\mathcal{L}(x,\eta)|+Lc_1\epsilon\\
              &\leq  (\frac{G^2}{\lambda\mu}N_2+L)c_1\epsilon,
        \end{aligned}
    \end{equation*}
    where in $(\Diamond)$ we use Assumption \ref{assm:dro}, in $(\triangle)$ we use the property of strong convexity and in $(\vee)$ we follow the derivation in \eqref{eq: FO_dro_tui}.
    Hence, using Lemma \ref{lm:connection1} we obtain
    $$
        \begin{aligned}
            \lambda_{\min}(\nabla^2 \Psi(x))&=\lambda_{\min}\left(\nabla_x^2 \mathcal{L}(x,\eta_x^*)+\frac{\lambda}{\mathbb{E}_\xi\left[(\psi^*)''(\frac{\ell(x,\xi)-\eta_x^*}{\lambda})\right]}\nabla_{x\eta}\mathcal{L}(x,\eta_x^*)\nabla_{x\eta}\mathcal{L}(x,\eta_x^*)^{\top}\right)\\
            &\geq \lambda_{\min}\left(\nabla_x^2 \mathcal{L}(x,\eta_x^*)\right)\\
             & \geq \lambda_{\min}(\nabla^2_x \mathcal{L}(x, \eta))-\|\nabla^2_x \mathcal{L}\left(x, \eta_x^{*}\right)-\nabla^2_x \mathcal{L}\left(x,
            \eta\right)\|  \\                     
             & \geq\lambda_{\min}(\nabla_x^2\mathcal{L}(x,\eta))-\left(\frac{G^2}{\lambda\mu}N_2+L\right)c_1\epsilon                                  \\
             & \geq-c_2'\sqrt{\epsilon}.
        \end{aligned}
    $$
\end{proof}
\subsection{Proof of Theorem \ref{thm:dro}}
Finally, we want to prove that the penalized version of DRO satisfies our first-order generalized smoothness and second-order generalized smoothness.
\begin{proof}
    From Lemma 3.3 and 3.4 of \cite{jin2021non}, Assumptions \ref{assm L0L1} and \ref{assm gradient variance} naturally hold for $\mathcal{L}(x,\eta)$. Therefore, We only prove Assumption \ref{assm: hessian}. Let $\left(\nabla \ell(x;\xi)\right)^2$ denote $\nabla \ell(x,\xi)\nabla \ell(x,\xi)^T$. We know that
    \begin{equation*}
        \nabla^2 \mathcal{L}(x,\eta)=\begin{bmatrix}
            A_1 & A_2  \\
            A_3 & A_4, \\
        \end{bmatrix},
    \end{equation*}
    where $A_1=\Expect_\xi\left[\frac{1}{\lambda}(\psi^{*})''\left(\frac{\ell(x,\xi)-\eta}{\lambda}\right)\left(\nabla \ell(x;\xi)\right)^2+(\psi^{*})'\left(\frac{\ell(x,\xi)-\eta}{\lambda}\right)\nabla^2\ell(x,\xi)\right] $, $A_4=\Expect_{\xi}\left[\frac{1}{\lambda}(\psi^{*})''\left(\frac{\ell(x,\xi)-\eta}{\lambda}\right)\right]$, $A_2=A_3=\Expect_{\xi}\left[-\frac{1}{\lambda}(\psi^{*})''\left(\frac{\ell(x,\xi)-\eta}{\lambda}\right)\nabla \ell(x,\xi)\right]$.
    Hence, we have
    \begin{equation*}
        \begin{aligned}
            \|A_1(x,\eta)-A_1(x',\eta')\| &
            \leq \Expect\Big[\norm{\frac{1}{\lambda}(\psi^{*})''\left(\frac{\ell(x,\xi)-\eta}{\lambda}\right)\brbra{(\nabla \ell(x,\xi))^2-(\nabla \ell(x',\xi))^2}}                                                \\
          & +\|\frac{1}{\lambda}(\nabla \ell(x',\xi))^2[(\psi^{*})''\left(\frac{\ell(x,\xi)-\eta}{\lambda}\right)-(\psi^{*})''\left(\frac{\ell(x',\xi)-\eta'}{\lambda}\right)]\|        \\
          & +\|[(\psi^{*})'\left(\frac{\ell(x,\xi)-\eta}{\lambda}\right)-(\psi^{*})'\left(\frac{\ell(x',\xi)-\eta'}{\lambda}\right)]\nabla^2\ell(x,\xi)\|                             \\&+\|(\psi^{*})'\left(\frac{\ell(x',\xi)-\eta'}{\lambda}\right)(\nabla^2\ell(x,\xi)-\nabla^2\ell (x',\xi))\|\Big]\\
          & \leq \Expect[\frac{N_1}{\lambda}\cdot 2GL\|x-x'\|+\frac{G^2N_2}{\lambda^2}\|\ell(x,\xi)-\ell(x',\xi)-(\eta-\eta')\|                           \\
          & +\frac{1}{\lambda}LN_1\|\ell(x,\xi)-\ell(x',\xi)-(\eta-\eta')\|+M\cdot|(\psi^{*})'\left(\frac{\ell(x,\xi)-\eta}{\lambda}\right)|\cdot\|x-x'\|]                               \\
          & \leq \left({\frac{2}{\lambda}N_1GL+\frac{G^2N_2}{\lambda^2}\sqrt{G^2+1}+\frac{1}{\lambda}LN_1\sqrt{G^2+1}+M+M\norm{\nabla\mcal L(x,\eta)}}\right)
          \|(x,\eta)-(x',\eta')\|,
        \end{aligned}
    \end{equation*}
    where in the last inequality we use the fact that $\nabla_\eta\mathcal{L}(x,\eta)=\Expect_\xi\left[1-(\psi^{*})'\left(\frac{\ell(x,\xi)-\eta}{\lambda}\right)\right]$ and $|\nabla_{\eta}\mcal L(x,\eta)|\leq\norm{\nabla \mcal L(x,\eta)}$. Similarly, we can get
    \begin{equation*}
        \begin{aligned}
            \|A_2(x,\eta)-A_2(x',\eta')\|=\|A_3(x,\eta)-A_3(x',\eta)\| & \leq \left(\frac{1}{\lambda}LN_1+\frac{1}{\lambda^2}GN_2\sqrt{G^2+1}\right)\|(x,\eta)-(x',\eta')\| \\
            \|A_4(x,\eta)-A_4(x',\eta')\|                                  & \leq\frac{1}{\lambda^2}N_2\sqrt{G^2+1}\|(x,\eta)-(x',\eta')\|.
        \end{aligned}
    \end{equation*}
    Since $\|\nabla^2 \mathcal{L}(x,\eta)\|\leq \|A_1\|+\|A_2\|+\|A_3\|+\|A_4\|$, we have
    \begin{equation*}
        \begin{aligned}
            \|\nabla^2 \mathcal{L}(x,\eta)\| \leq \left[\frac{1}{\lambda^2}N_2(G+1)^2\sqrt{G^2+1}+\frac{2}{\lambda}LN_1(G+1)+\frac{1}{\lambda}LN_1\sqrt{G^2+1}+M\brbra{1+\norm{\nabla \mathcal{L}(x,\eta)}}\right]\|(x,\eta)-(x',\eta')\|.
        \end{aligned}
    \end{equation*}
    Hence, the first part of  Assumption~\ref{assm: hessian} is satisfied with $M_0=\frac{1}{\lambda^2}N_2(G+1)^2\sqrt{G^2+1}+\frac{2}{\lambda}LN_1(G+1)+\frac{1}{\lambda}LN_1\sqrt{G^2+1}+M$, $M_1=M$.

    For the second part of Assumption \ref{assm: hessian}, let  $A$ be a  random matrix. Then  Jensen's inequality and convexity imply that
    $$
        \mathbb{V}[A]=\Expect[\|A-\Expect[A]\|^2]\leq \Expect_{A,A^{'}}[\|A-A^{'}\|^2]=2\mathbb{V}[A],$$
    where $A^{\prime}$ is an i.i.d. copy of $A$. Denote $A_i(\xi_j)=A_i(x,\eta,\xi_j)$. It follows that
    \begin{equation*}
        \begin{aligned}
            \mathbb{V}[A_1(\xi)]
             & \leq\Expect[\|A_1(\xi_1)-A_1(\xi_2)\|^2] \\
             & \leq 4\Expect\Big[\biggl\|\frac{1}{\lambda}(\psi^{*})''\left(\frac{\ell(x,\xi_1)-\eta}{\lambda}\right)[(\nabla \ell(x,\xi_1))^2-(\nabla \ell(x,\xi_2))^2]\biggl\|^2            \\
             & +4\biggl\|\frac{1}{\lambda}(\nabla \ell(x,\xi_2))^2\left[(\psi^{*})''\left(\frac{\ell(x,\xi_1)-\eta}{\lambda}\right)-(\psi^{*})''\left( \frac{\ell(x,\xi_2)-\eta}{\lambda}\right)\right]\biggl\|^2 \\
             & +4\biggl\|[(\psi^{*})'\left(\frac{\ell(x,\xi_1)-\eta}{\lambda}\right)-(\psi^{*})'\left(\frac{\ell(x,\xi_2)-\eta}{\lambda}\right)]\nabla^2\ell(x,\xi_1)\biggl\|^2                           \\
             & +4\biggl\|(\psi^{*})'\left(\frac{\ell(x,\xi_2)-\eta}{\lambda}\right)(\nabla^2\ell(x,\xi_1)-\nabla^2\ell(x,\xi_2))\biggl\|^2\Big]                                               \\
             & \leq \frac{16}{\lambda^2}G^4N_1^2+\frac{4}{\lambda^4}G^4N_2^2\Expect[(\ell(x,\xi_1)-\ell(x,\xi_2))^2]                                         \\
             & +\frac{4}{\lambda^2}L^2N_1\Expect[\ell(x,\xi_1)-\ell(x,\xi_2)^2]+16L^2\Expect\left[\biggl\|(\psi^{*})'\left(\frac{\ell(x,\xi_1)-\eta}{\lambda}\right)\biggl\|^2\right]                                  \\
             & \leq \frac{16}{\lambda^2}G^4N_1^2+\frac{8}{\lambda^4}G^4N_2^2\sigma^2+\frac{8}{\lambda^2}L^2N_1\sigma^2+16L^2\Expect\left[\biggl\|(\psi^{*})'\left(\frac{\ell(x,\xi_1)-\eta}{\lambda}\right)\biggl\|^2\right]
        \end{aligned}
    \end{equation*}
    Now, we deal with the first term. Using $2(a-1)^2+2 \geq a^2$ for any $a$, we have
    $$
        \begin{aligned}
            \mathbb{E}_{\xi}\left[\left((\psi^{*})'\left(\frac{\ell(x ; \xi)- \eta}{\lambda}\right)\right)^2\right] & \leq 2+2 \mathbb{E}_{\xi}\left[\left(1-(\psi^{*})'\left(\frac{\ell(x ; \xi)- \eta}{\lambda}\right)\right)^2\right]                    \\
            & = 2\left(1+\left\|\nabla_\eta \mathcal{L}(x, \eta)\right\|^2+ \mathbb{V}\left[\nabla_\eta \mathcal{L}(x, \eta ; \xi)\right]\right).
        \end{aligned}
    $$
    For $\mathbb{V}[\nabla_\eta \mathcal{L}(x, \eta ; \xi)]$, we have
    \begin{align*}
        \mathbb{V}[\nabla_\eta \mathcal{L}(x, \eta ; \xi)] & \leq\Expect[\|\nabla_\eta \mathcal{L}(x, \eta ; \xi_1)-\nabla_\eta \mathcal{L}(x, \eta ; \xi_2)\|^2]     \\
        &= \mathbb{E}\left[\left|(\psi^{*})'\left(\frac{\ell(x ; \xi_1)- \eta}{\lambda}\right)-(\psi^{*})'\left(\frac{\ell(x ; \xi_2)- \eta}{\lambda}\right)\right|^2\right]  \\
        &\leq \frac{N_1^2}{\lambda^2} \mathbb{E} \norm{\ell(x ; \xi_1)-\ell(x ; \xi_2)}^2 \\
        &\leq \frac{2}{\lambda^2}N_1^2\sigma^2.
    \end{align*}
    Therefore, $ \mathbb{V}[A_1(\xi)]$ can be bounded by constants and $\|\nabla \mathcal{L}(x,\eta)\|$. It follows that
    \begin{align*}
        \mathbb{V}[A_1(\xi)]&\leq\frac{16}{\lambda^2}G^4N_1^2+\frac{8}{\lambda^4}G^4N_2^2\sigma^2+\frac{8}{\lambda^2}L^2N_1\sigma^2+16L^2\Expect\left[\biggl\|(\psi^{*})'\left(\frac{\ell(x,\xi_1)-\eta}{\lambda}\right)\biggl\|^2\right] \\
        &\leq\frac{16}{\lambda^2}G^4N_1^2+\frac{8}{\lambda^4}G^4N_2^2\sigma^2+\frac{8}{\lambda^2}L^2N_1\sigma^2+32L^2\brbra{1+\left\|\nabla_\eta \mathcal{L}(x, \eta)\right\|^2+ \mathbb{V}\left[\nabla_\eta \mathcal{L}(x, \eta ; \xi)\right]} \\
        &\leq\frac{16}{\lambda^2}G^4N_1^2+\frac{8}{\lambda^4}G^4N_2^2\sigma^2+\frac{8}{\lambda^2}L^2N_1\sigma^2+32L^2+\frac{64}{\lambda^2}L^2N_1^2\sigma^2+32L^2\left\|\nabla \mathcal{L}(x, \eta)\right\|^2.
    \end{align*}
    Using a similar argument, we can show
    \begin{equation*}
        \begin{aligned}
            & \mathbb{V}[A_2(\xi)] =\mathbb{V}[A_3(\xi)]  \leq \frac{8}{\lambda^2}G^2N_1^2+ \frac{4}{\lambda^4}G^2N_2^2\sigma^2,\\
            & \mathbb{V}[A_4(\xi)]                        \leq \frac{2}{\lambda^4}N_2^2\sigma^2.
        \end{aligned}
    \end{equation*}
    Furthermore, according to Cauchy-Schwarz inequality, we have \[\mathbb{V}[\nabla^2 \mathcal{L}(x,\eta)\|\leq \mathbb{E}[(\|A_1-\mathbb{E}[A_1]\|+\|A_2-\mathbb{E}[A_2]\|+\|A_3-\mathbb{E}[A_3]\|+\|A_4-\mathbb{E}[A_4]\|)^2]\leq 4\sum_{i=1}^4  \mathbb{V}[A_i(\xi)].\] 
    Combining the above properties, we obtain the desired conclusions.
\end{proof}

\section{Proof of Second-Order Methods}\label{Sec: Appx_SO}
Analogous to standard results of second-order Lipschitz smooth functions \cite{nesterovLecturesConvexOptimization2018}, we develop some useful properties for the second-order generalized smoothness.

\begin{appxlem}\label{lm:coomon descent2}
    Under Assumption \ref{assm: hessian}, we have
    \begin{equation*}
        \begin{aligned}
                 & \|\nabla F(x')-\nabla F(x)-\nabla^{2} F(x)(x'-x)\| \leq \frac{M_0+M_1\|\nabla F(x)\|}{2} \|x'-x\|^2.
        \end{aligned}
    \end{equation*}
\end{appxlem}

\begin{appxlem}\label{lm:coomon descent1}
    Under Assumption \ref{assm: hessian}, we have
    \begin{equation*}
        \begin{aligned}
                 & F(x')-F(x) \leq \nabla F(x)^{T}(x'-x)+\frac{1}{2}(x'-x)^{T}\nabla^{2}F(x)(x'-x)+\frac{1}{6}\brbra{M_{0}+M_{1}\norm{\nabla F(x)}}\norm{x'-x}^{3}.
        \end{aligned}
    \end{equation*}
\end{appxlem}

\subsection{Proof of Lemma \ref{lm:coomon descent2} and Lemma \ref{lm:coomon descent1}}
\begin{proof}
    We have
    \begin{equation*}
        \begin{aligned}
        y^T \brbra{\nabla F(x')-\nabla F(x) - \nabla^2 F(x)(x'-x)} =& y^T\int_0^1 \brbra{\nabla^2 F(x+k(x'-x)) - \nabla^2 F(x)}(x'-x) \dtxt k \\
        \leq & \norm{y}\int_0^1\norm{\nabla^2 F(x+k(x'-x)) - \nabla^2 F(x)}\norm{x'-x}\dtxt k \\
        \leq & \norm{y}\frac{M_0+M_1\norm{\nabla F(x)}}{2}\norm{x'-x}^2
        \end{aligned}
    \end{equation*}
    and the result follows by Cauchy-Schwarz inequality.

    We prove the second equation analogously
    \begin{equation*}
        \begin{split}
        & F(x')-F(x) \\
        & =\int_{0}^{1}\nabla F(x+t(x'-x))^{T}(x'-x)\dtxt t\\
            & =\nabla F(x)^{T}(x'-x)+\int_{0}^{1}\brbra{\nabla F(x+t(x'-x))^{T}(x'-x)-\nabla F(x)^{T}(x'-x)}\dtxt t\\
            & =\nabla F(x)^{T}(x'-x)+\int_{0}^{1}\int_{0}^{1}t(x'-x)^{T}\nabla^{2}F(x+tk(x'-x))^{T}(x'-x)dkdt\\
            & =\nabla F(x)^{T}(x'-x)+\frac{1}{2}(x'-x)^{T}\nabla^{2}F(x)(x'-x)\\
            & \quad +\int_{0}^{1}\int_{0}^{1}t(x'-x)^{T}\brbra{\nabla^{2}F(x+tk(x'-x))-\nabla^{2}F(x)}^{T}(x'-x)\dtxt k\dtxt t\\
            & \leq\nabla F(x)^{T}(x'-x)+\frac{1}{2}(x'-x)^{T}\nabla^{2}F(x)(x'-x)+\int_{0}^{1}\int_{0}^{1}t\norm{x'-x}^{2}\norm{\nabla^{2}F(x+tk(x'-x))-\nabla^{2}F(x)}\dtxt k \dtxt t\\
            & \leq\nabla F(x)^{T}(x'-x)+\frac{1}{2}(x'-x)^{T}\nabla^{2}F(x)(x'-x)+\int_{0}^{1}\int_{0}^{1}t^{2}k\norm{x'-x}^{3}\brbra{M_{0}+M_{1}\norm{\nabla F(x)}}\dtxt k\dtxt t\\
            & =\nabla F(x)^{T}(x'-x)+\frac{1}{2}(x'-x)^{T}\nabla^{2}F(x)(x'-x)+\frac{1}{6}\brbra{M_{0}+M_{1}\norm{\nabla F(x)}}\norm{x'-x}^{3}.
        \end{split}
    \end{equation*}
\end{proof}
\subsection{Proof of Lemma \ref{lm var for hessian}}
\begin{proof}

    As for the variance bound of Hessian,
    it is a direct result of the following auxiliary lemma (whose proof can be found in \citet{arjevani2020second}. We derive it here just for completeness) by setting $A_{i}=\nabla^2 f(x_t;\xi_i), B=\nabla^2 F(x_t)$,$\sigma^2=K_0^2+K_1^2\|\nabla F(x_t)\|^2$,$m=|\mathcal{S}_2|$.

    \begin{appxlem}
        Let $\left(A_i\right)_{i=1}^m$ be a collection of i.i.d. sampled matrices in $\mathbb{S}^n$, with $\Expect\left[A_i\right]=B$ and $\Expect\left\|A_i-B\right\|^2 \leq$ $\sigma^2$. Then it holds that
        $$
            \Expect\left\|\frac{1}{m} \sum_{i=1}^m A_i-B\right\|^2 \leq \frac{22 \sigma^2 \log n}{m}.
        $$
    \end{appxlem}
    \textbf{Proof:} We drop the normalization by $m$ throughout this proof. We first symmetrize. Observe that by Jensen's inequality we have
    $$
        \begin{aligned}
            \Expect\left\|\sum_{i=1}^m A_i-mB\right\|^2 & \leq \Expect_A \Expect_{A^{\prime}}\left\|\sum_{i=1}^m A_i-A_i^{\prime}\right\|^2                          \\& =\Expect_A \Expect_{A^{\prime}}\left\|\sum_{i=1}^m\left(A_i-B\right)-\left(A_i^{\prime}-B\right)\right\|^2 \\
            & =\Expect_A \Expect_A \Expect_{\epsilon}\left\|\sum_{i=1}^m \epsilon_i\left(\left(A_i-B\right)-\left(A_i^{\prime}-B\right)\right)\right\|^2 \leq 4 \Expect_A \Expect_\epsilon\left\|\sum_{i=1}^m \epsilon_i\left(A_i-B\right)\right\|^2,
        \end{aligned}
    $$
    where $\left(A^{\prime}\right)_{i=1}^m$ is a sequence of independent copies of $\left(A_i\right)_{i=1}^m$ and $\left(\epsilon_i\right)_{i=1}^m$ are Rademacher random variables. Henceforth we condition on $A$. Let $p=\log n$, and let $\|\cdot\|_{S_p}$ denote the Schatten $p$-norm. In what follows, we will use that for any matrix $X,\|X\| \leq\|X\|_{S_{2 p}} \leq e^{1 / 2}\|X\|$. To begin, we have
    $$
        \Expect_{\epsilon}\left\|\sum_{i=1}^m \epsilon_i\left(A_i-B\right)\right\|^2 \leq \Expect_{\epsilon}\left\|\sum_{i=1}^m \epsilon_i\left(A_i-B\right)\right\|_{S_{2 p}}^2 \leq\left(\Expect_{\epsilon}\left\|\sum_{i=1}^m \epsilon_i\left(A_i-B\right)\right\|_{S_{2 p}}^{2 p}\right)^{1 / p},
    $$
    where the second inequality follows from Jensen's inequality. We  apply the matrix Khintchine inequality (\cite{mackey2014matrix}, Corollary 7.4), which implies that
    $$
        \begin{aligned}
            \left(\Expect_{\epsilon}\left\|\sum_{i=1}^m \epsilon_i\left(A_i-B\right)\right\|_{S_{2 p}}^{2p}\right)^{1/p} \leq(2 p-1)\left\|\sum_{i=1}^{m}\left(A_i-B\right)^2\right\|_{S_{2 p}} & \leq(2 p-1) \sum_{i=1}^m\left\|\left(A_i-B\right)\right\|_{S_{2 p}}^2 \\
            & \leq e(2 p-1) \sum_{i=1}^m\left\|\left(A_i-B\right)\right\|^2 .
        \end{aligned}
    $$
    Putting all the developments so far together and taking expectations with respect to $A$, we have
    $$
        \Expect\left\|\sum_{i=1}^m A_i-B\right\|^2 \leq 4 e(2 p-1) \sum_{i=1}^m \Expect_{A_i}\left\|\left(A_i-B\right)\right\|^2 \leq 4 e(2 p-1) m \sigma^2 .
    $$
    To obtain the final result, we divide both ends of the above inequality by $m^2$.
\end{proof}
\subsection{Proof of Theorem \ref{thm:full sgd}}
\begin{proof}
    According to Lemma \ref{lm:coomon descent1}, we have
    \begin{equation}\label{eq3.1}
        \begin{aligned}
             & F(x_{t+1})-F(x_t)                                                                                     \\
             & \leq \nabla F(x_t){^T}d_{t+1}+\frac{1}{2}d_{t+1}^{T}\nabla^{2}F(x_t)d_{t+1}+\frac{M_0+M_1\|\nabla F(x_t)\|}{6}\|d_{t+1}\|^{3}                                                                        \\
             & = g_t^T d_{t+1}+\frac{1}{2}(d_{t+1})^{T}H_td_{t+1}+(\nabla F(x_t)-g_{t})^T d_{t+1}+\frac{1}{2}(d_{t+1})^{T}(\nabla^{2}F(x_t)-H_t)d_{t+1} \\
             & \quad +\frac{M_0+M_1\|\nabla F(x_t)\|}{6}\|d_{t+1}\|^{3}           \\
             & \stackrel{(\Diamond)}{\leq}-\frac{1}{2}\lambda_t \|d_{t+1}\|^2+\|\nabla F(x_t)-g_{t}\|\Delta+\frac{1}{2}\|\nabla^2 F(x_t)-H_{t}\|\Delta^2+\frac{M_0+M_1\|\nabla F(x_t)\|}{6}\Delta^{3},
        \end{aligned}
    \end{equation}
    where $(\Diamond)$ is due to Cauchy-Schwarz inequality and Lemma \ref{lem:model reduction unified}. At $t$-th iteration, if $\|d_{t+1}\|=\Delta$, then
    we can bound $\|\lambda_t d_{t+1}\|$ by
    \begin{equation}\label{eq proof1}
        \|\lambda_td_{t+1}\|\leq\frac{2}{\Delta}(F(x_t)-F(x_{t+1}))+2\|\nabla F(x_t)-g_{t}\|+   \|\nabla^2 F(x_t)-H_t\|\Delta+\frac{M_0+M_1\|\nabla F(x_t)\|}{3}\Delta^{2}.
    \end{equation}
    Otherwise, if $d_{t+1}< \Delta$ , then by the optimal condition of \eqref{eq:optimal condition} we  have $\lambda_t=0$. Hence, the upper bound \eqref{eq proof1} still holds. 
    
    Considering $\norm{\nabla F(x_{t+1})}$, we have
    \begin{equation}\label{eq:grad_bound_so}
        \begin{aligned}
            & \Expect\| \nabla F(x_{t+1})\| \\
            & \leq \Expect\left[\| \nabla F(x_{t+1})-\nabla F(x_t)-\nabla^2 F(x_t)d_{t+1}\| 
            +\|\nabla F(x_t)-g_t\|+\|(\nabla^2 F(x_t)-H_t)d_{t+1}\|+\|g_t+H_kd_{t+1}\|\right]     \\
        & \stackrel{(\natural)}{\leq} \Expect\left[\frac{M_0+M_1\|\nabla F(x_t)\|}{2}\Delta^2+\|\nabla F(x_t)-g_t\|+\Delta \|\nabla^2 F(x_t)-H_t\|+\|\lambda_td_{t+1}\|\right]  \\
        & \stackrel{(\triangle)}{\leq} \Expect\left[\frac{5M_0+5M_1\|\nabla F(x_t)\|}{6}\Delta^2+3\|\nabla F(x_t)-g_t\|+2\Delta \|\nabla^2 F(x_t)-H_t\|+\frac{2}{\Delta}(F(x_t)-F(x_{t+1})\right],
        \end{aligned}
    \end{equation}
where $(\natural)$ is due to Lemma \ref{lm:coomon descent1} and the optimality condition \eqref{eq:optimal condition}, and $(\triangle)$ is due to \eqref{eq proof1}. 

By setting $\Delta=\sqrt{\epsilon}$, we have
\begin{align*}
   & \Expect\| \nabla F(x_{t+1})\|  \\
        & \leq \Expect\Big[\frac{5M_0+5M_1\|\nabla F(x_t)\|}{6}\epsilon+3(G_0+G_1\|\nabla F(x_t)\|)\epsilon+2\epsilon(K_0+K_1\|\nabla F(x_t)\|)+\frac{2}{\sqrt{\epsilon}}\brbra{F(x_t)-F(x_{t+1})}\Big] \\
        & \leq \left(\frac{5M_0}{6}+3G_0+2K_0\right)\epsilon+\left(\frac{5M_1\epsilon}{6}+3G_1\epsilon+2K_1\epsilon\right)\Expect\left[\|\nabla F(x_t)\|\right]+\Expect\left[\frac{2}{\sqrt{\epsilon}}\brbra{F(x_t)-F(x_{t+1})}\right]      \\
        & \leq \left(\frac{5M_0}{6}+3G_0+2K_0\right)\epsilon+\left(\frac{5M_1\epsilon}{6}+3G_1\epsilon+2K_1\epsilon\right)\Expect\left[\|\nabla F(x_{t+1})\|+\|\nabla F(x_{t+1})-F(x_t)\|\right]\\
        &\quad +\Expect\left[\frac{2}{\sqrt{\epsilon}}\brbra{F(x_t)-F(x_{t+1})}\right]\\
        & \leq  \left(\frac{5M_0}{6}+3G_0+2K_0\right)\epsilon+\left(\frac{5M_1\epsilon}{6}+3G_1\epsilon+2K_1\epsilon\right)\Expect[\|\nabla F(x_{t+1})\|]\\
        &+\Expect\left[\left(\frac{5M_1\epsilon}{6}+3G_1\epsilon+2K_1\epsilon\right)\left(L_0+L1\|\nabla F(x_{t+1})\|\right)\sqrt{\epsilon}\right]+\Expect\left[\frac{2}{\sqrt{\epsilon}}\brbra{F(x_t)-F(x_{t+1})}\right].
\end{align*}
Therefore, we have
    \begin{equation*}
        \begin{aligned}
            & \left(1-\frac{5M_1}{6}\epsilon-3G_1\epsilon-2K_1\epsilon-\mathcal{O}(\epsilon)\right)\Expect[\|\nabla F(x_{t+1})\|]\\
            & \leq \left(\frac{5M_0}{6}+3G_0+2K_0\right)\epsilon+\left(\frac{5M_1\epsilon}{6}+3G_1\epsilon+2K_1\epsilon\right)L_0\sqrt{\epsilon} \\
         & \quad +\Expect\left[\frac{2}{\sqrt{\epsilon}}(F(x_t)-F(x_{t+1})\right]                                               \\
         & \leq \mathcal{O}(\epsilon)+\Expect\left[\frac{2}{\sqrt{\epsilon}}(F(x_t)-F(x_{t+1})\right].                 \\
        \end{aligned}
    \end{equation*}
    Because $\epsilon < \frac{3}{5M_1+18G_1+12K_1}$ and $T=\mathcal{O}(\epsilon^{-3/2})$ , we have
    \begin{equation}\label{FOSP derivation}
        \begin{aligned}
\frac{1}{T}\Expect\left[\sum_{t=0}^{T-1}\|\nabla F(x_{t+1})\|\right] & \leq\mathcal{O}(\epsilon)+\frac{4\Delta_F}{T\sqrt{\epsilon}}  =\mathcal{O}(\epsilon).
        \end{aligned}
    \end{equation}
    For the second-order condition,
    \begin{equation}
        \begin{aligned}
            & \Expect[F(x_{t+1})-F(x_t)] \\
            \leq & \Expect\left[-\frac{1}{2}\lambda_t \|d_{t+1}\|^2\right]+\Expect\left[\|\nabla F(x_t)-g_{t}\|\Delta_{t}+\frac{1}{2}\|\nabla^2 F(x_t)-H_{t}\|\Delta_t^2+\frac{M_0+M_1\|\nabla F(x_t)\|}{6}\|d_{t+1}\|^{3}\right] \\
            \leq & \Expect[\lambda_t]\cdot(-\frac{1}{2}\epsilon)+ \left(G_0+\frac{1}{2}K_0+\frac{M_0}{6}\right)\epsilon^{3/2}+\left(G_1+\frac{1}{2}K_1+\frac{M_1}{6}\right)\|\nabla F(x_t)\| \epsilon^{3/2}
        \end{aligned}
    \end{equation}
    Summing from $t=0 $ to $T-1$, and plugging $T$ into the equation, we have
    \begin{equation}\label{descent speed}
        \frac{1}{T}\Expect\left[\sum_{t=0}^{T-1}(F(x_t)-F(x_{t+1}))\right] \leq\frac{1}{T} \cdot (F(x_0)-F^*)\leq\mathcal{O}(\epsilon^{\frac{3}{2}}).
    \end{equation}
    Combining \eqref{FOSP derivation}-\eqref{descent speed} , we have
    \begin{equation}\label{eq3.5}
        \begin{aligned}
            \Expect\left[\lambda_{\bar{t}}\right] & \leq \frac{2}{\epsilon}\bsbra{\mcal O(\epsilon^{\frac{3}{2}})+\left(G_0+\frac{1}{2}K_0+\frac{M_0}{6}\right)\epsilon^{3/2}+\left(G_1+\frac{1}{2}K_1+\frac{M_1}{6}\right)\Expect\left[\frac{\sum_{t=0}^{T-1}\|\nabla F(x_t)\|}{T}\right]\epsilon^{3/2}} \\
          & \leq \mcal O(\sqrt{\epsilon})+2\left(G_1+\frac{1}{2}K_1+\frac{M_1}{6}\right)\Expect\left[\frac{\sum_{t=0}^{T-1}\|\nabla F(x_{t+1})-\nabla F(x_t)\|}{T}]\epsilon^{3/2}\right] \\
          & \leq \mcal O(\sqrt{\epsilon})+2\left(G_1+\frac{1}{2}K_1+\frac{M_1}{6}\right)\Expect\left[\frac{\sum_{t=0}^{T-1}(L_0+L_1\|\nabla F(x_{t+1})\|)}{T}\epsilon^{3/2}\right] \\
          & \leq \mcal O(\sqrt{\epsilon}),
        \end{aligned}
    \end{equation}
    where $\bar{t}$ is uniformly sampled from $0, \dots, T-1$. By the optimality condition \eqref{eq:optimal condition}, we have
    \begin{equation*}
        \begin{aligned}
            -\Expect[\lambda_{\bar{t}}]I \preceq \Expect[H_{\bar{t}}]
             & \preceq \Expect[\nabla^2 F(x_{\bar{t}+1})]+\Expect[\|H_{\bar{t}}-\nabla^2 F(x_{\bar{t}+1})\|I]                                                                       \\
             & \preceq \Expect[\nabla^2 F(x_{\bar{t}+1})]+\Expect[\|H_{\bar{t}}-\nabla^2 F (x_{\bar{t}})\|I]+\Expect[\|\nabla^2 F(x_{\bar{t}})-\nabla^2 F(x_{\bar{t}+1})\|I]     \\
             & \stackrel{(\Diamond)}{\preceq} \Expect[\nabla^2 F(x_{\bar{t}+1})]+ \mathcal{O}(\sqrt{\epsilon})I+\mathcal{O}(\sqrt{\epsilon}) \Expect[\|\nabla F(x_{\bar{t}})\|I] \\
             & \preceq \Expect[\nabla^2 F(x_{\bar{t}+1})]+ \mathcal{O}(\sqrt{\epsilon}),                                                                                         \\
        \end{aligned}
    \end{equation*}
    where $(\Diamond)$ is due to Assumption \ref{assm: hessian} and Lemma \ref{lm var for hessian}. Hence we conclude
    $$
        \Expect[\lambda_{\min}(\nabla^2 F(x_{\bar{t}+1}))]\geq-\mathcal{O}(\sqrt{\epsilon}).$$
\end{proof}

\section{Variance Reduction}
\label{Sec：appx-VR}
Note that the smoothness condition in Assumption~\ref{assm L0L1} is solely imposed on the main objective $F$ irrespective of its components. As it is the standard assumption in variance-reduced optimization \cite{fang2018spider}, we impose the following averaged smoothness condition of $F$ and its components.

\begin{assumption}\label{assm:spider}
    In the stochastic setting, it holds that
    $$\Expect\norm{\nabla f(x,\xi) - \nabla f(x',\xi)} \leq (L_0+L_1\norm{\nabla F(x)})\norm{x-x'}.$$
\end{assumption}
We first give an important lemma, which helps to bound the variance of the SPIDER gradient estimator.
\begin{appxlem}\label{lem:sum_vr_bound} Let $\hat{\delta}_{t}=\nabla f(x_t;\mcal S_3)-\nabla f(x_{t-1};\mcal S_3)-\brbra{\nabla F(x_{t})-\nabla F(x_{t-1})}$. Given $q$, $\Delta_t=\Delta$, we have
    $$ \Expect\left[\biggl\|\sum_{\tau=1}^{t}\hat{\delta}_{\tau}\biggl\|\right] \leq \sqrt{2}L_1\frac{\Delta}{\sqrt{|\mcal S_3|}}\sum_{\tau=1}^{t}\Expect\bsbra{\Vert\nabla F(x_{\tau})\Vert}+\sqrt{2}L_0\frac{\Delta}{\sqrt{|\mcal S_3|}} \sqrt{q}$$.
\end{appxlem}
\subsection{Proof of Theorem \ref{thm:fo_vr}}
\begin{proof}
    Consider $t$ such that $0<t<q$. Let $\delta_t=g_{t}-\nabla F(x_{t})$, $\hat{\delta}_{t}=\nabla f(x_t;\mcal S_3)-\nabla f(x_{t-1};\mcal S_3)-\brbra{\nabla F(x_{t})-\nabla F(x_{t-1})}$, then we have
    \begin{equation}
        \begin{split}\text{\ensuremath{\delta_{t}}} & =g_{t}-\nabla F(x_{t})\\
            & =g_{t-1}-\nabla F(x_{t-1})+\nabla f(x_t;\mcal S_3)-\nabla f(x_{t-1};\mcal S_3)-\brbra{\nabla F(x_{t})-\nabla F(x_{t-1})}\\
            & =\delta_{t-1}+\hat{\delta}_{t}\\
            & =\delta_{0}+\hat{\delta}_{1}+\cdots+\hat{\delta}_{t}.
        \end{split}
    \end{equation}
    Then we can bound $\norm{\delta_t}$ by the following inequality
    \begin{equation}
        \begin{split}\Vert\text{\ensuremath{\delta_{t}}}\Vert & \leq\Vert\delta_{0}\Vert+\Vert\hat{\delta}_{1}+\cdots+\hat{\delta}_{t}\Vert\end{split}
    \end{equation}

    Let $\{\xi_{j}^{(\tau)}\}$ be the random
    samples used in the $\tau$-th iteration. Let $\mathcal{F}_{t}=\sigma\{x_1,...,x_t\}$. We use $\Expect_{\mathcal{F}_{t}}$ to denote the expectation on $\mathcal{F}_{t}$ and $\Expect_{t}$ to denote the conditional expectation on $\mathcal{F}_{t}$. Note that $x_{t}$ is $\mathcal{F}_{t}$-measurable.
We first bound $\norm{\hat{\delta}_t}$
    \begin{align*}
        \mathbb{E}_{t}\left[\Vert\hat{\delta}_{t}\Vert^{2}\right] = & \mathbb{E}_{t}\Big[\Big\|\frac{1}{|\mcal S_3|}\sum_{j=1}^{|\mcal S_3|}\brbra{\nabla f(x_{t};\xi_{j})-\nabla f(x_{t-1};\xi_{j})}-\brbra{\nabla F(x_{t})-\nabla F(x_{t-1})}\Big\|^{2}\Big] \\
        =                                                             & \frac{1}{|\mcal S_3|^{2}}\sum_{j=1}^{|\mcal S_3|}\mathbb{E}_{t}\bsbra{\norm{\nabla f(x_{t};\xi_{j})-\nabla f(x_{t-1};\xi_{j})}^{2}+\norm{\nabla F(x_{t})-\nabla F(x_{t-1})}^{2}          \\
        & -2\brbra{\nabla f(x_{t};\xi_{j})-\nabla f(x_{t-1};\xi_{j})}^{T}\brbra{\nabla F(x_{t})-\nabla F(x_{t-1})}}\\
        \overset{\Diamond}{=}                                         & \frac{1}{|\mcal S_3|^{2}}\sum_{j=1}^{|\mcal S_3|}\mathbb{E}_{t}\bsbra{\norm{\nabla f(x_{t};\xi_{j})-\nabla f(x_{t-1};\xi_{j})}^{2}-\norm{\nabla F(x_{t})-\nabla F(x_{t-1})}^{2}} \\
        \leq & \frac{1}{|\mcal S_3|^{2}}\sum_{j=1}^{|\mcal S_3|}\mathbb{E}_{t}\bsbra{\norm{\nabla f(x_{t};\xi_{j})-\nabla f(x_{t-1};\xi_{j})}^{2}} \\
        \overset{\triangle}{\leq} & \frac{1}{|\mcal S_3|^{2}}\sum_{j=1}^{|\mcal S_3|}\brbra{L_{0}+L_{1}\norm{\nabla F(x_{t})}}^{2}\norm{d_{t}}^{2} \\
        \leq & 2\frac{\Delta^{2}}{|\mcal S_3|}\brbra{L_{0}^{2}+L_{1}^{2}\norm{\nabla F(x_{t})}^{2}},
    \end{align*}
    where $\Diamond$ is due to $\Expect\bsbra{\nabla f(x_{t};\xi_{j})-\nabla f(x_{t-1};\xi_{j})} = \nabla F(x_{t})-\nabla F(x_{t-1})$ and $\triangle$ is due to Assumption \ref{assm:spider}.

    Applying Lemma \ref{lem:sum_vr_bound}, we have
    \begin{equation*}
        \begin{split}\mathbb{E}\left[\biggl\|\sum_{\tau=1}^{t}\hat{\delta}_{\tau}\biggl\|\right] & \leq\sqrt{2}L_{1}\frac{\Delta}{\sqrt{|\mcal S_3|}}\sum_{\tau=1}^{t}\mathbb{E}\bsbra{\norm{\nabla F(x_{\tau})}}+\sqrt{2}L_0\frac{\Delta}{\sqrt{|\mcal S_3|}} \sqrt{q}.
        \end{split}
    \end{equation*}
    According to Lemma \ref{lm sgd_var-bound for g-estimator}, we have $\Expect\bsbra{\norm{\delta_{0}}}=\Expect\bsbra{\norm{\nabla F(x_{0})-g_{0}}}\leq\frac{1}{\sqrt{|\mcal S_1|}}\left(G_{0}+G_{1}\Expect\bsbra{\norm{\nabla F(x_{0})}}\right)$
    \begin{align} \label{eq:vr_cound_thm4}
        \begin{split}
            \Expect\bsbra{\norm{\delta_{t}}} & \leq \Expect\bsbra{\norm{\delta_{0}}}+\Expect\bsbra{\norm{\hat{\delta}_{1}+\cdots+\hat{\delta}_{t}}}\\
            & \leq\frac{1}{\sqrt{|\mcal S_1|}}\left(G_{0}+G_{1}\Expect\bsbra{\norm{\nabla F(x_{0})}}\right)+\sqrt{2}L_{1}\frac{\Delta}{\sqrt{|\mcal S_3|}}\sum_{\tau=1}^{t}\mathbb{E}\bsbra{\norm{\nabla F(x_{\tau})}}+\sqrt{2}L_{0}\frac{\Delta}{\sqrt{|\mcal S_3|}}\sqrt{q}.
        \end{split}
    \end{align}
    Because $\Delta=\epsilon$, $|\mcal S_1|=\frac{1}{\epsilon^2}$, $|\mcal S_3|=\frac{1}{\epsilon}$, $q=\frac{1}{8G_1\epsilon}$ and $\epsilon\leq\frac{G_1^2}{2L_1^2}$, we have $\sqrt{2}L_1\epsilon^{3/2}\leq G_1\epsilon$,  and
    \begin{align}
        \begin{split}
            \Expect\bsbra{\norm{\delta_{t}}} & \leq \left(G_0+\frac{L_0}{2\sqrt{G_1}}\right)\epsilon + G_1\epsilon \Expect\bsbra{\norm{\nabla F(x_{0})}} + \sqrt{2}L_1\epsilon^{3/2}\sum_{\tau=1}^{t}\mathbb{E}\bsbra{\norm{\nabla F(x_{\tau})}}\\
            & \leq \left(G_0+\frac{L_0}{2\sqrt{G_1}}\right)\epsilon + G_1\epsilon \sum_{\tau=0}^{t}\mathbb{E}\norm{\nabla F(x_{\tau})}, \\
            \sum_{\tau=0}^{t}\Expect\bsbra{\norm{\delta_{\tau}}} & \leq \left(G_0+\frac{L_0}{2\sqrt{G_1}}\right)\epsilon \cdot t + G_1\epsilon \cdot t \sum_{\tau=0}^{t}\mathbb{E}\bsbra{\norm{\nabla F(x_{\tau})}}\\
            & \leq \left(G_0+\frac{L_0}{2\sqrt{G_1}}\right)\epsilon \cdot t + G_1\epsilon \cdot q \sum_{\tau=0}^{t}\mathbb{E}\bsbra{\norm{\nabla F(x_{\tau})}} \\
            & \leq \left(G_0+\frac{L_0}{2\sqrt{G_1}}\right)\epsilon \cdot t + \frac{1}{8} \sum_{\tau=0}^{t}\mathbb{E}\bsbra{\norm{\nabla F(x_{\tau})}}.
        \end{split}
    \end{align}
    It is easy to verify that the above inequality also holds for any $t\geq q$. According to the equation \eqref{eq:gradient_bound_thm1} of Theorem \ref{thm:fo_psd}, we are able to bound $\norm{\nabla F(x_{t})}$
    \begin{equation}
        \begin{split}\Expect\bsbra{\norm{\nabla F(x_{t})}}\leq & \frac{1}{\Delta}\Expect\bsbra{F(x_{t})-F(x_{t+1})}+2\Expect\bsbra{\norm{\nabla F(x_{t})-g_{t}}}+2\norm{B_t}\Delta+\frac{1}{2}\brbra{L_{0}+L_{1}\Expect\bsbra{\norm{\nabla F(x_{t})}}}\Delta\\
            = & \frac{1}{\Delta}\Expect\bsbra{F(x_{t})-F(x_{t+1})}+2\Expect\bsbra{\norm{\delta_{t}}}+2\beta \Delta+\frac{1}{2}\brbra{L_{0}+L_{1}\Expect\bsbra{\norm{\nabla F(x_{t})}}}\Delta.
        \end{split}
        \label{eq:gradient_bound1}
    \end{equation}
    Summing \eqref{eq:gradient_bound1} from $0 $ to $T-1$ and taking expectation on both sides, we have
    \begin{equation}
        \begin{split}
        & \sum_{t=0}^{T-1}\mathbb{E}\bsbra{\norm{\nabla F(x_{t})}} \\
        &\leq  \frac{1}{\Delta}\mathbb{E}\Bsbra{\sum_{t=0}^{T-1}\brbra{F(x_{t})-F(x_{t+1})}}+2\sum_{t=0}^{T-1}\mathbb{E}\bsbra{\norm{\delta_{t}}}+2\beta\Delta T+\frac{1}{2}\sum_{t=0}^{T-1}\brbra{L_{0}+L_{1}\mathbb{E}\bsbra{\norm{\nabla F(x_{t})}}}\Delta\\
         &   \leq  \frac{1}{\epsilon}\mathbb{E}\bsbra{F(x_{0})-F(x_{T})}+\Brbra{2G_0+\frac{L_0}{\sqrt{G_1}}+2\beta+\frac{1}{2}L_0}T\epsilon+\brbra{\frac{1}{4}+\frac{1}{2}L_{1}\epsilon}\sum_{t=0}^{T-1}\Expect\bsbra{\norm{\nabla F(x_{\tau})}},
        \end{split}
    \end{equation}
    and
    \begin{align}
        \Brbra{1-\frac{1}{4}-\frac{1}{2}L_{1}\epsilon}\sum_{t=0}^{T-1}\mathbb{E}\bsbra{\norm{\nabla F(x_{t})}} \leq & \frac{1}{\epsilon}\Delta_F + \Brbra{2G_0+\frac{L_0}{\sqrt{G_1}}+2\beta+\frac{1}{2}L_0}T\epsilon.
    \end{align}
    Since $\epsilon \leq \frac{1}{2L_1}$ and $T=\mcal O(\epsilon^{-2})$, we have
    \begin{equation}
        \begin{split}\frac{1}{T}\sum_{t=0}^{T-1}\mathbb{E}\bsbra{\norm{\nabla F(x_{t})}}\leq & \frac{\Delta_F}{\epsilon T}+\mcal O(\epsilon) =  \mcal O(\epsilon).
        \end{split}
    \end{equation}

\end{proof}

\subsection{Proof of Theorem \ref{thm:so_vr}}
\begin{proof} Let $\epsilon_t = \nabla^2 F(x_t)-H_{t}$. According to Lemma \ref{lm var for hessian}, we have
\begin{align*}
    \Expect\left\|\epsilon_t\right\| \leq
        (K_0+K_1\Expect\|\nabla F(x_t)\|)\sqrt{\epsilon}.
\end{align*}
According to equation \eqref{eq:vr_cound_thm4} in Theorem \ref{thm:fo_vr}, we are able to bound $\Expect\bsbra{\norm{\delta_{t}}}$
    \begin{align*}
        \begin{split}
            \Expect\bsbra{\norm{\delta_{t}}} & \leq \Expect\bsbra{\norm{\delta_{0}}}+\Expect\bsbra{\norm{\hat{\delta}_{1}+\cdots+\hat{\delta}_{t}}}\\
            & \leq\frac{1}{\sqrt{|\mcal S_1|}}\left(G_{0}+G_{1}\Expect\bsbra{\norm{\nabla F(x_{0})}}\right)+\frac{\Delta}{\sqrt{|\mcal S_3|}}\sqrt{2}L_{1}\sum_{\tau=1}^{t}\mathbb{E}\bsbra{\norm{\nabla F(x_{\tau})}}+\frac{\Delta}{\sqrt{|\mcal S_3|}}\sqrt{2q}L_{0}
        \end{split}
    \end{align*}
    Because $\Delta=\sqrt{\epsilon}$, $|\mcal S_1|=\frac{1}{\epsilon^2}$, $\mcal S_2=\frac{22\log(n)}{\epsilon}$, $|\mcal S_3|=\frac{1}{\epsilon^{3/2}}$, $q=\frac{1}{\epsilon^{1/2}}$ and $\epsilon\leq\frac{G_1^4}{4L_1^4}$, we have
    \begin{equation}
        \begin{split}\mathbb{E}\bsbra{\Vert\delta_{t}\Vert} & \leq\mathbb{E}\bsbra{\Vert\delta_{0}\Vert}+\mathbb{E}\bsbra{\Vert\hat{\delta}_{1}+\cdots+\hat{\delta}_{t}\Vert}\\
            & \leq\left(G_{0}+G_{1}\Expect\bsbra{\norm{\nabla F(x_{0})}}\right)\epsilon+\sqrt{2}L_{1}\epsilon^{5/4}\sum_{\tau=1}^{t}\mathbb{E}\bsbra{\norm{\nabla F(x_{\tau})}}+\sqrt{2}L_{0}\epsilon\\
            & \leq\left(G_{0}+\sqrt{2}L_{0}\right)\epsilon+G_{1}\epsilon\sum_{\tau=0}^{t}\mathbb{E}\bsbra{\norm{\nabla F(x_{\tau})}}.
        \end{split}
    \end{equation}
    It follows that
    \begin{equation}
        \begin{split}\sum_{\tau=0}^{t-1}\mathbb{E}\Vert\delta_{\tau}\Vert\leq & \left(G_{0}+\sqrt{2}L_{0}\right)t\epsilon+G_{1}\epsilon\cdot t\sum_{\tau=0}^{t-1}\mathbb{E}\norm{\nabla F(x_{\tau})}\\
            \leq & \left(G_{0}+\sqrt{2}L_{0}\right)t\epsilon+G_{1}\epsilon\cdot q\sum_{\tau=0}^{t-1}\mathbb{E}\norm{\nabla F(x_{\tau})}\\
            = & \left(G_{0}+\sqrt{2}L_{0}\right)t\epsilon+G_{1}\epsilon^{1/2}\sum_{\tau=0}^{t-1}\mathbb{E}\norm{\nabla F(x_{\tau})}.
        \end{split}
    \end{equation}
    According to the equation \eqref{eq:grad_bound_so} in the analysis of Theorem \ref{thm:full sgd}, we are able to bound $\norm{\nabla F(x_{t})}$
    \begin{equation*}
        \begin{split}\Expect[\|\nabla F(x_{t+1})\|]\leq & \Expect\Big[\frac{5M_{0}+5M_{1}\|\nabla F(x_{t})\|}{6}\epsilon+3\|\nabla F(x_{t})-g_{t}\|+2\|\nabla^{2}F(x_{t})-H_{t}\|\sqrt{\epsilon}+\frac{2}{\sqrt{\epsilon}}(F(x_{t})-F(x_{t+1}))\Big]
        \end{split}
    \end{equation*}
    and
    \begin{align}
        \begin{split}
            & \frac{1}{T}\sum_{\tau=0}^{T-1}\Expect[\|\nabla F(x_{t+1})\|]\\
           & \leq  \left(\frac{5M_{0}}{6}+3G_{0}+3\sqrt{2}L_{0}+2K_{0}\right)\epsilon+\left(\frac{5M_{1}}{6}\epsilon+3G_{1}\sqrt{\epsilon}+2K_{1}\epsilon\right)\frac{1}{T}\sum_{\tau=0}^{T-1}\Expect[\|\nabla F(x_{t})\|]\\
            & \quad +\frac{2}{T\sqrt{\epsilon}}\Expect[F(x_{0})-F(x_{T})] \\
            & \leq  \mcal O(\epsilon)+\left(\mcal O(\epsilon) + 3G_{1}\sqrt{\epsilon}\right)\frac{1}{T}\sum_{\tau=0}^{T-1}\Expect[\|\nabla F(x_{t+1})\|]\\
            &\quad + \left(\mcal O(\epsilon) + 3G_{1}\sqrt{\epsilon}\right)\frac{1}{T}\sum_{\tau=0}^{T-1}\brbra{L_0+L_1\Expect\norm{\nabla F(x_{t+1})}}\sqrt{\epsilon}+\frac{2}{T\sqrt{\epsilon}}\Delta_F \\
            & \leq  \mcal O(\epsilon) + \left(\mcal O(\epsilon)+3G_{1}\sqrt{\epsilon}\right)\frac{1}{T}\sum_{\tau=0}^{T-1}\Expect[\|\nabla F(x_{t+1})\|] + \frac{2}{T\sqrt{\epsilon}}\Delta_F.
        \end{split}
    \end{align}
    Since $\epsilon\leq\frac{1}{36G_1^2}$
    and $T=\mathcal{O}(\epsilon^{-3/2})$, it follows that
    \begin{equation}
        \begin{split}\frac{1}{T}\sum_{\tau=0}^{T-1}\Expect[\|\nabla F(x_{t+1})\|]\leq & \mathcal{O}(\epsilon)+\frac{4}{T\sqrt{\epsilon}}\Delta_{J} \leq  \mathcal{O}(\epsilon).
        \end{split}
    \end{equation}
    For the second-order condition, we have
    \begin{align}
        \begin{split}
            & F(x_{t+1})-F(x_t)\\
            &\leq-\frac{1}{2}\lambda_t \|d_{t+1}\|^2+\|\nabla F(x_t)-g_{t}\|\Delta+\frac{1}{2}\|\nabla^2 F(x_t)-H_{t}\|\Delta^2+\frac{M_0+M_1\|\nabla F(x_t)\|}{6}\Delta^3 \\
            &= -\frac{1}{2}\lambda_t\Delta^2 + \|\delta_t\|\Delta+\frac{1}{2}\|\epsilon_t\|\Delta^2+\frac{M_0+M_1\|\nabla F(x_t)\|}{6}\Delta^3.
        \end{split}
    \end{align}
    Therefore, we have
    \[
    \lambda_t \leq \frac{2}{\Delta} \norm{\delta_t} + \|\epsilon_t\| + \frac{M_0+M_1\|\nabla F(x_t)\|}{3}\Delta + \frac{2}{\Delta^2}\brbra{F(x_t) - F(x_{t+1})},
    \]
    and hence
    \begin{align}
        \begin{split}
            &\frac{1}{T}\sum_{t=0}^{T-1} \mbb E [\lambda_t] \\
            &\leq \frac{2}{T\sqrt{\epsilon}}\sum_{t=0}^{T-1} \mbb E\norm{\delta_t} + \frac{1}{T}\sum_{t=0}^{T-1} \mbb E\|\epsilon_t\| + \frac{1}{T}\sum_{t=0}^{T-1}\frac{M_0+M_1\mbb E\|\nabla F(x_t)\|}{3}\sqrt{\epsilon} + \frac{2}{T\epsilon}\mbb E\bsbra{F(x_0) - F(x_{T})} \\
            &\leq \left(2G_{0}+2\sqrt{2}L_{0}+K_{0}+\frac{M_0}{3}\right)\sqrt{\epsilon}+\left(2G_{1} + K_1\sqrt{\epsilon}+\frac{M_1}{3}\sqrt{\epsilon}\right)\frac{1}{T}\sum_{t=0}^{T-1}\mathbb{E}\norm{\nabla F(x_{t})} + \frac{2}{T\epsilon}\Delta_F \\
            &\leq \left(2G_{0}+2\sqrt{2}L_{0}+K_{0}+\frac{M_0}{3}\right)\sqrt{\epsilon}+\left(2G_{1} + K_1\sqrt{\epsilon}+\frac{M_1}{3}\sqrt{\epsilon}\right)\frac{1}{T}\sum_{t=0}^{T-1}\mathbb{E}\norm{\nabla F(x_{t+1})} \\
            &+ \left(2G_{1} + K_1\sqrt{\epsilon}+\frac{M_1}{3}\sqrt{\epsilon}\right)\frac{1}{T}\sum_{t=0}^{T-1}\brbra{L_0+L_1\mathbb{E}\norm{\nabla F(x_{t+1})}}\sqrt{\epsilon} + \frac{2}{T\epsilon}\Delta_F \\
            &= \mcal O(\sqrt{\epsilon}) + \brbra{2G_1+2L_1G_1+\mcal O(\sqrt{\epsilon})}\frac{1}{T}\sum_{t=0}^{T-1}\mathbb{E}\norm{\nabla F(x_{t+1})} + \frac{2}{T\epsilon}\Delta_F \\
            &\overset{\Diamond}{\leq} \mcal O(\sqrt{\epsilon}) + \mcal O(\epsilon) + \mcal O(\sqrt{\epsilon})\\
            &= \mcal O(\sqrt{\epsilon}),
        \end{split}
    \end{align}
where $\Diamond$ is due to $\mathbb{E}\norm{\nabla F(x_{\bar{t}+1})}\leq \mcal O(\epsilon)$ and $T=\mcal O(\epsilon^{-3/2})$.

    By the optimality condition \eqref{eq:optimal condition} and the $(M_0,M_1)$-smoothness of Hessian, we have
    \begin{equation*}
        \begin{aligned}
            -\Expect[\lambda_{t}]I \preceq \Expect[H_{t}]
             & \preceq \Expect[\|H_{t}-\nabla^2 F(x_{t})\|]I + \Expect[\nabla^2 F(x_{t})]                                                         \\
             & \preceq \Expect[\|H_{t}-\nabla^2 F(x_{t})\|]I + \Expect[\nabla^2 F(x_{t+1})] + \Expect[\|\nabla^2 F(x_{t})-\nabla^2 F(x_{t+1})\|]I \\
             & \preceq \Expect[\|\epsilon_t\|]I + \Expect[\nabla^2 F(x_{t+1})] + \Expect\bsbra{\brbra{M_0+M_1\norm{\nabla F(x_{t+1})}}\Delta}I.
        \end{aligned}
    \end{equation*}
    According to Lemma \ref{lm var for hessian}, we have
    \begin{align}
        \begin{split}
            \mathbb{E}\norm{\epsilon_{\bar{t}}} = \frac{1}{T}\sum_{t=0}^{T-1}\mathbb{E}\Vert\epsilon_{t}\Vert & \leq K_{0}\sqrt{\epsilon}+\frac{1}{T}K_1\sqrt{\epsilon}\sum_{\tau=0}^{T-1}\mathbb{E}\norm{\nabla F(x_{t})} \\
            & \leq K_{0}\sqrt{\epsilon}+K_1\sqrt{\epsilon}\,\mathbb{E}\norm{\nabla F(x_{\bar{t}+1})} + \frac{1}{T}K_1\sqrt{\epsilon}\sum_{\tau=0}^{T-1}\mathbb{E}\bsbra{\brbra{L_0+L_1\norm{\nabla F(x_{t+1})}}\Delta} \\
            & = \left(K_{0}\sqrt{\epsilon}+L_0K_1\epsilon\right)+\brbra{K_1\sqrt{\epsilon}+L_1K_1\epsilon}\mathbb{E}\norm{\nabla F(x_{\bar{t}+1})} \\
            & \leq \mcal O(\sqrt{\epsilon}).
        \end{split}
    \end{align}
    Therefore, we have
    \begin{align}
        \begin{split}
            -\Expect[\lambda_{\bar{t}}]I \preceq \Expect[H_{\bar{t}}]
            &\preceq \Expect[\|\epsilon_{\bar{t}}\|]I + \Expect[\nabla^2 F(x_{\bar{t}+1})] + \Expect\bsbra{\brbra{M_0+M_1\norm{\nabla F(x_{\bar{t}+1})}}\Delta}I \\
            &\preceq \mcal O(\sqrt{\epsilon})I + \Expect[\nabla^2 F(x_{\bar{t}+1})] + \Expect\bsbra{\brbra{M_0+M_1\norm{\nabla F(x_{\bar{t}+1})}}\sqrt{\epsilon}}I \\
            &\preceq \mcal O(\sqrt{\epsilon})I + \Expect[\nabla^2 F(x_{\bar{t}+1})].
        \end{split}
    \end{align}
    Hence, we have
    $$
        \Expect[\lambda_{\min}(\nabla^2 F(x_{\bar{t}+1}))]\geq-\mathcal{O}(\sqrt{\epsilon}).$$

\end{proof}

\subsection{Proof of Lemma \ref{lem:sum_vr_bound}}
\begin{proof} We prove that for any $i\in\{0,1,\cdots,t\}$
    \[
    \Expect\Bsbra{\Bigl\|\sum_{\tau=1}^{t}\hat{\delta}_{\tau}\Bigl\|}\leq  \sqrt{2}L_1\frac{\Delta}{\sqrt{|\mcal S_3|}}\sum_{\tau=t-i+1}^{t}\Expect\bsbra{\Vert\nabla F(x_{\tau})\Vert}+\Expect\left[\sqrt{2L_0^2\frac{\Delta^2}{|\mcal S_3|}\cdot i+\biggl\|\sum_{\tau=1}^{t-i}\hat{\delta}_{\tau}\biggl\|^{2}}\right].
    \]
    We prove it by induction. The above inequality holds for $i=0$. Suppose that it holds for $i$, we then prove that it also holds for $i+1$.

    \begin{align*}
\Expect\left[\sqrt{2L_0^2\frac{\Delta^2}{|\mcal S_3|}\cdot i+\biggl\|\sum_{\tau=1}^{t-i}\hat{\delta}_{\tau}\biggl\|^{2}}\right] & =\Expect_{\mathcal{F}_{t-i}}\left[\Expect_{t-i}\sqrt{2L_0^2\frac{\Delta^2}{|\mcal S_3|}\cdot i+\biggl\|\sum_{\tau=1}^{t-i}\hat{\delta}_{\tau}\biggl\|^{2}}\right]                  \\
  & \leq \Expect_{\mathcal{F}_{t-i}}\left[\sqrt{\Expect_{t-i}\left[2L_0^2\frac{\Delta^2}{|\mcal S_3|}\cdot i+\Vert\sum_{\tau=1}^{t-i}\hat{\delta}_{\tau}\Vert^{2}\right]}\right].
    \end{align*}
    Because $\hat{\delta}_{\tau}=\nabla f(x_{\tau};\mcal S_3)-\nabla f(x_{\tau-1};\mcal S_3)-\brbra{\nabla F(x_{\tau})-\nabla F(x_{\tau-1})}$ is $\mathcal{F}_{t-i}$-measurable, we have
    any $1\leq\tau\leq t-i-1$,
    \begin{align*}
        \Expect_{t-i}\left[\Big\Vert\sum_{\tau=1}^{t-i}\hat{\delta}_{\tau}\Big\Vert^{2}\right] & =\Expect_{t-i}\left[\Vert\hat{\delta}_{t-i}\Vert^{2}+\Big\Vert\sum_{\tau=1}^{t-i-1}\hat{\delta}_{\tau}\Big\Vert^{2}+2\sum_{\tau=1}^{t-i-1}\hat{\delta}_{\tau}^{T}\hat{\delta}_{t-i}\right] \\   
        & =\Expect_{t-i}\left[\Vert\hat{\delta}_{t-i}\Vert^{2}\right]+\Big\Vert\sum_{\tau=1}^{t-i-1}\hat{\delta}_{\tau}\Big\Vert^{2}+2\sum_{\tau=1}^{t-i-1}\hat{\delta}_{\tau}^{T}\Expect_{t-i}\left[\hat{\delta}_{t-i}\right].
    \end{align*}
    In view of  $\Expect_{t-i}\bsbra{\hat{\delta}_{t-i}}=0$, we have
    \begin{align*}
        \Expect_{t-i}\left[\Big\Vert\sum_{\tau=1}^{t-i}\hat{\delta}_{\tau}\Big\Vert^{2}\right] & =\Expect_{t-i}\left[\Vert\hat{\delta}_{t-i}\Vert^{2}\right]+\Big\Vert\sum_{\tau=1}^{t-i-1}\hat{\delta}_{\tau}\Big\Vert^{2}                    \\
        & \leq 2L_0^2\frac{\Delta^2}{|\mcal S_3|}+2L_1^2\frac{\Delta^2}{|\mcal S_3|}\norm{\nabla F(x_{t-i})}^2+\Big\Vert\sum_{\tau=1}^{t-i-1}\hat{\delta}_{\tau}\Big\Vert^{2}.
    \end{align*}
    Therefore, 
    \begin{align*}
        \Expect\left[\sqrt{2L_0^2\frac{\Delta^2}{|\mcal S_3|}\cdot i+\biggl\|\sum_{\tau=1}^{t-i}\hat{\delta}_{\tau}\biggl\|^{2}}\right] & \leq \Expect_{\mathcal{F}_{t-i}}\left[\sqrt{\Expect_{t-i}\left[2L_0^2\frac{\Delta^2}{|\mcal S_3|}\cdot i+\Vert\sum_{\tau=1}^{t-i}\hat{\delta}_{\tau}\Vert^{2}\right]}\right]                                          \\
        & \leq \Expect_{\mathcal{F}_{t-i}}\left[\sqrt{2L_0^2\frac{\Delta^2}{|\mcal S_3|}\cdot(i+1)+2L_1^2\frac{\Delta^2}{|\mcal S_3|}\Vert\nabla F(x_{t-i})\Vert^{2}+\biggl\|\sum_{\tau=1}^{t-i-1}\hat{\delta}_{\tau}\biggl\|^{2}}\right]       \\
         & \leq \Expect_{\mathcal{F}_{t-i}}\left[\sqrt{2}L_1\frac{\Delta}{\sqrt{|\mcal S_3|}}\Vert\nabla F(x_{t-i})\Vert+\sqrt{2L_0^2\frac{\Delta^2}{|\mcal S_3|}\cdot(i+1)+\biggl\|\sum_{\tau=1}^{t-i-1}\hat{\delta}_{\tau}\biggl\|^{2}}\right] \\
         & =\sqrt{2}L_1\frac{\Delta}{\sqrt{|\mcal S_3|}}\Expect\bsbra{\Vert\nabla F(x_{t-i})\Vert}+\Expect\left[\sqrt{2L_0^2\frac{\Delta^2}{|\mcal S_3|}\cdot(i+1)+\biggl\|\sum_{\tau=1}^{t-i-1}\hat{\delta}_{\tau}\biggl\|^{2}}\right].
    \end{align*}
    Therefore,
    \begin{align*}
    \Expect\Bsbra{\Bigl\|\sum_{\tau=1}^{t}\hat{\delta}_{\tau}\Bigl\|} & \leq \sqrt{2}L_1\frac{\Delta}{\sqrt{|\mcal S_3|}}\sum_{\tau=t-i+1}^{t}\Expect\bsbra{\Vert\nabla F(x_{\tau})\Vert}+\Expect\left[\sqrt{2L_0^2\frac{\Delta^2}{|\mcal S_3|}\cdot i+\biggl\|\sum_{\tau=1}^{t-i}\hat{\delta}_{\tau}\biggl\|^{2}}\right]    \\
    & \leq \sqrt{2}L_1\frac{\Delta}{\sqrt{|\mcal S_3|}}\sum_{\tau=t-i}^{t}\Expect\bsbra{\Vert\nabla F(x_{\tau})\Vert}+\Expect\left[\sqrt{2L_0^2\frac{\Delta^2}{|\mcal S_3|}\cdot(i+1)+\biggl\|\sum_{\tau=1}^{t-i-1}\hat{\delta}_{\tau}\biggl\|^{2}}\right].
    \end{align*}
    Let $i=t$ and note that $t\leq q=\epsilon^{-1/2}$, we deduce
    \begin{align*}
    \Expect\Bsbra{\Bigl\|\sum_{\tau=1}^{t}\hat{\delta}_{\tau}\Bigl\|} & \leq \sqrt{2}L_1\frac{\Delta}{\sqrt{|\mcal S_3|}}\sum_{\tau=1}^{t}\Expect\bsbra{\Vert\nabla F(x_{\tau})\Vert}+\sqrt{2L_0^2\frac{\Delta^2}{|\mcal S_3|}\cdot t} \\
        & \leq \sqrt{2}L_1\frac{\Delta}{\sqrt{|\mcal S_3|}}\sum_{\tau=1}^{t}\Expect\bsbra{\Vert\nabla F(x_{\tau})\Vert}+\sqrt{2L_0^2\frac{\Delta^2}{|\mcal S_3|}\cdot q}.
    \end{align*}
\end{proof}
\newpage

\section{Inexactness of Second-Order Trust Region Method}\label{sec.lowrsogs}
To remedy computational costs in second-order methods, generally, one can apply two different strategies:
\begin{itemize}[leftmargin=*]
    \item Hessian approximation using probabilistic models \cite{bandeiraConvergenceTrustRegionMethods2014,bollapragadaExactInexactSubsampled2019,xuSubsampledNewtonMethods2016} and subspace (sketching) techniques \cite{woodruffSketchingToolNumerical2014,berahasInvestigationNewtonSketchSubsampled2020,zhang2022drsom,liu2023stochastic}.
    \item Inexact minimization of the subproblems that are frequently imposed in Krylov subspace methods; see \citet{nocedal1999numerical,cartis_adaptive_2011} for example.
\end{itemize}
Assumption \ref{assm:drsom} assembles the common results for second-order Lipschitzian functions\footnote{For example, see \cite{xuNewtontypeMethodsNonconvex2020}.} that naturally extends for generalized smoothness. Note that it can be satisfied with a large enough Krylov subspace to construct $V_t$, however, the choice can be flexible \cite{woodruffSketchingToolNumerical2014,cartisEvaluationComplexityAlgorithms2022}. Then we have the following results.
\begin{lemma}\label{lm3.5}
    Under Assumption \ref{assm:drsom}, then we have
    \begin{equation*}
        \Expect_t\bsbra{\|(H_t-\tilde{H}_t)d_{t+1}\|} \leq \tilde{C} \Delta^2,
    \end{equation*}
    where $\tilde{C}=C_0+2K_0+(C_1+2K_1)\|\nabla F(x_t)\|$ and $d_t \leq \Delta=\sqrt{\epsilon}$.
\end{lemma}
\begin{proof}
    Recalling $\tilde{\nabla^2} F(x_t)=V_t V_t{^T}\nabla^2 F(x_t)V_t V_t{^T} $, we have
    \begin{equation*}
        \begin{aligned}
             & \Expect_t\|(H_t-\tilde{H}_t)d_{t+1}\|\\
             & \leq \Expect_t\bsbra{\|(\nabla^2 F(x_t)-\tilde{\nabla^2} F(x_t))d_{t+1}\|+ \|(H_t-\nabla^2 F(x_t))d_{t+1}\|+ \|V_t V_t{^T}(\nabla^2 F(x_t)-H_t)V_t V_t{^T} d_{t+1}\|} \\
             & \stackrel{(\Diamond)}{\leq} (C_0+C_1\|\nabla F(x_t)\|)\Delta^2+ \Expect_t[\|(H_t-\nabla^2 F(x_t))d_{t+1}\|]+\Expect_t[\|V_t V_t{^T}\|\cdot\|(\nabla^2 F(x_t)-H_t)d_{t+1}\|] \\
             & \stackrel{(\sharp)}{\leq} (C_0+C_1\|\nabla F(x_t)\|)\Delta^2+2\Expect_t[\|(\nabla^2 F(x_t)-H_t)d_{t+1}\|] \\
             & \stackrel{(\natural)}{\leq} (C_0+C_1\|\nabla F(x_t)\|)\Delta^2+2\cdot (K_0+K_1\|\nabla F(x_t)\|)\sqrt{\epsilon}\Delta= (C_0+2K_0+(C_1+2K_1)\|\nabla F(x_t)\|)\Delta^2,
        \end{aligned}
    \end{equation*}
    where $(\Diamond)$ is due to Assumption \ref{assm:drsom} and the fact that $V_t V_t{^T} d_{t+1}=d_{t+1}$; $(\sharp)$  is due to the fact that $V_t{^T}V_t=I$ and $V_tV_t{^T}$ has the same non-zero eigenvalue with $V_t{^T}V_t$ ; $(\natural)$ is due to  Lemma \ref{lm sgd_var-bound for g-estimator} and  the fact that $d_{t+1}\leq \Delta_{t}=\Delta =\sqrt{\epsilon}$.
\end{proof}

\paragraph{Selecting $V_t$ from $\text{span}\{g_t, d_t\}$}

Inspired by \citet{zhang2022drsom,liu2023stochastic},
we find that by introducing a specific $V_t$, subproblem \eqref{eq:unified} can be solved equivalently by solving a two-dimensional trust region model. This model is devised to ascertain the step size for both the gradient and momentum within the context of the Heavy Ball method. Notably, this approach provides a significantly simpler alternative to the full-dimensional quadratic program traditionally employed in standard trust region methods, offering increased efficiency without sacrificing the effectiveness of the optimization process.

More specifically, let $d_t=x_t-x_{t-1}$, $B_t=\tilde{H}_t$, where $\tilde{H}_t=V_t V_t^T H_t V_t V_t^T$ and $V_t$ is the orthonormal bases for $\mathcal{L}_t:=\operatorname{span}\left\{g_t, d_t\right\}$. We find in this case, \eqref{eq:unified} is equivalent to a two-dimension model;  see \citet{zhang2022drsom}.
\begin{lemma}\label{lmdrsom1}
    When setting $B_t=\tilde{H}_t$, the subproblem (\ref{eq:unified}) is equivalent to
    \begin{equation}\label{eq:DRTR}
        \begin{aligned}
             & \min_{\alpha \in \mathbb{R}^2} &  & m_t(\alpha):=F\left(x_t\right)+c_t^T \alpha+\frac{1}{2} \alpha^T Q_t \alpha \\
             & \   \textup{ s.t.}             &  & \|\alpha\|_{G_t} \leq \Delta_t,
        \end{aligned}
    \end{equation}
    if we update by
    \[x_{t+1}=x_t-\alpha_t^1 g_t+\alpha_t^2 d_t,\]with
    $$
        Q_t=\begin{bmatrix}
            g_t^T H_t g_t  & -d_t^T H_t g_t \\
            -d_t^T H_t g_t & d_t^T H_t d_t
        \end{bmatrix} \in \mathcal{S}^2,\\$$
    $$
        c_t:=\begin{pmatrix}
            -\|g_t\|^2 \\
            g_t^T d_t
        \end{pmatrix},\ G_t=\begin{bmatrix}
            g_t^T g_t  & -g_t^T d_t \\
            -g_t^T d_t & d_t^T d_t
        \end{bmatrix},
    $$
    and $\|\alpha\|_{G_t}=\sqrt{\alpha^T G_t \alpha}$.
\end{lemma}
The vector $\alpha_t$ is the global solution to DRTR problem~\eqref{eq:DRTR} if it is feasible and there exists a  Lagrange multiplier $\lambda_t \geq 0$ such that $\left(\alpha_t, \lambda_t\right)$ is the solution to the following equations:
\begin{equation}\label{optimal condi for original}
    \left(Q_t+\lambda G_t\right) \alpha+c_t=0, Q_t+\lambda G_t \succeq 0, \lambda\left(\Delta-\|\alpha\|_{G_t}\right)=0 .
\end{equation}
Then by the optimal condition \eqref{optimal condi for original}, we have the closed form solution of $\alpha_t$:
$$
    \alpha_t=-(Q_t+\lambda_tG_t)^{-1}c_t.
$$
Since $\alpha_t$ only has two dimensions, it can be easily solved numerically.
Therefore, by Lemma \ref{lmdrsom1}, we can solve subproblem \eqref{eq:unified} without explicitly computing the Hessian when setting $B_t=H_t$, although it conceptually utilizes the curvature information. In fact,  we only require two additional Hessian-vector products to formulate~\eqref{eq:DRTR}. This leads to the following more practical second-order algorithm in Algorithm \ref{alg:drsopo}.

\begin{algorithm}[ht]
    \caption{Dimension-Reduced Trust Region Method}\label{alg:drsopo}
    \begin{algorithmic}[1]
        \STATE Given $T$, error $\epsilon$
        \FOR{$t=1, \dots ,T$}
        \STATE  Draw samples $\mcal |\mcal S_1|$ and compute $g_t=\nabla f(x_t;\mcal |\mcal S_1|)$
        \STATE Draw samples $\mcal \mcal S_2$ and compute $H_t=\nabla^2 f(x_t;\mcal \mcal S_2)$
        \STATE Compute $\lambda_t$ by formalizing $Q_t$ and set $\Delta_t$
        \STATE Compute  stepsize $\alpha_1,\alpha_2 $ by solving the DRTR problem \eqref{eq:DRTR}
        \STATE Update: $x_{t+1} \gets x_{t}-\alpha_1 g_t+\alpha_2 d_t$
        \ENDFOR
    \end{algorithmic}
\end{algorithm}
\subsection{Proof of Theorem \ref{thm drtr}}
\begin{proof}
    By the definition of $\tilde H_t$, we know that $d_{t+1}^\trans H_td_{t+1}=d_{t+1}t^\trans \tilde{H}_td_{t+1}$, so the derivation of \eqref{eq3.1} still holds, and we have
     \begin{equation}\label{eq proofF.1}
        \|\lambda_td_{t+1}\|\leq\frac{2}{\sqrt{\epsilon}}(F(x_t)-F(x_{t+1}))+2\|\nabla F(x_t)-g_{t}\|+   \|\nabla^2 F(x_t)-H_t\|\Delta_t+\frac{M_0+M_1\|\nabla F(x_t)\|}{3}\|d_{t+1}\|^{2}.
    \end{equation}
So we have 
\begin{equation*}
        \begin{aligned}
            & \Expect[\| \nabla F(x_{t+1})\|] \\
            & \leq\Expect[\| \nabla F(x_{t+1})-\nabla F(x_t)-\nabla^2 F(x_t)d_{t+1}\|+\|\nabla F(x_t)-g_t\|+\|(\nabla^2 F(x_t)-H_t)d_{t+1}\|\\
            &\quad +\|g_t+\tilde{H_t}d_{t+1}\|+\|(H_t-\tilde{H}_t)d_{t+1}\|]     \\
        & \stackrel{(\natural)}{\leq} \Expect\left[\frac{M_0+M_1\|\nabla F(x_t)\|}{2}\Delta^2+\|\nabla F(x_t)-g_t\|+\Delta \|\nabla^2 F(x_t)-H_t\|+\|\lambda_td_{t+1}\|+\tilde{C}\Delta^2\right]                                             \\
        & \leq \Expect\left[\frac{5M_0+5M_1\|\nabla F(x_t)\|+6\tilde{C}}{6}\Delta^2+3\|\nabla F(x_t)-g_t\|+2\Delta \|\nabla^2 F(x_t)-H_t\|+\frac{2}{\sqrt{\epsilon}}(F(x_t)-F(x_{t+1})\right]                  \\
        & \leq \Expect\left[\frac{5M_0+5M_1\|\nabla F(x_t)\|+6\tilde{C}}{6}\epsilon+3(G_0+G_1\|\nabla F(x_t)\|)\epsilon+2\epsilon(K_0+K_1\|\nabla F(x_t)\|)+\frac{2}{\sqrt{\epsilon}}(F(x_t)-F(x_{t+1})\right] \\
        & \leq \left(\frac{5M_0}{6}+3G_0+4K_0+C_0\right)\epsilon+\left(\frac{5M_1\epsilon}{6}+3G_1\epsilon+4K_1\epsilon+C_1\epsilon\right)\Expect[\|\nabla F(x_t)\|]+\Expect\left[\frac{2}{\sqrt{\epsilon}}(F(x_t)-F(x_{t+1})\right]      \\
        & \leq \left(\frac{5M_0}{6}+3G_0+4K_0+C_0\right)\epsilon+\left(\frac{5M_1\epsilon}{6}+3G_1\epsilon+4K_1\epsilon+C_1\epsilon\right)\left(\Expect[\|\nabla F(x_{t+1})\|+\|\nabla F(x_{t+1})-F(x_t)\|]\right)             \\
        &\quad +\Expect[\frac{2}{\sqrt{\epsilon}}(F(x_t)-F(x_{t+1})]\\
        & \leq  \left(\frac{5M_0}{6}+3G_0+4K_0+C_0\right)\epsilon+\left(\frac{5M_1\epsilon}{6}+3G_1\epsilon+4K_1\epsilon+C_1\epsilon\right)\Expect[\|\nabla F(x_{t+1})\|] \\
        &\quad +\Expect[\left(\frac{5M_1\epsilon}{6}+3G_1\epsilon+4K_1\epsilon+C_1\epsilon\right)(L_0+L1\|\nabla F(x_{t+1})\|)\Delta]+\Expect[\frac{2}{\sqrt{\epsilon}}(F(x_t)-F(x_{t+1})].
        \end{aligned}
    \end{equation*}
It follows that 
    \begin{equation*}
        \begin{aligned}
           & \left(1-\frac{5M_1}{6}\epsilon-3G_1\epsilon-4K_1\epsilon-C_1\epsilon-\mathcal{O}(\epsilon)\right)\Expect[\|\nabla F(x_{t+1})\|] \\
        & \leq   \left(\frac{5M_0}{6}+3G_0+4K_0+C_0\right)\epsilon +\left(\frac{5M_1\epsilon}{6}+3G_1\epsilon+4K_1\epsilon+C_1\epsilon\right)L_0\Delta  +\Expect\left[\frac{2}{\sqrt{\epsilon}}(F(x_t)-F(x_{t+1})\right] \\
        & \leq   \mathcal{O}(\epsilon)+\Expect\left[\frac{2}{\sqrt{\epsilon}}(F(x_t)-F(x_{t+1})\right].
        \end{aligned}
    \end{equation*}
    Choose $\epsilon$ such that $(\frac{5M_1}{6}+3G_1+4K_1+C_1)\epsilon < 1/2 $ and T such that $T=\mathcal{O}(\epsilon^{-3/2})$ , we can have
    \begin{equation}\label{eq: dr:FOSP derivation}
        \begin{aligned}
            \frac{1}{T}\Expect\left[\sum_{t=0}^{T-1}\|\nabla F(x_{t+1})\|\right] & \leq\mathcal{O}(\epsilon)+\frac{4\Delta_F}{T\sqrt{\epsilon}}  =\mathcal{O}(\epsilon).
        \end{aligned}
    \end{equation}
For the second-order conditions, following the same derivation in the proof of Theorem \ref{thm:full sgd}, we have 
$$
\mathbb{E}[\lambda_{\bar{t}}]\leq \mathcal{O}(\sqrt{\epsilon}),$$
 where $\bar{t}$ is uniformly sampled from $0, \dots, T-1$. By the optimality condition \eqref{eq:optimal condition}, we have
    \begin{equation*}
        \begin{aligned}
            -\Expect[\lambda_{\bar{t}}]I \preceq \Expect[\tilde{H}_{\bar{t}}]
             & \preceq \Expect[\tilde{\nabla}^2 F(x_{\bar{t}+1})]+\Expect[\|\tilde{H}_{\bar{t}}-\tilde{\nabla}^2 F_{\bar{t}+1}\|I]\\
             & \preceq \Expect[\tilde{\nabla}^2 F(x_{\bar{t}+1})]+\Expect[\|\tilde{H}_{\bar{t}}-\tilde{\nabla}^2 F (x_{\bar{t}})\|I]+\Expect[\|\tilde{\nabla}^2 F(x_{\bar{t}})-\tilde{\nabla}^2 F(x_{\bar{t}+1})\|I]     \\
             & \stackrel{(\Diamond)}{\preceq} \Expect[\tilde{\nabla}^2 F(x_{\bar{t}+1})]+ \mathcal{O}(\sqrt{\epsilon})I+\mathcal{O}(\sqrt{\epsilon}) \Expect[\|\nabla F(x_{\bar{t}})\|I] \\
             & \preceq \Expect[\tilde{\nabla}^2 F(x_{\bar{t}+1})]+ \mathcal{O}(\sqrt{\epsilon}),  
        \end{aligned}
    \end{equation*}
    where $\Diamond$ is due to Assumption \ref{assm: hessian}, Lemma \ref{lm var for hessian} and the fact that $\|V_tV_t^\trans\|=1$. So we have
    $$
        \Expect[\lambda_{\min}(\tilde{\nabla}^2 F(x_{\bar{t}+1}))]\geq-\mathcal{O}(\sqrt{\epsilon}).$$
\end{proof}
\newpage

\section{Experiments}\label{sec:expappendix}
We implement the algorithms in PyTorch that enables experiments in neural network training so as to provide a comparison of trust region methods to SGD. The SGD optimizer used in our experiments is provided by the official implementation of PyTorch\footnote{For details, see \url{https://pytorch.org/docs/stable/generated/torch.optim.SGD.html}}.  As we have shown, the normalized SGD is a first-order variant of our trust region framework (\cref{alg:unified}); we use it to represent a first-order trust region method denoted as FOTRGS. To develop a practical second-order variant, we only compute the second-order derivatives in the low-rank subspace, which is enabled by several inquiries of the Hessian-vector product, see Algorithm \ref{alg:drsopo}. While such a method is merely an inexact or low-rank trust region method with limited second-order information, we use it to represent the SOTRGS and demonstrate the benefits of high-order information. 

For each of the datasets, including MNIST, Fashion MNIST and CIFAR10, we fix the category ratios at
$$\{
    0.738,
    0.986,
    0.446,
    0.254,
    0.768,
    0.593,
    0.918,
    0.731,
    0.929,
    0.284
    \}$$ 
so that the datasets become \textbf{imbalanced}, involving $33,260$ out of the original $50,000$ training samples.
For the first-order methods, we uniformly set a momentum with parameter $\beta=0.9$ in all of our experiments, following the standard definition of PyTorch. 
Since SGD typically needs a smaller learning rate, the search intervals are different for normalized SGD
and SGD. 
\textbf{There is no learning rate for the second-order trust region method SOTRGS.}
Note again we use normalized SGD to represent a first-order trust region method denoted as FOTRGS hereafter.

\subsection{Detailed Results of MNIST and Fashion-MNIST} 

For MNIST and Fashion MNIST, we use the mini-batches of size $64$ and a simple convolutional neural network with two convolutional and two fully-connected layers, which has about 16.84 million parameters. We run the optimizers for 25 epochs and repeat each of the combinations -- optimizer $\times$ learning rate $\eta$ -- with 5 runs. 

We put \textbf{the training curves are reported in the main text}.  For completeness, we report the overall average statistics of training loss, training accuracy, and finally test accuracy of all categories.
Results on MNIST and Fashion-MNIST with smoothed $\chi^2$ are listed in Table \ref{tab.tune.m.x2} and \ref{tab.tune.fm.x2}, respectively, and Table \ref{tab.tune.m.cvar} and \ref{tab.tune.fm.cvar} for results with smoothed CVaR. 

\begin{center}
\begin{table}[h]
    \centering
    \begin{subtable}[h]{0.45\textwidth}
    \footnotesize
    \centering
    \begin{tabular}{llrrr}
        \toprule
        Method                  & $\eta$ & train\_loss & train\_acc & test\_acc \\
        \midrule
        SOTRGS                  & -      & 0.000            & 100.000         & 99.036         \\
        \midrule
        \multirow{3}{*}{FOTRGS} & 0.010  & 0.000            & 100.000         & 99.152         \\
                                & 0.050  & 0.000            & 100.000         & 99.352         \\
                                & 0.100  & 0.005            & 99.994          & 99.174         \\
        \midrule
        \multirow{2}{*}{SGD}    & 0.005  & 0.001            & 100.000         & 99.196         \\
                                & 0.010  & 0.003            & 99.996          & 99.112         \\
        \bottomrule
    \end{tabular}
    \caption{Imbalanced MNIST}\label{tab.tune.m.x2}
    \end{subtable}
    \begin{subtable}[h]{0.45\textwidth}
    \footnotesize
    \centering
    \begin{tabular}{llrrr}
        \toprule
        Method                  & $\eta$ & train\_loss & train\_acc & test\_acc \\
        \midrule
        SOTRGS                  & -      & 0.001            & 100.000         & 91.330         \\
        \midrule
        \multirow{3}{*}{FOTRGS} & 0.010  & 0.013            & 100.000         & 90.965         \\
                                & 0.050  & 0.000            & 100.000         & 91.602         \\
                                & 0.100  & 0.009            & 99.980          & 90.993         \\
        \midrule
        \multirow{2}{*}{SGD}    & 0.005  & 0.012            & 99.997          & 91.177         \\
                                & 0.010  & 0.025            & 99.927          & 90.955         \\
        \bottomrule
    \end{tabular}
    \caption{Imbalanced Fashion-MNIST}\label{tab.tune.fm.x2}
    \end{subtable}
    \caption{Tuning results of imbalanced datasets with smoothed $\chi^2$ loss. \\The results are averaged over 5 runs}
\end{table}

\begin{table}[h]
    \footnotesize
    \centering
    \begin{subtable}[h]{0.45\textwidth}
    \begin{tabular}{llrrr}
    \toprule
    Method                  & $\eta$ & train\_loss & train\_acc & test\_acc \\
    \midrule
    SOTRGS                  & -    &           0.000 &        100.000 &        99.020 \\
    \midrule
    \multirow{3}{*}{FOTRGS} & 0.0010 &           0.028 &         99.725 &        98.610 \\
                            & 0.0050 &           0.000 &        100.000 &        98.975 \\
                            & 0.0010 &           0.000 &        100.000 &        98.971 \\
    \midrule
    \multirow{2}{*}{SGD} & 0.0001 &           0.027 &         99.747 &        98.615 \\
        & 0.0005 &           0.004 &        100.000 &        98.885 \\
    \bottomrule
    \end{tabular}
    \caption{Imbalanced MNIST}\label{tab.tune.m.cvar}
    \end{subtable}
    \begin{subtable}[h]{0.45\textwidth}
    \footnotesize
    \centering
    \begin{tabular}{llrrr}
    \toprule
      Method                  & $\eta$ & train\_loss & train\_acc & test\_acc \\
    \midrule
    SOTRGS & - &           0.003 &        100.000 &        90.880 \\
    \midrule
    \multirow{3}{*}{FOTRGS} & 0.0010 &           0.347 &         92.590 &        88.760 \\
        & 0.0050 &           0.064 &         99.701 &        90.590 \\
        & 0.0100 &           0.005 &        100.000 &        91.150 \\
    \midrule
    \multirow{2}{*}{SGD} & 0.0001 &           0.210 &         95.629 &        89.610 \\
        & 0.0005 &           0.027 &        100.000 &        91.150 \\
    \bottomrule
    \end{tabular}
    \caption{Imbalanced Fashion-MNIST}\label{tab.tune.fm.cvar}
    \end{subtable}
    \caption{Tuning results of imbalanced datasets with smoothed CVaR loss. \\The results are averaged over 5 runs}
\end{table}
\end{center}

We note again that the datasets used in the tests are imbalanced and thus incomplete: each category consists of fewer samples.
\subsection{Detailed Results of CIFAR10}
The original version of CIFAR10 contains 50,000 training images and 10,000 validation images of size 32×32 with 10. For this dataset, we apply the same \textbf{imbalanced} sampling as the previous section that similarly gives $33,260$ samples. We train a ResNet-18 model with mini-batches of size $128$ and run the optimizers for 200 epochs. For CIFAR10,  we decrease the learning rates of SGD and Normalized-SGD $\eta \leftarrow 0.1\cdot \eta$ at 80, 160 epochs. 
We provide training results of CIFAR10 with a ResNet18 model. We include tuning curves in Figure \ref{fig.train.cifarov} after some simple verification. Again SGD needs smaller learning rates so that we choose from $\{0.01,0.001\}$.

\begin{figure}[htb!]
    \centering
    \begin{minipage}{0.49\linewidth}
    \centering
    \includegraphics[width=0.99\linewidth]{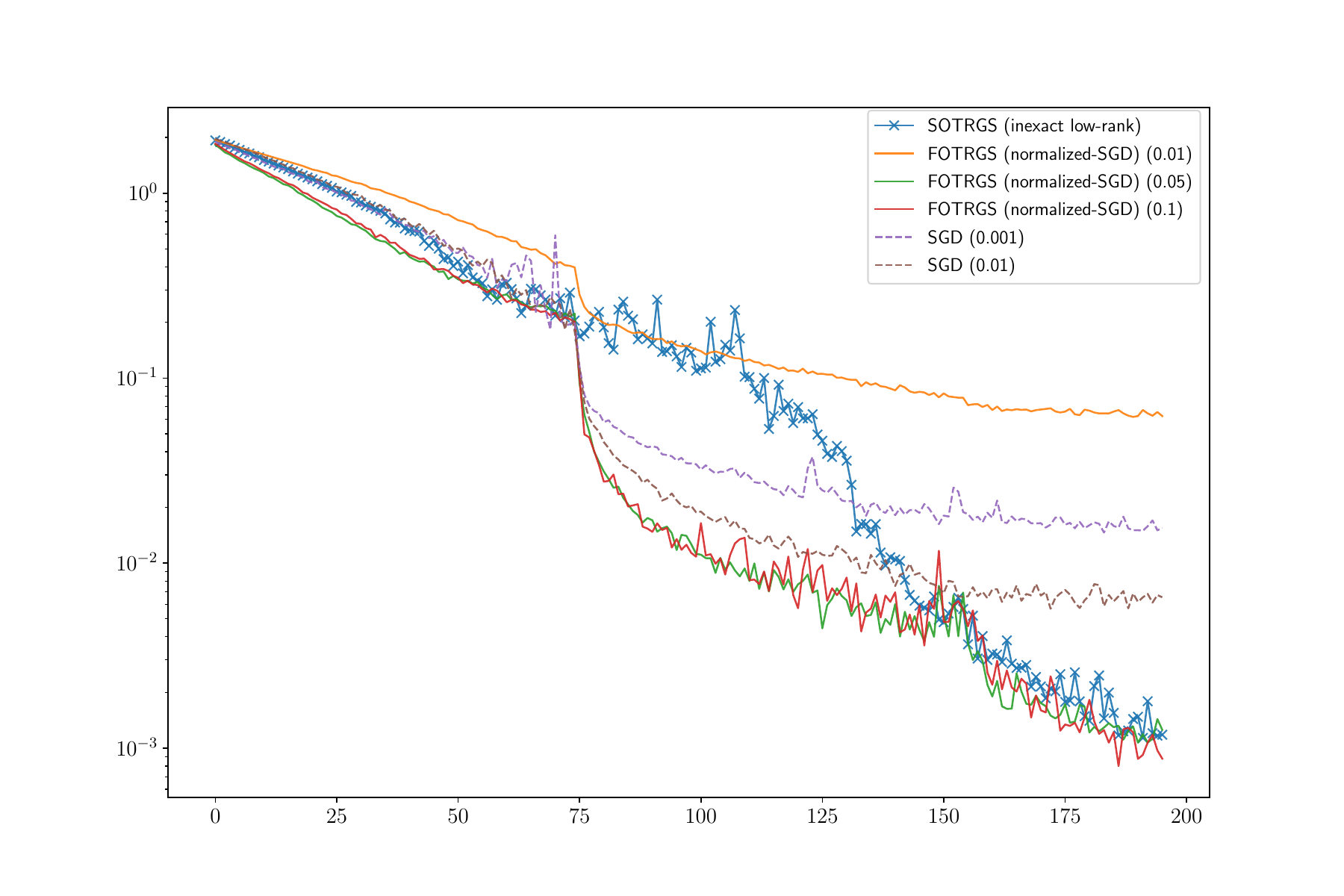}
    \subcaption{Imbalanced CIFAR10 with $\chi^2$ loss}\label{fig.x2.cifar}
    \end{minipage}
    \begin{minipage}{0.49\linewidth}
        \centering
        \includegraphics[width=0.99\linewidth]{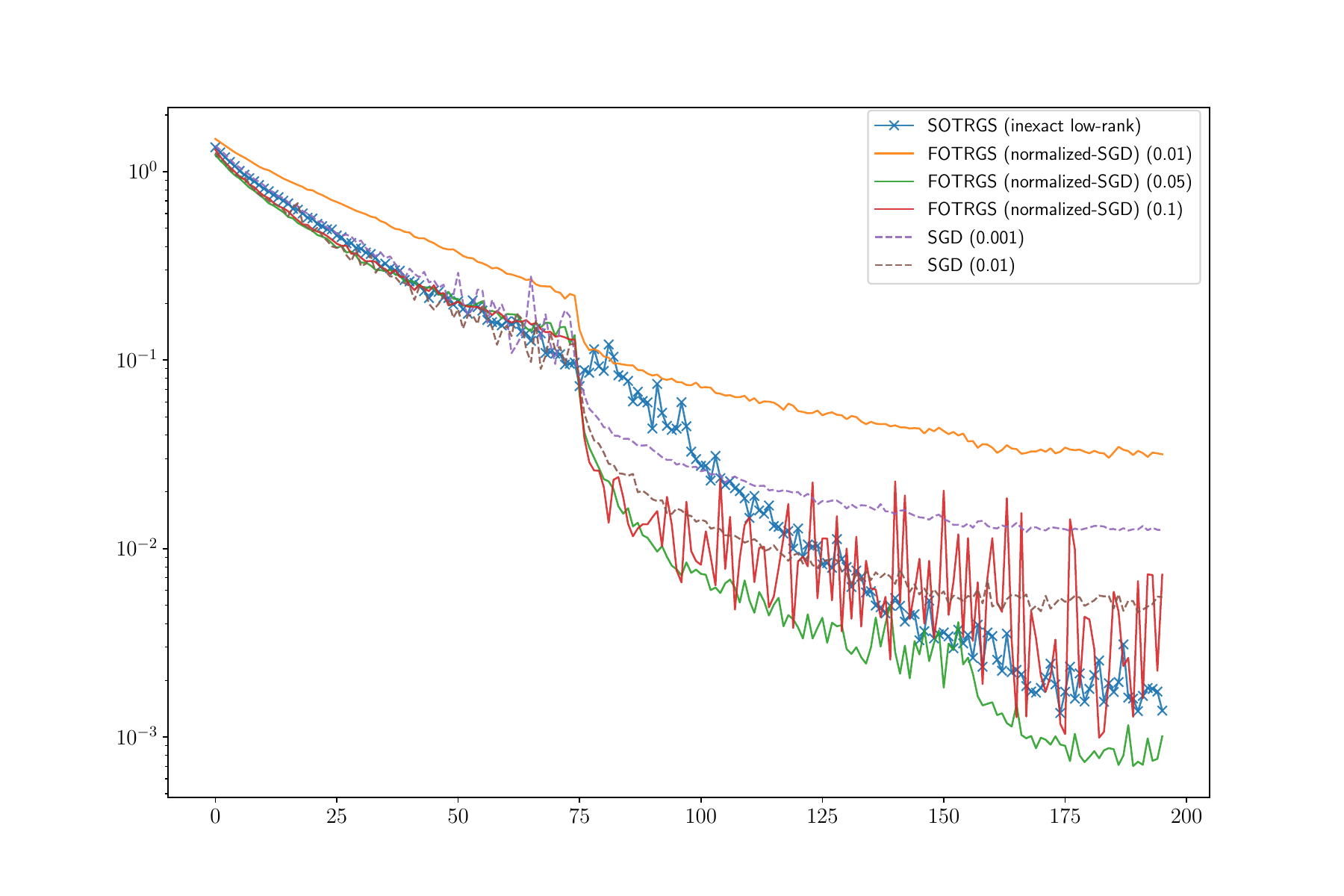}
    \subcaption{Imbalanced CIFAR10 with CVaR loss}\label{fig.cvar.cifar}
    \end{minipage}
    \caption{Training curves of imbalanced CIFAR10.}\label{fig.train.cifarov}
\end{figure}

For testing results, we provide an overall accuracy in Table \ref{tab.tuneall.cifarov}.
\begin{table}[h]
    \begin{subtable}[h]{0.48\textwidth}
    \centering
    \begin{tabular}{llrrr}
        \toprule
        Method                  & $\eta$ & test\_acc \\
        \midrule
        SOTRGS                  & -      & 88.91          \\
        \midrule
        \multirow{3}{*}{FOTRGS} & 0.010  & 83.01             \\
                                & 0.050  & 89.78          \\
                                & 0.100  & 89.73          \\
        \midrule
        \multirow{2}{*}{SGD}    & 0.001  & 88.66           \\
                                & 0.010  & 89.46          \\
        \bottomrule
    \end{tabular}
    \caption{Imbalanced CIFAR10 with smoothed $\chi^2$ loss.}\label{tab.tuneall.cifarov.x2}
    \end{subtable}
    \begin{subtable}[h]{0.48\textwidth}
    \centering
    \begin{tabular}{llrrr}
        \toprule
        Method                  & $\eta$ & test\_acc \\
        \midrule
        SOTRGS                  & -      & 89.56          \\
        \midrule
        \multirow{3}{*}{FOTRGS} & 0.010  & 85.23             \\
                                & 0.050  & 89.88          \\
                                & 0.100  & 88.22          \\
        \midrule
        \multirow{2}{*}{SGD}    & 0.001  & 87.41           \\
                                & 0.010  & 88.79          \\
        \bottomrule
    \end{tabular}
    
    \caption{Imbalanced CIFAR10 with smoothed CVaR loss.}\label{tab.tuneall.cifarov.cvar}
    \end{subtable}
    \caption{Overall test results of imbalanced CIFAR10.}\label{tab.tuneall.cifarov}
\end{table}

More importantly, since we work on imbalanced datasets that inherit disparity amongst categories, we also put emphasis on the test accuracy of the category that takes the least group portion. We present these results in Table \ref{tab.tune.cifarov} to complement the main text.
\begin{table}
    \centering
    \begin{subtable}[h]{0.88\textwidth}
    \centering
    \begin{tabular}{lc|c|rrr|rr|rrr}
        \toprule
        \multirow{2}{*}{$Y$} & \multirow{2}{*}{group ratio in training} & {SOTRGS}       & \multicolumn{3}{c|}{FOTRGS} & \multicolumn{2}{c|}{SGD}                                  \\
                             &                                    & -              & 0.005                       & 0.01                     & 0.1   & 0.001          & 0.01  \\
        \midrule
        0                    & 0.738                              & 0.905          & 0.867                       & 0.936                    & 0.934 & 0.913          & 0.935 \\
        1                    & 0.986                              & 0.981          & 0.908                       & 0.973                    & 0.981 & 0.975 & 0.972 \\
        2                    & 0.446                              & 0.832          & 0.804                       & 0.826                    & 0.823 & 0.821 & 0.822 \\
        3                    & \textbf{0.254}                     & \textbf{0.681} & \textbf{0.705}              & 0.635                    & 0.651 & 0.608 & \textbf{0.629} \\
        4                    & 0.768                              & 0.893          & 0.820                       & 0.939                    & 0.930 & 0.927 & 0.932 \\
        5                    & 0.593                              & 0.876          & 0.759                       & 0.898                    & 0.883 & 0.872 & 0.877 \\
        6                    & 0.918                              & 0.951          & 0.885                       & 0.961                    & 0.957 & 0.951 & 0.963 \\
        7                    & 0.731                              & 0.908          & 0.846                       & 0.944                    & 0.939 & 0.935 & 0.942 \\
        8                    & 0.929                              & 0.939          & 0.908                       & 0.957                    & 0.960 & 0.958 & 0.950 \\
        9                    & 0.284                              & 0.891          & 0.921                       & 0.902                    & 0.905 & 0.906 & 0.904 \\
        \bottomrule
    \end{tabular}
    \caption{Tuning results of imbalanced CIFAR10 with $\chi^2$ loss.}\label{tab.tune.cifar.x2}
    \end{subtable}
    
    \vspace{0.5cm}
    
    \begin{subtable}[h]{0.88\textwidth}
    \centering
    \begin{tabular}{lc|c|rrr|rr|rrr}
        \toprule
        \multirow{2}{*}{$Y$} & \multirow{2}{*}{group ratio in training} & {SOTRGS}       & \multicolumn{3}{c|}{FOTRGS} & \multicolumn{2}{c|}{SGD}                                  \\
                             &                                    & -              & 0.005                       & 0.01                     & 0.1   & 0.001          & 0.01  \\
        \midrule
        0                    & 0.738          & 0.901 & 0.941 & 0.898 &  0.917  &  0.932  & 0.937 \\
        1                    & 0.986          &  0.981 & 0.968 & 0.968 &  0.969  &  0.967  & 0.977 \\
        2                    & 0.446          &  0.855 & 0.816 & 0.756 &  0.787  &  0.833  & 0.869 \\
        3                    & \textbf{0.254} &  \textbf{0.616} & \textbf{0.615} & 0.552 &  0.562  &  \textbf{0.607}  & 0.567 \\
        4                    & 0.768          &  0.950 & 0.955 & 0.890 &  0.928  &  0.925  & 0.928 \\
        5                    & 0.593          &  0.881 & 0.893 & 0.817 &  0.871  &  0.887  & 0.877 \\
        6                    & 0.918          &  0.974 & 0.963 & 0.948 &  0.952  &  0.948  & 0.954 \\
        7                    & 0.731          &  0.942 & 0.945 & 0.890 &  0.924  &  0.934  & 0.925 \\
        8                    & 0.929          &  0.959 & 0.962 & 0.946 &  0.955  &  0.950  & 0.955 \\
        9                    & 0.284          &  0.744 & 0.929 & 0.858 &  0.872  &  0.896  & 0.907 \\
        \bottomrule
    \end{tabular}
    \caption{Tuning results of imbalanced CIFAR10 with CVaR loss.}\label{tab.tune.cifar.cvar}
    \end{subtable}
    \caption{Per-category test results of imbalanced CIFAR10. The {bold text} either recognizes the least group or the best test accuracy among the trials of each method.}\label{tab.tune.cifarov}
\end{table}

\clearpage

\end{document}